\documentclass[a4paper]{amsart}

\newcommand{\cE}{\mathcal{E}}

\newcommand{\cO}{\mathcal{O}}

\newcommand{\cS}{\mathcal{S}}
\newcommand{\cU}{\mathcal{U}}\newcommand{\cV}{\mathcal{V}}
\newcommand{\cX}{\mathcal{X}}

\newcommand{\bR}{\mathbb{R}}

\newcommand{\bZ}{\mathbb{Z}}

\usepackage{lmodern}
\usepackage[dvipsnames,svgnames,x11names,hyperref]{xcolor}
\usepackage{mathtools,url,graphicx,verbatim,amssymb,enumerate,stmaryrd,nicefrac}
\usepackage[pagebackref,colorlinks,citecolor=Mahogany,linkcolor=Mahogany,urlcolor=Mahogany,filecolor=Mahogany]{hyperref}
\usepackage{microtype}
\usepackage[margin=1.5in]{geometry}
\usepackage{tikz, tikz-cd}
\usetikzlibrary{matrix,calc,arrows,patterns,decorations.pathreplacing}
\usepackage{caption}
\usepackage{mathrsfs}
\usepackage[capitalise,noabbrev]{cleveref}

\newtheorem{theorem}{Theorem}[section]
\newtheorem{lemma}[theorem]{Lemma}
\newtheorem{proposition}[theorem]{Proposition} 
\newtheorem{corollary}[theorem]{Corollary}

\newtheorem*{corollary*}{Corollary}

\newtheorem{atheorem}{Theorem}

\theoremstyle{definition}
\newtheorem{definition}[theorem]{Definition}

\newtheorem{convention}[theorem]{Convention}
\theoremstyle{remark}
\newtheorem{example}[theorem]{Example} 
\newtheorem{remark}[theorem]{Remark}

\newcommand{\cat}[1]{\mathsf{#1}}
\newcommand{\mr}[1]{{\rm #1}}

\newcommand{\lra}{\longrightarrow}

\newcommand{\inj}{\hookrightarrow}

\newcommand{\rel}{\,\mr{rel}\,}

\AtBeginDocument{%
	\def\MR#1{}
}

\title{The homotopy type of the topological cobordism category}

\author{Mauricio Gomez Lopez}
\address{Department of Mathematics \\
	Lafayette College, Pardee Hall\\
	Easton, PA, 18042 \\USA}
\email{gomezlom@lafayette.edu}

\author{Alexander Kupers}
\address{Department of Computer and Mathematical Sciences, University of Toronto Scarborough, 1265 Military Trail, Toronto, ON M1C 1A4, Canada}
\email{a.kupers@utoronto.ca}

\date{\today}

\begin{document}

\begin{abstract}We define a cobordism category of topological manifolds and prove that if $d \neq 4$ its classifying space is weakly equivalent to $\Omega^{\infty -1} MT\mr{Top}(d)$, where $MT\mr{Top}(d)$ is the Thom spectrum of the inverse of the canonical bundle over $B\mr{Top}(d)$. We also give versions for manifolds with tangential structures and/or boundary. The proof uses smoothing theory and excision in the tangential structure to reduce the statement to the computation of the homotopy type of smooth cobordism categories due to Galatius-Madsen-Tillman-Weiss.\end{abstract}

\maketitle 

\tableofcontents

\section{Introduction} This paper extends to topological manifolds a result of Galatius, Madsen, Tillmann and Weiss \cite{gmtw}. Let $\mr{Top}(d)$ denote the topological group of homeomorphisms of $\bR^d$ fixing the origin, in the compact open topology. Our main statement involves the Thom spectrum  $MT\mr{Top}(d)$ of the inverse of the canonical $\bR^d$-bundle over $B\mr{Top}(d)$, and a topological cobordism category $\smash{\cat{Cob}^\mr{Top}}(d)$. Its objects are given by a simplicial set of closed $(d-1)$-dimensional topological manifolds and its morphisms are given by a simplicial set of compact $d$-dimensional topological cobordisms.

\begin{atheorem}\label{thm.mainwithout}Let $d \neq 4$, then $B\cat{Cob}^\mr{Top}(d) \simeq \Omega^{\infty-1}MT\mr{Top}(d)$.\end{atheorem}

The proof of this theorem extends to cobordism categories of topological manifolds with tangential structure. For us a tangential structure is a numerable fibre bundle $\xi$ with fibre $\bR^{d}$ and structure group $\mr{Top}(d)$, and given this we define a cobordism category $\smash{\cat{Cob}^\mr{Top,\xi}}(d)$ of $d$-dimensional topological manifolds with $\xi$-structure, i.e.\ a map $TM \to \xi$ of $d$-dimensional topological microbundles, as well as the Thom spectrum $MT^\mr{Top}\xi$ for the inverse of $\xi$. By its definition, up to weak equivalence this spectrum only depends on the stable spherical fibration associated to $\xi$.

\begin{atheorem}\label{thm.mainwith} Let $d \neq 4$ and $\xi$ be a tangential structure, then $B\cat{Cob}^\mr{Top,\xi}(d) \simeq \Omega^{\infty-1}MT^\mr{Top}\xi$. \end{atheorem}

Theorem \ref{thm.mainwithout} follows from Theorem \ref{thm.mainwith} by Lemma \ref{lem.universaltangentialtop}. The proof of Theorem \ref{thm.mainwith} can be summarized by the following sequence of weak equivalences (the terms will be defined later):
\begin{align*}B\cat{Cob}^\mr{Top,\xi}(d) &\simeq \Omega^{\infty-1}(\Psi^\mr{Top,\xi}(d)) \\
&\simeq \Omega^{\infty-1}(\Psi^\mr{Top,||\xi_\bullet||}(d)) \\
&\simeq \Omega^{\infty-1}(||\Psi^{\mr{Top},\epsilon^{d} \times B_\bullet}(d)||) && (1)\\
&\simeq \Omega^{\infty-1}(||MT^\mr{Top}(\epsilon^d \times B_\bullet)||) && (2)\\
&\simeq \Omega^{\infty-1}MT^\mr{Top}||\xi_{\bullet}|| \\
&\simeq \Omega^{\infty-1}MT^\mr{Top}\xi.\end{align*}
It is worth pointing out that we do \emph{not} repeat the scanning argument that is used for smooth manifolds in the setting of topological manifolds, but use (1) an excision result for the dependence of spectra of topological manifolds on the tangential structure, and (2) smoothing theory for spectra of topological manifolds. We believe these results are of independent interest. In particular, these methods imply that the classifying space of the framed topological bordism category is equivalent to that of the framed smooth bordism category in dimensions $\neq 4$, as suggested by \cite[Remark 2.4.30]{lurietft}.

\begin{remark}It is only in the smoothing theory of (2) that we use $d \neq 4$. To remove this condition in our argument, one needs a parametrized version of stable smoothing theory for topological 4-manifolds generalizing \cite{lashofshaneson} or \cite[Section 8.6]{freedmanquinn}.\end{remark}

We also generalize to topological manifolds a result of Genauer \cite{genauer} on cobordism categories of manifolds with boundary, in Section \ref{sec.boundary}. These results are used in \cite{raptissteimle} to reprove the Dwyer--Weiss--Williams theorem. The matter of parametrized surgery theory and its applications will be discussed in a sequel \cite{gomezlopezkupers2}.

\begin{remark}Apart from minor technical modifications, the results of this paper should also hold for PL manifolds (in fact, without the condition $d \neq 4$). The main obstacle is that \cite{siebenmannstratified} claims but does not prove respectful versions of isotopy extension and gluing of submersions charts for PL manifolds. However, the step where these are used can be replaced by an iterated delooping argument, as was done by the first author in \cite{mauriciopaper}. In the forthcoming paper  \cite{mauriciopaper2}, the first author will prove the PL analogue of the main theorem of the present paper. More precisely, in  \cite{mauriciopaper2} the first author will show that there is a weak equivalence $B\cat{Cob}^{\mr{PL}}(d) \simeq \Omega^{\infty-1}MT\mr{PL}(d)$, where $\cat{Cob}^{\mr{PL}}(d)$ is the PL cobordism category as defined in \cite{mauriciopaper} and $MT\mr{PL}(d)$ is a PL Madsen--Tillmann spectrum.
\end{remark}

\subsection{Conventions}

\begin{convention}All our manifolds will be second countable, paracompact and Hausdorff.\end{convention}

\begin{convention}Our preferred category of topological spaces is that of compactly generated weakly Hausdorff spaces.\end{convention}

\subsection{Acknowledgements}  The first author would like to thank Oscar Randal-Williams for introducing him to the field of cobordism categories and for stimulating discussions during the author's graduate studies which led to some of the arguments of the present paper. The second author would like to thank S\o ren Galatius, Oscar Randal-Williams, Manuel Krannich, George Raptis, and Wolfgang Steimle for interesting discussions regarding this project. MGL was supported by the Department of Mathematics of the University of Oregon during the development of this project. AK was supported by the Danish National Research Foundation through the Centre for Symmetry and Deformation (DNRF92), by the European Research Council (ERC) under the European Union's Horizon 2020 research and innovation programme (grant agreement No.\ 682922), by NSF grant DMS-1803766, by the Natural Sciences and Engineering Research Council of Canada (NSERC) [funding reference number 512156 and 512250], as well as the Research Competitiveness Fund of the University of Toronto at Scarborough.

\section{Recollection on smooth cobordism categories} Though our interest is the case of topological manifolds, our proofs use comparisons to the smooth case. Hence we recall the results in \cite{gmtw} on the homotopy type of smooth cobordism categories, though we shall follow the streamlined approach of \cite{grwmonoids}. We also prove a new result about these cobordism categories (\cref{sec.smoothexcision}), an analogue of which will play an important role in our discussion of topological cobordism categories. 

\subsection{Spaces of smooth manifolds} \label{sec.smoothmanifolds} 

\subsubsection{Spaces of smooth submanifolds} \label{sec.smoothsubmfds} In \cite[Section 2]{grwmonoids}, a topological space of smooth submanifolds is defined; we recall their definition here. Note their notation is $\Psi_d(-)$ instead of $\Psi^\mr{Diff}_d(-)$: we add the superscript to indicate that we are dealing with smooth submanifolds.

\begin{definition}Let $M$ be a smooth manifold. The underlying set of $\Psi^\mr{Diff}_d(M)$ is given by the set of closed subsets $X \subset M$ that are smooth $d$-dimensional submanifolds (without boundary). This is topologized as in \cite[Section 2.1]{grwmonoids}.\end{definition}

\begin{figure}[t]
	\begin{tikzpicture}
		\begin{scope}
		\draw [dotted] (0,0) rectangle (3,3);
		\draw [thick] (1,0) to[out=100,in=-80] (2,3);
		\draw [thick] (2.3,1.5) circle (.5cm);
		\node at (1.5,0) [below] {$t=0$};
		\end{scope}

		\begin{scope}[xshift=4cm]
		\begin{scope}
		    \clip(0,0) rectangle (3,3);
		    \draw [thick] (2.9,1.4) circle (.8cm);
		 \end{scope}
		\draw [dotted] (0,0) rectangle (3,3);
		\draw [thick] (1,0) to[out=85,in=-90] (2,3);
		\node at (1.5,0) [below] {$t=\nicefrac{1}{2}$};
		\end{scope}
		
		\begin{scope}[xshift=8cm]
		\draw [dotted] (0,0) rectangle (3,3);
		\draw [thick] (1,0) to[out=70,in=-100] (2,3);
		\node at (1.5,0) [below] {$t=1$};
		\end{scope}
	\end{tikzpicture}
	\label{fig:psi1path}
	\caption{A path in $\Psi_1^\mr{Diff}(\bR^2)$. Note that parts of manifolds may ``disappear at infinity.''}
\end{figure}
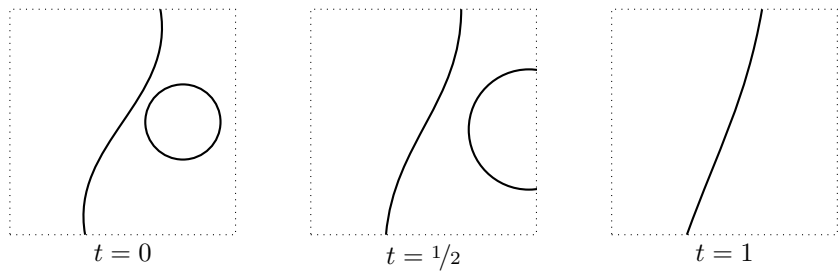

We forego a precise definition of the topology, because this section is a summary of the results in the cited papers (though we will need the details in the proof of Theorem \ref{thm.excisiontangentsmooth}). Heuristically the topology on $\Psi^\mr{Diff}_d(M)$ is one of convergence on compact subsets, and the $X'$ that are close to $X$ near some compact $K \subset M$ are those that can be obtained as the image of a graph of a section of the normal bundle to $X$ under a tubular neighbourhood. 

We next discuss the dependence of $\Psi^\mr{Diff}_d(M)$ on $M$, but first introduce some notation:

\begin{definition} \
	\begin{itemize}
		\item Let $\mr{Top}$ denote the category with class of objects given by compactly generated weakly Hausdorff topological spaces, and set of morphisms from $X$ to $Y$ given by the continuous maps.
		\item Let $\cat{Top}$ denote the $\mr{Top}$-enriched category with class of objects given by compactly generated weakly Hausdorff topological spaces, and topological space of morphisms from $X$ to $Y$ given by the mapping space $\mr{Map}(X,Y)$ in the compact-open topology.
		\item Let $\cat{Mfd}^\mr{Diff}$ denote the $\mr{Top}$-enriched category with class of objects given by smooth manifolds and topological space of morphisms from $P$ to $Q$ given by the space of smooth codimension zero embeddings $\mr{Emb}^\mr{Diff}(P,Q)$ in the (weak) $C^\infty$-topology.
	\end{itemize}
\end{definition}

Observe that $\cat{Mfd}^\mr{Diff}$ is a disjoint union of its subcategories $\cat{Mfd}^\mr{Diff}_n$ of $n$-dimensional manifolds, as there are only codimension zero embeddings between manifolds of the same dimension.

\begin{definition}An \emph{invariant topological sheaf on smooth manifolds} is a functor of $\mr{Top}$-enriched categories $\Phi \colon (\cat{Mfd}^\mr{Diff})^\mr{op} \to \cat{Top}$  such that for any open cover $\{U_i\}_{i \in I}$ of a smooth manifold $M$, the following is an equalizer diagram in $\mr{Top}$:
		\[\begin{tikzcd} \Phi(M) \ar[r] & \prod_i \Phi(U_i) \rar[shift left=.5 ex] \rar[shift left=-.5 ex]& \prod_{i,j} \Phi(U_i \cap U_j) .\end{tikzcd}\]\end{definition}

In \cite[Section 2.2]{grwmonoids} and \cite[Section 3]{rwembedded}, it is proven that $M \mapsto \Psi^\mr{Diff}_d(M)$ is an invariant topological sheaf on smooth manifolds.

\subsubsection{Spaces of smooth submanifolds with tangential structure} \label{sec.smoothtangentialstruc}
The definitions of \cref{sec.smoothsubmfds} may be generalized to include tangential structures. 

\begin{definition}A $d$-dimensional \emph{tangential structure} is a numerable $d$-dimensional vector bundle $\upsilon$.\end{definition}

\begin{remark}
This definition is equivalent to the standard definition as a map $B \to B\mr{O}(d)$. See \cref{sec.asidebundles} for more on this perspective.
\end{remark} 

For us, a map of vector bundles is an isomorphism on fibres and thus there can only be maps between vector bundles of the same dimension. The objects of the cobordism category are of a different dimension than its morphisms, and thus for it is helpful to allow the tangential structure to be of higher dimension than the manifolds. However, as we explain in Section \ref{sec.asidebundles}, for any $d'$-dimensional tangential structure $\upsilon$ with $d' \geq d$, there exists a $d$-dimensional tangential structure $\overline{\upsilon}$ such that for any $d$-dimensional manifold $X$ the space of maps $TX \oplus \epsilon^{d'-d} \to \upsilon$ of $d'$-dimensional vector bundles is weakly equivalent to the space of maps $TX \to \overline{\upsilon}$ of $d$-dimensional vector bundles. Here $\epsilon$ denotes the trivial 1-dimensional vector bundle and $\epsilon^k$ its $k$-fold direct sum.

\begin{definition}Let $M$ be a smooth manifold and $\upsilon$ be a $d'$-dimensional tangential structure with $d' \geq d$. The underlying set of $\Psi^{\mr{Diff},\upsilon}_d(M)$ is given by the set of pairs $\overline{X} = (X,\varphi_X)$ of an element $X \in \Psi^\mr{Diff}_d(M)$ and a map  $\varphi_X \colon TX \oplus \epsilon^{d'-d} \to \upsilon$ of $d'$-dimensional vector bundles. This is topologized as in \cite[Section 2.3]{grwmonoids}.\end{definition}

The functor $M \mapsto \Psi^{\mr{Diff},\upsilon}_d(M)$ is again an invariant topological sheaf on smooth manifolds, cf.~\cite[Section 2.3]{grwmonoids}.

\subsubsection{Simplicial sets of smooth submanifolds}\label{sec.ssetsmooth} By taking the singular simplicial set, we obtain from the topological space $\Psi^\mr{Diff,\upsilon}_d(M)$ a simplicial set which has the same homotopy type. However, it is a consequence of \cite[Section 2.4]{grwmonoids} that this singular simplicial set is weakly equivalent to the subsimplicial set of smooth simplices:

\begin{definition}\label{def.smoothsimplices} The simplicial set $\mr{SmSing}(\Psi^{\mr{Diff},\upsilon}_d(M))$ of \emph{smooth singular simplices} has $k$-simplices given by the maps $\Delta^k \to \Psi^{\mr{Diff},\upsilon}_d(M)$ satisfying the following conditions: 
	\begin{enumerate}[(i)]
		\item the graph $X \subset \Delta^k \times M$ is a smooth neat submanifold with corners and the map $\pi \colon X \hookrightarrow \Delta^k \times M \to \Delta^k$ is a smooth submersion of relative dimension $d$, 
		\item the tangent maps assemble into a map $\varphi_X \colon T_\pi X \oplus \epsilon^{d'-d} \to \upsilon$ of $d'$-dimensional vector bundles, where $T_\pi X$ is the vertical tangent bundle.
	\end{enumerate}
\end{definition}

This is not Kan; this might be fixed by a careful normalization near the faces of the simplices (as in \cite[Appendix 1]{burgheleaaut}), but we shall not need to do this. Instead, for us it will suffice that the smooth approximation techniques of \cite[Section 2.4]{grwmonoids} imply that any element of the $i$th homotopy group of $\Psi^{\mr{Diff},\upsilon}_d(M)$ may be represented by the following data: a smooth submanifold $X \subset \partial D^{i+1} \times M$ such that $X \to \partial D^{i+1}$ is a smooth submersion of relative dimension $d$, together with a map of $d'$-dimensional vector bundles $T_\pi X \oplus \epsilon^{d'-d} \to \upsilon$, where $T_\pi X$ denotes the vertical tangent bundle.

\subsubsection{The universal tangential structure} \label{sec.smoothuniversal} It is well-known that there exist universal numerable $d$-dimensional bundles; examples include the Grassmannian $\mr{Gr}^\mr{Diff}_d(\bR^{\infty})$ of $d$-dimensional planes in $\bR^\infty$ \cite[\S 5]{milnorstasheff}, or the associated $d$-dimensional vector bundle to the universal numerable principal $\mr{O}(d)$-bundle over the bar construction on $\mr{O}(d)$ \cite[Section 8]{mayclassifying}.

The classifying property of such a universal numerable $d$-dimensional vector bundle $\upsilon^\mr{univ}$, with underlying map $p^\mr{univ} \colon E^{\mr{univ}} \to B^\mr{univ}$, is as follows: for every topological space $B$, homotopy classes of maps $[B,B^\mr{univ}]$ are in bijection with isomorphism classes of numerable $d$-dimensional vector bundles over $B$. The bijection is given in one direction by pullback. In \cite[\S IV.8]{kirbysiebenmann} it is shown that for a topological manifold $B$, it also satisfies the following relative classifying property: for every germ of a numerable $d$-dimensional vector bundle $\upsilon_A$ on some open neighbourhood $\cO(A)$ of a subset $A \subset M$, classified by $f_A \colon \cO(A) \to B^\mr{univ}$, the homotopy classes of extensions of the germ of $f_A$ are in bijection with the isomorphism classes of extensions of the germ of $\upsilon_A$. We shall use this to prove the following result:

\begin{lemma}\label{lem.universaltangential} The natural transformation $\Psi^\mr{Diff,\upsilon^{univ}}_d(-) \to \Psi^\mr{Diff}_d(-)$ which forgets the $\upsilon^\mr{univ}$-structure is a natural weak equivalence.\end{lemma}

\begin{proof}Given a commutative diagram
	\[\begin{tikzcd} \partial D^i \rar \dar & \Psi^\mr{Diff,\upsilon^{univ}}_d(M) \dar \\[-2pt]
	D^i \rar & \Psi^\mr{Diff}_d(M),\end{tikzcd}\]
	we need to provide a lift after a homotopy through commutative diagrams. By \cref{sec.ssetsmooth}, we may assume this commutative diagram is represented by a smooth neat submanifold $X \subset D^i \times M$ such that $X \hookrightarrow D^i \times M \to D^i$ is a smooth submersion of relative dimension $d$, together with an $\upsilon^\mr{univ}$-structure on the vertical tangent bundle over $\partial D^i$. We may assume that this $\upsilon^\mr{univ}$-structure extends over an open collar of $\partial D^i$ in $D^i$. Using the relative classifying property described above we can extend this to the entire vertical tangent bundle.
\end{proof}

\subsubsection{An aside on spaces of bundle maps}\label{sec.asidebundles} Given $\upsilon$ as above, the space $\mr{Bun}(TX\oplus \epsilon^{d'-d},\upsilon)$ of maps $TX \oplus \epsilon^{d'-d} \to \upsilon$ of $d'$-dimensional vector bundles, admits a weakly equivalent construction of a more homotopy-theoretic flavor.

There is a map $f \colon B \to B\mr{O}(d')$, unique up to homotopy, classifying $\upsilon$ in the sense that $f^* \upsilon^\mr{univ} \cong \upsilon$ where $\upsilon^\mr{univ}$ is the universal $d'$-dimensional vector bundle over $B\mr{O}(d')$. A choice of such an isomorphism is a covering of $f$ by a bundle map $\psi$. We can replace $f$ with a homotopy equivalent map $f' \colon B' \to B\mr{O}(d')$ that is a Hurewicz fibration;
\[\begin{tikzcd} B \rar{\simeq} \arrow{rd}[swap]{f} & B' \dar{f'} \\
& B\mr{O}(d').\end{tikzcd}\]
By an argument similar to that in Lemma \ref{lem.vectorweq}, $\mr{Bun}(TX\oplus \epsilon^{d'-d},\upsilon)$ is weakly equivalent to the space of vector bundle maps from $TX \oplus \epsilon^{d'-d}$ to $(f')^* \upsilon^\mr{univ}$, so we may as well assume that $f$ is a Hurewicz fibration.

The space $\mr{Bun}(TX \oplus \epsilon^{d'-d},\upsilon)$ is homeomorphic to the space $\cX(X)$ of commutative diagrams of vector bundle maps (recall $\psi$ is fixed)
\[\begin{tikzcd} & \upsilon \dar{\psi} \\
TX \oplus \epsilon^{d'-d} \rar{\omega} \arrow{ru}{\varphi} & \upsilon^\mr{univ}. \end{tikzcd}\]
The identification is given by sending $\varphi \in \mr{Bun}(TX \oplus \epsilon^{d'-d},\upsilon)$ to the pair $(\varphi,\psi \circ \varphi) \in \cX(X)$. 

Let us now fix a classifying map $\varpi \colon TX \oplus \epsilon^{d'-d} \to \upsilon^\mr{univ}$ with underlying map $g \colon X \to B\mr{O}(d')$. There is a map $r \colon \cX(X) \to \mr{Bun}(TX \oplus \epsilon^{d'-d},\upsilon^\mr{univ})$ sending $(\varphi,\omega)$ to $\omega$. The target of $r$ is weakly contractible by the relative classifying property for the universal bundle. Furthermore, $r$ is a Serre fibration because $f$ was a Serre fibration; lift the underlying map, and using the homotopy covering property to extend the map of vector bundles. We conclude that $\cX(X)$ is weakly equivalent to the subspace $\cX(X,\varpi)$ of $\cX(X)$ of pairs of the form $(\varphi,\varpi)$. 

We claim $\cX(X,\varpi)$ in turn can be identified with the space of lifts $\tilde{g}$ of $g$ along $f$. There is a forgetful map from $\cX(X,\varpi)$ to this space of lifts, sending $(\varphi,\varpi)$ to the underlying map $\tilde{g}$ of $\varphi$. One can uniquely recover the map $\varphi$ by identifying the fibre $(TX \oplus \epsilon^{d'-d})_x$ over $x$ with $\upsilon_{\tilde{g}(x)}$ using the map $(\psi_{\tilde{g}(x)})^{-1} \circ \varpi_{g(x)}$. 

In conclusion, we have explained the well-known weak equivalence
\[\mr{Bun}(TX \oplus \epsilon^{d'-d},\upsilon) \simeq \mr{Lift}(X,B \to B\mr{O}(d')),\]
where the right hand side is the space of lifts along $f$ of a fixed classifying map $g$. This uses the assumption that $f$ is a Serre fibration.

\smallskip

We can replace $B\mr{O}(d')$ by $B\mr{O}(d)$ here by taking a different tangential structure. If we take the (homotopy) pullback of $f'$ along the map $B\mr{O}(d) \to B\mr{O}(d')$ induced by the inclusion, we obtain another Hurewicz fibration $\overline{f} \colon \overline{B} \to B\mr{O}(d)$:
\[\begin{tikzcd}\overline{B} \dar[swap]{\overline{f}} \rar & B' \dar{f'} \\
B\mr{O}(d) \rar & B\mr{O}(d'),\end{tikzcd}\]	
and a $d$-dimensional vector bundle $\overline{\upsilon} \coloneqq \overline{f}^* \upsilon^\mr{univ}$. Then the canonical map $\mr{Bun}(TX,\overline{\upsilon}) \to \mr{Bun}(TX \oplus \epsilon^{d'-d},\upsilon)$ is a weak equivalence.

\subsubsection{Spectra of smooth submanifolds} \label{sec.smoothspectra}
From the spaces $\Psi_d^{\mr{Diff},\upsilon}(-)$ we can produce a spectrum. Our notion of spectrum is quite naive, even though the spectrum is visibly an orthogonal spectrum: for us a \emph{spectrum} $E$ is a sequence $\{E_n\}_{n \geq 0}$ of pointed topological spaces with pointed maps $\Sigma E_n \to E_{n+1}$. We call $E_n$ the \emph{$n$th level}.

\begin{definition}The spectrum $\Psi^{\mr{Diff},\upsilon}(d)$ has $n$th level given by $\Psi^{\mr{Diff},\upsilon}_d(\bR^n)$, with basepoint $\varnothing$. The structure maps $S^1 \wedge \smash{\Psi^{\mr{Diff},\upsilon}_d}(\bR^n) \to \smash{\Psi^{\mr{Diff},\upsilon}_d}(\bR^{n+1})$ are given by identifying $S^1$ with the one-point compactification of $\bR$, and defining it to be
	\[(t,\overline{X}) \longmapsto \begin{cases} \varnothing & \text{if $(t,\overline{X}) = (\infty,\overline{X})$,}\\
	\iota (\overline{X}) + t \cdot e_1 & \text{otherwise,} \end{cases}\]
	 with $\iota \colon \bR^n \to \bR^{n+1}$ the inclusion on the last $n$ coordinates.\end{definition}

\subsection{Smooth cobordism categories} \label{sec.smoothcobordism} 

\subsubsection{The definition of the smooth cobordism category} \label{sec.smoothcobordismdef}
The definition of the smooth cobordism category uses long manifolds, as below:

\begin{definition}For $0 \leq p \leq n$, we let $\psi^{\mr{Diff},\upsilon}(d,n,p)$ be the subspace of $\Psi^{\mr{Diff},\upsilon}_d(\bR^n)$ consisting of those $\overline{X} = (X,\varphi_X)$ such that $X \subset \bR^p \times (0,1)^{n-p}$.\end{definition}

The smooth cobordism category will be a non-unital category internal to the category $\mr{Top}$ of topological spaces, i.e.\ have spaces of objects and morphisms, but no specified identity morphisms.

\begin{definition}\label{def.smoothcob} The \emph{smooth cobordism category} $\cat{Cob}^{\mr{Diff},\upsilon}(d,n)$ is the non-unital category internal to $\mr{Top}$ with
	\begin{itemize}
	\item object space is given by $\psi^{\mr{Diff},\upsilon}(d-1,n-1,0)$, 
	\item morphism space is given by the subspace of $(0,\infty) \times \psi^{\mr{Diff},\upsilon}(d,n,1)$ of those $(t,\overline{X})$ such that there exists an $\epsilon>0$ so that 
	\begin{align*}\qquad &\overline{X} \cap ((-\infty,\epsilon) \times (0,1)^{n-1}) = (-\infty,\epsilon) \times (\overline{X} \cap (\{0\} \times (0,1)^{n-1})) \quad \text{and} \\
	&\overline{X} \cap ((t-\epsilon,\infty) \times (0,1)^{n-1}) = (t-\epsilon,\infty)\times (\overline{X} \cap (\{t\} \times (0,1)^{n-1})),\end{align*}
	\item the source map sends $(t,\overline{X})$ to $\overline{X} \cap (\{0\} \times (0,1)^{n-1})$ and the target map sends it to $\overline{X} \cap (\{t\} \times (0,1)^{n-1})$,
	\item the composition map sends a pair of composable morphisms $(t_0,\overline{X}_0)$ and $(t_1,\overline{X}_1)$ to $(t_0+t_1,\overline{X})$ with $\overline{X}$ defined by
	\[\qquad \overline{X} \coloneqq \begin{cases}\overline{X}_0 \cap ((-\infty,t_0] \times (0,1)^{n-1}) & \text{in $((-\infty,t_0] \times (0,1)^{n-1})$,} \\
	\overline{X}_1 \cap ([0,\infty) \times (0,1)^{n-1})+t_0 \cdot e_1 & \text{in $([t_0,\infty) \times (0,1)^{n-1})$.}\end{cases}\]
\end{itemize}\end{definition}

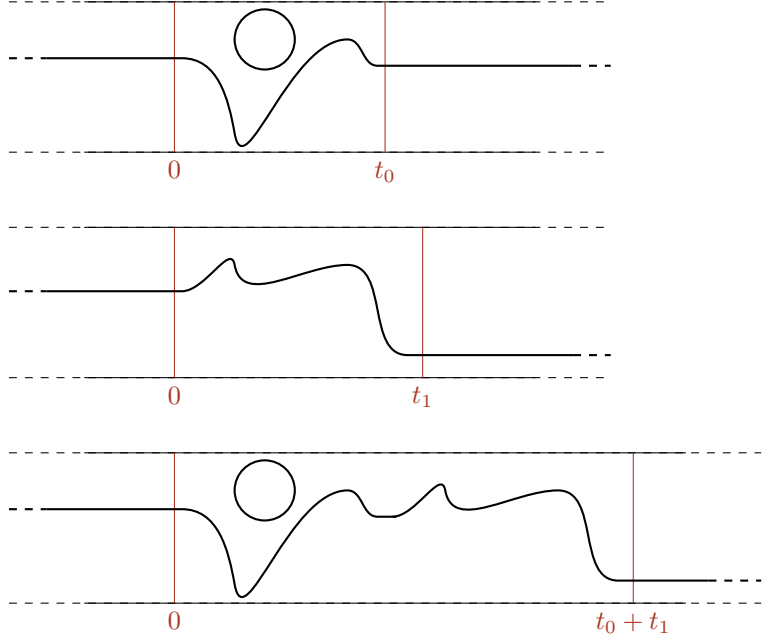
\begin{figure}
	\centering
	\begin{tikzpicture}
	\begin{scope}
	\draw [dashed] (-4,0) -- (4,0);
	\draw [dashed] (-4,2) -- (4,2);	
	\draw (-3,0) -- (3,0);
	\draw (-3,2) -- (3,2);
	\draw [Mahogany] (-1.8,0) -- (-1.8,2);
	\draw [Mahogany] (1,0) -- (1,2);
	
	\node at (-1.8,0) [below,Mahogany] {$0$};
	\node at (1,0) [below,Mahogany] {$t_0$};
	
	
	\draw [thick,dashed] (-4,1.25) -- (-3.5,1.25);
	\draw [thick,dashed] (3.5,1.15) -- (4,1.15);
	\draw [thick] (-3.5,1.25) -- (-1.7,1.25) to[out=0,in=100] (-1,0.25) to[out=-80,in=180] (.5,1.5) to[out=0,in=180] (.9,1.15) -- (1.55,1.15) to[out=0,in=180] (3.5,1.15);
	\draw [thick] (-.6,1.5) ellipse (.4 cm and .4 cm);
	\end{scope}
	
	\begin{scope}[yshift=-3cm]
	\draw [dashed] (-4,0) -- (4,0);
	\draw [dashed] (-4,2) -- (4,2);	
	\draw (-3,0) -- (3,0);
	\draw (-3,2) -- (3,2);
	\draw [Mahogany] (-1.8,0) -- (-1.8,2);
	\draw [Mahogany] (1.5,0) -- (1.5,2);
	
	\node at (-1.8,0) [below,Mahogany] {$0$};
	\node at (1.5,0) [below,Mahogany] {$t_1$};
	
	
	\draw [thick,dashed] (-4,1.15) -- (-3.5,1.15);
	\draw [thick,dashed] (3.5,.3) -- (4,.3);
	\draw [thick] (-3.5,1.15) -- (-1.7,1.15) to[out=0,in=100] (-1,1.5) to[out=-80,in=180] (.5,1.5) to[out=0,in=180] (1.3,.3) -- (3.5,.3);
	\end{scope}
	
	\begin{scope}[yshift=-6cm]
	\draw [dashed] (-4,0) -- (6,0);
	\draw [dashed] (-4,2) -- (6,2);	
	\draw (-3,0) -- (5,0);
	\draw (-3,2) -- (5,2);
	\draw [Mahogany] (-1.8,0) -- (-1.8,2);
	\draw [Mahogany] (4.3,0) -- (4.3,2);

	\node at (-1.8,0) [below,Mahogany] {$0$};
	\node at (4.3,0) [below,Mahogany] {$t_0+t_1$};


	\draw [thick,dashed] (-4,1.25) -- (-3.5,1.25);
	\draw [thick] (-3.5,1.25) -- (-1.7,1.25) to[out=0,in=100] (-1,0.25) to[out=-80,in=180] (.5,1.5) to[out=0,in=180] (.9,1.15) -- (1.15,1.15);
	\draw [thick,xshift=2.8cm] (-1.7,1.15) to[out=0,in=100] (-1,1.5) to[out=-80,in=180] (.5,1.5) to[out=0,in=180] (1.3,.3) -- (2.5,.3);
	\draw [thick,dashed,xshift=2.8cm] (2.5,.3) -- (3.2,.3);
	\draw [thick] (-.6,1.5) ellipse (.4 cm and .4 cm);
	\end{scope}	
	\end{tikzpicture}
	\caption{Two morphisms in $\mr{Cob}^\mr{Diff}(1,2)$ and their composition.}
	\label{fig:xppoint}
\end{figure}

We are most interested in the case where $n$ is arbitrarily large. The inclusion of $\bR^n$ into $\bR^{n+1}$ on the last $n$ coordinates gives a continuous functor $\cat{Cob}^{\mr{Diff},\upsilon}(d,n) \to \cat{Cob}^{\mr{Diff},\upsilon}(d,n+1)$.

\begin{definition}We define $\cat{Cob}^{\mr{Diff},\upsilon}(d)$ to be $\mr{colim}_{n \to \infty}\, \cat{Cob}^{\mr{Diff},\upsilon}(d,n)$.\end{definition}

\subsubsection{The classifying space of the smooth cobordism category} \label{sec.smoothcobclass}
The nerve of this non-unital topological category is a semisimplicial space $N_\bullet \cat{Cob}^{\mr{Diff},\upsilon}(d,n)$, and its thick geometric realisation as in \cite[Section 1.2]{rwebertsemi} will be denoted $B\cat{Cob}^{\mr{Diff},\upsilon}(d,n)$.

By iterated delooping arguments \cite{grwmonoids} or by an application of Gromov's $h$-principle \cite{rwembedded}, there is a zigzag of weak equivalences between $B\cat{Cob}^{\mr{Diff},\upsilon}(d,n)$ and $\Omega^{k-1} \psi^{\mr{Diff},\upsilon}(d,n,k)$ for $1 \leq k \leq n$. In particular, taking $k=n$, we get a zigzag of weak equivalences
\[B\cat{Cob}^{\mr{Diff},\upsilon}(d,n) \simeq \Omega^{n-1} \Psi^{\mr{Diff},\upsilon}_d(\bR^n).\] Tracing through the arguments, one proves that the following diagram commutes
	\[\begin{tikzcd} B\cat{Cob}^{\mr{Diff},\upsilon}(d,n) \dar & \cdots \lar[swap]{\simeq} \rar{\simeq} \dar & \Omega^{n-1} \Psi^{\mr{Diff},\upsilon}_d(\bR^n) \dar\\
	B\cat{Cob}^{\mr{Diff},\upsilon}(d,n+1) & \cdots \lar[swap]{\simeq} \rar{\simeq} & \Omega^{n} \Psi^{\mr{Diff},\upsilon}_d(\bR^{n+1}),\end{tikzcd}\]
with the right vertical map the adjoint of the structure map. Letting $n \to \infty$, one obtains \cite[Theorem 3.13]{grwmonoids}, originally proven in \cite{gmtw}.

\begin{theorem}[Galatius-Madsen-Tillmann-Weiss] \label{thm.cobinfiniteloop} There is a zigzag of weak equivalences
	\[B\cat{Cob}^{\mr{Diff},\upsilon}(d) \simeq \Omega^{\infty-1} \Psi^{\mr{Diff},\upsilon}(d).\]
\end{theorem}
	
\subsection{The scanning argument} \label{sec.smoothscanning} The spectrum $\Psi^{\mr{Diff},\upsilon}(d)$ is weakly equivalent to the \emph{Madsen-Tillmann spectrum} $MT\upsilon$, which in this paper shall be denoted $MT^\mr{Diff} \upsilon$ to distinguish it from the topological version to be defined in \cref{sec:thom-spectrum-top}. It is an example of a Thom spectrum, and we describe its construction in detail. 

Let $\mr{O}(n+d\,\mr{fix}\,d) \subset \mr{O}(n+d)$ denote the subgroup fixing $\bR^d$ pointwise and $\mr{O}(n+d\,\mr{pres}\,d) \subset \mr{O}(n+d)$ denote the subgroup preserving $\bR^d$ setwise, so that $\mr{O}(n+d\,\mr{pres}\,d)$ is the semi-direct product $\mr{O}(d) \ltimes \mr{O}(n+d\,\mr{fix}\,d)$. The quotient $\mr{O}(n+d)/ \mr{O}(n+d\,\mr{pres}\,d)$ is the \emph{Grassmannian  $\mr{Gr}^\mr{Diff}_d(\bR^{n+d})$ of $d$-planes in $\bR^{n+d}$}. As $\mr{O}(n+d\,\mr{fix}\,d) \simeq \mr{O}(n)$ by Gram--Schmidt, we see that $\mr{Gr}^\mr{Diff}_d(\bR^{n+d}) \simeq \mr{O}(n+d)/(\mr{O}(d) \times \mr{O}(n))$. We may identify $\mr{Gr}^\mr{Diff}_d(\bR^{n+d})$ with the subspace of $\Psi_d^\mr{Diff}(\bR^{n+d})$ of those smooth submanifolds that are $d$-dimensional planes through the origin. We can add a $d$-dimensional tangential structure to this:

\begin{definition}The \emph{Grassmannian $\mr{Gr}^{\mr{Diff},\upsilon}_d(\bR^{n+d})$ of $d$-planes in $\bR^{n+d}$ with $\upsilon$-structure} is the topological space of pairs $(P,\varphi_P)$ of a $d$-dimensional linear plane $P \in \mr{Gr}^\mr{Diff}_d(\bR^{n+d})$ and a map of $d$-dimensional vector bundles $\varphi_P \colon TP \to \upsilon$.\end{definition}

The Grassmannian $\mr{Gr}^\mr{Diff}_d(\bR^{n+d})$ carries an $n$-dimensional vector bundle $\gamma^\perp_{d,n+d}$ given by those vectors $v \in \smash{\bR^{n+d}}$ such that $v \perp P$. We pull it back along the map 
\[\pi_\upsilon \colon \mr{Gr}^{\mr{Diff},\upsilon}_d(\bR^{n+d}) \lra \mr{Gr}^\mr{Diff}_d(\bR^{n+d})\]
which forgets $\varphi_P$ and take its \emph{Thom space} $\mr{Thom}(\pi_\upsilon^* \gamma^\perp_{d,n+d})$ to obtain the $n$th level of $\smash{MT^\mr{Diff}\upsilon}$. Explicitly, this Thom space is obtained by taking the unit sphere bundle $S(\gamma^\perp_{d,n+d})$ and unit disc bundle $D(\gamma^\perp_{d,n+d})$ with respect to the standard Riemannian metric, pulling them back along $\pi_{\upsilon}$ and taking the pushout
\[\begin{tikzcd}\pi^*_\upsilon S(\gamma^\perp_{d,n+d}) \rar \dar & \ast \dar \\
\pi^*_\upsilon D(\gamma^\perp_{d,n+d})  \rar & \mr{Thom}(\pi_\upsilon^* \gamma^\perp_{d,n+d}).\end{tikzcd}\]
The left vertical map can be presented as an NDR-pair and hence the Thom space is obtained as a pushout along a cofibration, so is also a homotopy pushout.

To understand the structure maps, we use that the inclusion $\bR^{n+d} \hookrightarrow \bR^{1+n+d}$ on the last $n+d$ coordinates induces a map $\iota \colon \mr{Gr}^{\mr{Diff},\upsilon}_d(\bR^{n+d}) \to \mr{Gr}^{\mr{Diff},\upsilon}_d(\bR^{1+n+d})$ along which $\gamma^\perp_{d,1+n+d}$ gets pulled back to the Whitney sum $\epsilon \oplus \gamma^\perp_{d,n+d}$. To produce the structure map, we then use the homeomorphism $\mr{Thom}(\epsilon \oplus \zeta) \cong S^1 \wedge \mr{Thom}(\zeta)$. 

\begin{definition}The \emph{Madsen-Tillmann spectrum $MT^\mr{Diff} \upsilon$} has $(n+d)$th level given by the Thom space $\mr{Thom}(\pi_\upsilon^* \gamma^\perp_{d,n+d})$. The structure maps are given by
\[S^1 \wedge \mr{Thom}(\pi_\upsilon^* \gamma^\perp_{d,n+d}) \cong \mr{Thom}(\pi_\upsilon^* \iota^* \gamma^\perp_{d,1+n+d}) \lra \mr{Thom}(\pi_\upsilon^* \gamma^\perp_{d,1+n+d}).\]\end{definition}

There is a map of spectra $MT^\mr{Diff} \upsilon \to \Psi^{\mr{Diff},\upsilon}(d)$ given on $(n+d)$th levels by the map
\[\mr{Thom}(\pi_\upsilon^* \gamma^\perp_{d,n+d}) \lra \Psi^{\mr{Diff},\upsilon}_d(\bR^n),\]
obtained by considering a $d$-dimensional plane $P$ as a $d$-dimensional closed smooth manifold, and using the map $\varphi_P \colon TP \to \upsilon$ to endow this with an $\upsilon$-structure. The following is \cite[Theorem 3.22]{grwmonoids}, and using Theorem \ref{thm.cobinfiniteloop} completes the computation of the homotopy type of the cobordism category.

\begin{theorem} \label{thm.gmtw} The map of spectra $MT^\mr{Diff} \upsilon \to \Psi^{\mr{Diff},\upsilon}(d)$ is a weak equivalence.
\end{theorem}

\begin{remark}The spectrum $MT^\mr{Diff} \upsilon$ is $(-d)$-connective; the first non-trivial reduced homology class of the $(n+d)$th level is in degree $n = -d+(n+d)$. However, the homotopy type of the classifying space of the cobordism category depends only on the connective cover. An $n$-categorical variant of the cobordism category depends on the entire spectrum \cite{bokstedtmadsen}. We believe their results can be adapted to PL and topological manifolds.\end{remark}

\subsection{Excision in the tangential structure} \label{sec.smoothexcision} We now discuss the dependence of $\Psi^{\mr{Diff},\upsilon}(d)$ on the tangential structure $\upsilon$.

\subsubsection{Semisimplicial vector bundles}
We need some technical results about resolving vector bundles. A \emph{semisimplicial $d$-dimensional vector bundle} is a functor from $\Delta^\mr{op}_\mr{inj}$, the opposite of the category $\Delta_\mr{inj}$ of finite non-empty ordered sets and order-preserving injections, to the category of spaces with $d$-dimensional vector bundles. By passing to the skeleton of $\Delta^\mr{op}_\mr{inj}$, this is the same data as a $d$-dimensional vector bundle $\upsilon_p$ for each $[p] = (0<1<\ldots<p)$, and a map of vector bundles $\upsilon_{p'} \to \upsilon_p$ for each order-preserving injection $[p] \to [p']$.

Given such a semisimplicial vector bundle, we can take the thick geometric realisations of its total and base spaces, and get a map $p^{||\upsilon_{\bullet}||} \colon E^{||\upsilon_\bullet||} \to B^{||\upsilon_\bullet||}$ with fibres that can be identified with $\bR^d$. The following lemma tells us when it is a (numerable) vector bundle, i.e.\  a (numerable) locally trivial bundle with fibres $\bR^d$ and structure group $\mr{GL}_d(\bR)$.

\begin{lemma}\label{lem.realisationvectorbundle} If $\upsilon_{\bullet}$ is a levelwise numerable semisimplicial $d$-dimensional vector bundle, then $||\upsilon_{\bullet}||$ is a $d$-dimensional vector bundle. It is numerable if it is levelwise trivial.
\end{lemma}

\begin{proof}We shall use the following general fact: if $G$ is a well-pointed topological group and $\xi_{\bullet}$ is a semisimplicial principal $G$-bundle such that each $p^{\xi_p} \colon E^{\xi_p} \to B^{\xi_p}$ is numerable, then $||\xi_{\bullet}||$ is a principal $G$-bundle. It is numerable if each $p^{\xi_p}$ is trivial. This is proven for the bar construction in \cite[Theorem 8.2]{mayclassifying} and in general in \cite[Theorem 2]{robertsstevenson}. The latter theorem is stated in larger generality: to obtain the statement above, we take $B = \ast$, take the simplicial group to be constant in the simplicial direction, and use that the thick geometric realisation of a semisimplicial space equals the geometric realisation of the proper simplicial space obtained by freely adjoining degeneracies (see the proof of \cite[Lemma 1.8]{rwebertsemi}).
		
To deduce the statement of the lemma from this, we use the natural bijection between isomorphism classes of (numerable) $d$-dimensional vector bundles and (numerable) principal $\mr{GL}_d(\bR)$-bundles, given in one direction by taking the associated bundles and in the other direction by taking the frame bundle. Note that taking associated bundles commutes with geometric realisation, because both taking a product with a space (see \cite[Remark 2.23]{gepnerhenriques}) and taking the quotient with respect to a group action do (as colimits commute with colimits).
\end{proof}

All numerable vector bundles are up to homotopy equivalence obtained by thick geometric realisation of a levelwise trivial semisimplicial vector bundle:

\begin{lemma}\label{lem.trivialvectorbundle} For every numerable $d$-dimensional vector bundle $\upsilon$, there is a levelwise trivial semisimplicial $d$-dimensional vector bundle $\upsilon_\bullet$ such that $\upsilon$ is homotopy equivalent to $||\upsilon_{\bullet}||$ as a numerable $d$-dimensional vector bundle.\end{lemma}

\begin{proof}Pick an open cover $\cU$ of $B$ such that for each $U_i \in \cU$ the bundle $\upsilon|_{U_i}$ is trivialisable, and which admits a subordinate partition of unity. This exists by definition of a numerable $\bR^d$-bundle. 	
Consider the semisimplicial space $\cU_\bullet$ given by
	\[[k] \longmapsto \bigsqcup_{(i_0,\ldots,i_k)} (U_{i_0} \cap \ldots \cap U_{i_k}),\]
	where the right-hand side is indexed by the ordered $(k+1)$-tuples of indices of $\cU$. This has an augmentation to $B$, and \cite[\S 4]{Se2} describes how to use the partition of unity to obtain a homotopy inverse $B \to ||\cU_\bullet||$ (which happens to be a section).
	
	By our choice of $\cU$, we have a trivialisation of the bundle $\upsilon$ over each $U_i$, and hence over $U_\sigma$ for each $\sigma = (i_0,\ldots,i_k)$ using the trivialisation of $U_{i_0}$. Now write down a semisimplicial $\bR^d$-bundle $\upsilon_\bullet$ as follows: it is a levelwise trivial bundle \[[k] \longmapsto \bigsqcup_{(i_0,\ldots,i_k)} (U_{i_0} \cap \ldots \cap U_{i_k}) \times \bR^d,\]
	whose face maps $d_i$ for $i>0$ are given by the inclusion, while the face map $d_0$ uses the transition function $U_{i_0} \cap U_{i_1} \to\mr{GL}_d(\bR)$. This semisimplicial space is augmented over the total space of $\upsilon$, yielding upon thick geometric realisation a map $||\upsilon_\bullet|| \to \upsilon$ which exhibits the left term as the pullback of $\upsilon$ along the map $||\cU_\bullet|| \to B$.
\end{proof}

\subsubsection{Excision for Thom spectra} The construction $\upsilon \mapsto MT^\mr{Diff}\upsilon$ gives a functor from $d$-dimensional vector bundles to spectra, because both the construction of $\mr{Gr}^{\mr{Diff},\upsilon}_d(\bR^{n+d})$ and the Thom space construction are natural in the $d$-dimensional vector bundle $\upsilon$. It has a number of properties, the first of which involves weak equivalences of vector bundles; these are maps of vector bundles whose underlying map of spaces is a weak equivalence.

\begin{lemma}\label{lem.smooththomspectraweq}The construction $\upsilon \mapsto MT^\mr{Diff}\upsilon$ takes weak equivalences to weak equivalences.\end{lemma}

\begin{proof}Let $\upsilon \to \upsilon'$ be a weak equivalence, with underlying map $B \to B'$ of base spaces. It suffices to prove this induces a levelwise pointed weak equivalence. The Thom space was the homotopy pushout of the diagram 
	\[\pi^*_\upsilon D(\gamma^\perp_{d,n+d}) \longleftarrow \pi^*_\upsilon S(\gamma^\perp_{d,n+d}) \lra \ast.\]
It thus suffices to show that the maps $\pi^*_\upsilon D(\gamma^\perp_{d,n+d}) \to \pi^*_{\upsilon'} D(\gamma^\perp_{d,n+d})$ and $\pi^*_\upsilon S(\gamma^\perp_{d,n+d}) \to \pi^*_{\upsilon'} S(\gamma^\perp_{d,n+d})$ are weak equivalences. The first map is equivalent to $B \to B'$ which is a weak equivalence by hypothesis. The second map fits into a map of fibre sequences
\[\begin{tikzcd} S^{n-1} \dar{\simeq} \rar & \pi^*_\upsilon S(\gamma^\perp_{d,n+d}) \rar \dar & B \dar{\simeq} \\
S^{n-1} \rar & \pi^*_{\upsilon'} S(\gamma^\perp_{d,n+d}) \rar & B'.\end{tikzcd} \]
Since the maps on fibres and base spaces are weak equivalences, so is that on total spaces.
\end{proof}

Given a $d$-dimensional semisimplicial vector bundle $\upsilon_\bullet$ which is levelwise numerable, by Lemma \ref{lem.realisationvectorbundle} we can take its thick geometric realisation to obtain a vector bundle, and take its Thom space $\mr{Thom}(\pi_{||\upsilon_{\bullet}||}^* \gamma^\perp_{d,n+d})$. The map $\Delta^k \wedge \upsilon_k \to ||\upsilon_{\bullet}||$ induces compatible pointed maps
\[\Delta^k_+ \wedge \mr{Thom}(\pi_{\upsilon_k}^* \gamma^\perp_{d,n+d}) \cong \mr{Thom}(\pi_{\Delta^k \times \upsilon_k}^* \gamma^\perp_{d,n+d}) \lra \mr{Thom}(\pi_{||\upsilon_{\bullet}||}^* \gamma^\perp_{d,n+d})\]
and by the universal property of colimits these induce a map 
\[||\mr{Thom}(\pi_{\upsilon_\bullet}^* \gamma^\perp_{d,n+d})|| \lra \mr{Thom}(\pi_{||\upsilon_{\bullet}||}^* \gamma^\perp_{d,n+d}),\]
where the first thick geometric realisation is in the category of pointed spaces. These assemble to a map of spectra 
\begin{equation}\label{eqn:mt-comp}||MT^\mr{Diff}\upsilon_{\bullet}|| \lra MT^\mr{Diff}||\upsilon_{\bullet}||.
\end{equation}

A major difference between these two spectra is that the left term involves manifolds where the underlying map of the tangential structure lands in some $B_p$, while for the right term it may hit much of $||B_\bullet||$. Thus to prove that \eqref{eqn:mt-comp} is a weak equivalence we introduce the following variation:

\begin{definition}The subspace $\mr{Gr}^\mr{Diff,\upsilon,const}_d(\bR^{n+d})$ of $\mr{Gr}^\mr{Diff,\upsilon}_d(\bR^{n+d})$ consists of those $(P,\varphi_P)$ such that the map $\phi_P \colon P \to B$ underlying $\varphi_P \colon TP \to \nu$ factors over a point.\end{definition}

\begin{lemma}\label{lem.thomconstsmooth} The inclusion $\mr{Gr}^\mr{Diff,\upsilon,const}_d(\bR^{n+d}) \hookrightarrow \mr{Gr}^\mr{Diff,\upsilon}_d(\bR^{n+d})$ is a weak equivalence.\end{lemma}

\begin{proof}Suppose we are given a commutative diagram
	\[\begin{tikzcd} \partial D^i \rar \dar & {\mr{Gr}^\mr{Diff,\upsilon,const}_d(\bR^{n+d})} \dar \\[-2pt]
	D^i \rar &  {\mr{Gr}^\mr{Diff,\upsilon}_d(\bR^{n+d})}.\end{tikzcd}\]
	Then we need to provide a lift after a homotopy through commutative diagrams. By smooth approximation as in \cref{sec.ssetsmooth}, we may assume this is represented a smooth neat submanifold $X \subset D^i \times \bR^{n+d}$ such that each of the fibres of $\pi \colon X \hookrightarrow D^i \times \bR^{n+d} \to D^i$ is $d$-dimensional planes through the origin, together with a $\nu$-structure on its vertical tangent bundle. The underlying map of this $\nu$-structure factors through a point on the fibre over every $b \in \partial D^i$; we say that it is \emph{of the desired form}. Our goal is to homotope $\varphi_X$ such that it is of the desired form for all $b \in D^i$.
	
	Since the fibres are planes through the origin, there is an evident deformation retraction $H$ of $X$ onto $D^i \times \{0\}$. Then the composition $\phi_X \circ H \colon [0,1] \times X \to B$ is a homotopy from $\phi_X$ to a map which factors over a point on each fibre. Applying the bundle homotopy covering theorem, we cover this by a homotopy from $\varphi_X$ to a bundle map which is of the desired form for all $b \in B$.
\end{proof}

Pulling back the universal bundle along the inclusion, we get a map of Thom spaces
\[\mr{Thom}^\mr{const}(\pi_{\upsilon}^* \gamma^\perp_{d,n+d}) \lra \mr{Thom}(\pi_{\upsilon}^* \gamma^\perp_{d,n+d})\]
which is a pointed weak equivalence.

\begin{lemma}\label{lem.smooththomgeomrel} If $\upsilon_\bullet$ is a $d$-dimensional semisimplicial vector bundle, then the canonical map $||MT^\mr{Diff}\upsilon_{\bullet}|| \to MT^\mr{Diff}||\upsilon_{\bullet}||$ is a weak equivalence.\end{lemma}

\begin{proof}We apply Lemma \ref{lem.thomconstsmooth} to get pointed weak equivalences $\mr{Thom}^\mr{const}(\pi_{\upsilon_p}^* \gamma^\perp_{d,n+d}) \hookrightarrow \mr{Thom}(\pi_{\upsilon_p}^* \gamma^\perp_{d,n+d})$. Since thick geometric realisation preserves levelwise weak equivalences by \cite[Theorem 2.2]{rwebertsemi}, the horizontal maps in the commutative diagram
	\[\begin{tikzcd}  {||\mr{Thom}^\mr{const}(\pi_{\upsilon_\bullet}^* \gamma^\perp_{d,n+d})||} \dar \rar{\simeq} & {||\mr{Thom}(\pi_{\upsilon_\bullet}^* \gamma^\perp_{d,n+d})||} \dar \\
	\mr{Thom}^\mr{const} (\pi_{||\upsilon_\bullet||}^* \gamma^\perp_{d,n+d}) \rar{\simeq} &\mr{Thom}(\pi_{||\upsilon_\bullet||}^* \gamma^\perp_{d,n+d})\end{tikzcd}\]
are pointed weak equivalences. Hence it suffices to establish a weak equivalence
	\[||\mr{Thom}^\mr{const}(\pi_{\upsilon_\bullet}^* \gamma^\perp_{d,n+d})|| \lra \mr{Thom}^\mr{const}(\pi_{||\upsilon_\bullet||}^* \gamma^\perp_{d,n+d}).\]
We first observe that the quotient map $\mr{Bun}(TM,\bR^d) \times_{\mr{GL}_d(\bR)} \upsilon \to \mr{Bun}^\mr{Diff,const}(TM,\upsilon)$ is a homeomorphism. Since thick geometric realisatio commutes with products (see \cite[Remark 2.23]{gepnerhenriques}) and quotients (since colimits commute with colimits), this implies that $||\mr{Bun}^\mr{const}(TP,\upsilon_\bullet)|| \to \mr{Bun}^\mr{const}(TP,||\upsilon_\bullet||)$ is also a homeomorphism, where the superscript as always means that the underlying map $\phi_P$ of a bundle map $\varphi_P$ factors over a point. Similarly, the formation of unit sphere and unit disc bundles commute with thick geometric realisation, so we have a homeomorphism between the diagrams defining the two Thom space constructions. Finally, thick geometric realisation commutes with homotopy pushouts up to weak equivalence, and as $||{\ast}||$ (with $\ast$ the constant semisimplicial space on a point) is contractible preserves homotopy cofibres up to weak equivalence.\end{proof}

\subsubsection{Excision for spectra of smooth manifolds} We claim that $\upsilon \mapsto \Psi^{\mr{Diff},\upsilon}(d)$ satisfies similar properties.

\begin{lemma}\label{lem.vectorweq} The construction $\upsilon \mapsto \Psi^{\mr{Diff},\upsilon}(d)$ takes weak equivalences to weak equivalences.\end{lemma}

\begin{proof}Let $\varphi \colon \upsilon \to \upsilon'$ be a weak equivalence of $d$-dimensional vector bundles, with underlying map of base spaces $\phi \colon B \to B'$. It suffices to establish a levelwise pointed weak equivalence $\Psi^{\mr{Diff},\upsilon}_d(\bR^n) \to \Psi^{\mr{Diff},\upsilon'}_d(\bR^n)$: given a commutative diagram
	\[\begin{tikzcd} \partial D^i \rar \dar & \Psi^{\mr{Diff},\upsilon}_d(\bR^n) \dar \\[-2pt]
	D^i \rar & \Psi^{\mr{Diff},\upsilon'}_d(\bR^n),\end{tikzcd}\]
we need to provide a lift after a homotopy through commutative diagrams fixing $\varnothing$. Using smooth approximation as in \cref{sec.ssetsmooth}, we may assume the bottom map is represented by a smooth neat submanifold $X \subset D^i \times \bR^n$ such that the map $\pi \colon X \hookrightarrow D^i \times \bR^n \to D^i$ is a  smooth submersion of relative dimension $d$, together with an $\upsilon'$-structure on its vertical tangent bundle. The latter is described by a map of $d$-dimensional vector bundles $\varphi_X' \colon T_\pi X \to \upsilon'$ with underlying map of base spaces $\phi'_X \colon X \to B'$. We also have a lift to an $\upsilon$-structure over $\partial D^i$, described by a map $\varphi_{X|_{\partial D^i}} \colon T_\pi X|_{\partial D^i} \to \upsilon$ such that $\varphi \circ \varphi_{X|_{\partial D^i}} = \varphi_X'|_{T_\pi X|_{\partial D^i}}$. 

By \cite{milnorcw}, $(X,X|_{\partial D^i})$ has the homotopy type of a pair of CW-complexes. As the map $\phi \colon B \to B'$ underlying $\varphi$ is a weak equivalence, this implies that we can lift the map $\phi'_X \colon X \to B'$ underlying $\varphi'_X$ to a map $X \to B$ after a homotopy through commutative diagrams (only modifying the maps to $B$ and $B'$).

Using the bundle homotopy covering theorem, we extend $\varphi_X|_{\partial D^i}$ to a bundle map $\varphi_{X|_{[0,1] \times \partial D^i}} \colon [0,1] \times T_\pi X|_{\partial D^i} \to \upsilon$ covering the homotopy $[0,1] \times X|_{\partial D^i} \to B$. We may combine $g \circ \varphi_{[0,1] \times X|_{\partial D^i}}$ and $\varphi'_X$ into a bundle map 
\[[0,1] \times T_\pi X|_{\partial D^i} \cup \{0\} \times T_\pi X \lra \upsilon',\]
whose underlying map $[0,1] \times X|_{\partial D^i} \cup \{0\} \times X \to B$ extends to $X \times I$. Again using the bundle homotopy covering theorem, we may extend this to a bundle map $\varphi'_{[0,1] \times X} \colon [0,1] \times T_\pi X \to \upsilon'$. That is, our homotopy through commutative diagrams may be extended to a homotopy through commutative diagrams of bundle maps.

Note that the map $\varphi \colon \upsilon \to \upsilon'$ provides an identification of the fibre $\upsilon_b$ of $\upsilon$ over $b$ with the fibre of $\upsilon'_{g(b)}$ of $\upsilon'$ over $g(b)$. Since at the end of the homotopy the map $\phi'_{X \times \{1\}} \colon X \to B'$ underlying $\varphi'_{X \times \{1\}}$ lifts to a map $\phi_{X \times \{1\}} \to B$ rel boundary, a lift of $\varphi'_{X \times \{1\}}$ to a bundle map $\varphi_{X \times\{1\}} \colon T_\pi X \to \upsilon$ is given uniquely using the identification 
\[\varphi_{\phi_{X \times \{1\}}(b)}^{-1} \circ (\varphi_{X \times\{1\}})_b \colon \{1\} \times TX_b \overset{\cong}\lra \upsilon_{\phi_{X \times \{1\}}(b)}.\qedhere\]\end{proof}

As for Thom spectra, we claim that there is a canonical map of spectra \[||\Psi^\mr{Diff,\upsilon_{\bullet}}(d)|| \lra \Psi^{\mr{Diff},||\upsilon_{\bullet}||}(d).\] 
To describe it, we need to give pointed maps $||\Psi^\mr{Diff,\upsilon_{\bullet}}_d(\bR^n)|| \to \Psi^{\mr{Diff},||\upsilon_{\bullet}||}_d(\bR^n)$. For each $k$ we have a pointed map $\Delta^k_+ \wedge \Psi^\mr{Diff,\upsilon_k}_d(\bR^n) \to\Psi^{\mr{Diff},||\upsilon_{\bullet}||}_d(\bR^n)$ given by sending $(\vec{t},X,\varphi_X)$ to the manifold $X$ with tangential structure $\{\vec{t}\} \times \varphi_X \colon TX \to \Delta^k \times \upsilon_k \to ||\upsilon_{\bullet}||$. This is pointed because it sends $\varnothing$ to $\varnothing$.

The following excision result for $\upsilon \mapsto \Psi^\mr{Diff,\upsilon}(d)$ may be deduced from Theorem \ref{thm.gmtw} and the corresponding excision result for Thom spectra, but we will prove it directly.

\begin{theorem}\label{thm.excisiontangentsmooth} If $\upsilon_\bullet$ is a levelwise numerable $d$-dimensional semisimplicial vector bundle, then the canonical map $||\Psi^\mr{Diff,\upsilon_{\bullet}}(d)|| \to \Psi^{\mr{Diff},||\upsilon_{\bullet}||}(d)$ is a weak equivalence.\end{theorem}

To prove this theorem, it will be convenient to prove a stronger result.

\begin{definition}The subspace $\Psi^\mr{Diff,\upsilon,const}_d(\bR^n)$ of $\Psi^\mr{Diff,\upsilon}_d(\bR^n)$ consists of those $(X,\varphi_X)$ such that the map $\phi_X \colon X \to B$ underlying $\varphi_X \colon TX \to \upsilon$ factors over a point.\end{definition}

\begin{lemma}\label{lem.tangentconstsmooth} If the base $B$ of $\upsilon$ is a CW complex, the inclusion $\Psi^\mr{Diff,\upsilon,const}_d(\bR^n) \hookrightarrow \Psi^\mr{Diff,\upsilon}_d(\bR^n)$ is a pointed weak equivalence.\end{lemma}

\begin{proof}Suppose we are given a commutative diagram
\[\begin{tikzcd} \partial D^i \rar \dar & {\Psi^\mr{Diff,\upsilon,const}_d(\bR^n)} \dar \\[-2pt]
D^i \rar &  {\Psi^{\mr{Diff},\upsilon}_d(\bR^n)}.\end{tikzcd}\]
Then we need to provide a lift after a homotopy through commutative diagrams fixing $\varnothing$. Using smooth approximation as in \cref{sec.ssetsmooth}, we may assume that the commutative diagram is represented by a closed smooth neat submanifold $X \subset D^i \times \bR^n$ such that $\pi \colon X \hookrightarrow D^i \times \bR^n \to D^i$ is a smooth submersion of relative dimension $d$, together with a $\upsilon$-structure $\varphi_X \colon T_\pi X \to \upsilon$ on its vertical tangent bundle. Let $\phi_X \colon X \to B$ denote its underlying map. For every $b \in \partial D^i$, the restriction of $\phi_X$ to $X_b$ factors over a point; in this case, we say that it is \emph{of the desired form}.

Let $B_r \subset \bR^n$ denote the closed ball of radius $r$ around the origin. Since $D^i \times B_1$ is compact and $X$ is closed, $X \cap (D^i \times B_1)$ is compact. Hence there is a finite subcomplex $K$ of $B$ such that the image of $X \cap (D^i \times B_1)$ under $\phi_X$ lies in $K$. Any finite CW-complex is an ANR, so in particular $K$ is. Hence there exists an open neighbourhood $U$ of the diagonal $\Delta_K \subset K\times K$ which deformation retracts onto $\Delta_K$ via a homotopy $H \colon [0,1] \times U \to U$. Evaluating the map $\phi_X$ at pairs of points in a fibre of $X \cap (D^i \times B_1)$, we obtain a map
\[\rho \colon \{(x,x') \in X \times_{\pi} X \mid ||x||,||x'|| \leq 1\} \lra K \times K,\]
with $||-||$ denoting the Euclidean norm on $\bR^n$, sending the diagonal to the diagonal. Since $X \cap (D^i \times B_1)$ is compact, there exists an $r \in (0,1)$ such that $\rho^{-1}(U)$ contains $\{(x,x') \in X \times_{\pi} X \mid ||x||,||x'||<r\}$.

Let $V \subset D^i$ be the open subset $\{ b \in D^i \mid X_b \neq \varnothing\}$, i.e.\ those points over which the fibre of $X$ is non-empty. Since $X|_V \to V$ is a submersion with non-empty fibres, it has local sections. Since $X \cap (D^i \times B_{r/2})$ is compact, so is the subset $L$ of $V$ given by $\{b \in D^i \mid X \cap (\{b\} \times B_{r/2}) \neq \varnothing\}$. We can thus find an $0<r'<r$ and a finite open cover $\cV = \{V_i\}_{1 \leq i \leq s}$ of some open neighbourhood $V'$ of $L$ in $V$, such that over each $V_i$ we have a section $s_i \colon V_i \to X|_{V_i}$ which lands in $X \cap B_{r'}$. Pick a partition of unity subordinate to $\cV$. By normalization, we construct from this partition of unity a collection of compactly-supported functions $\lambda_i \colon V_i \to [0,1]$ such that for each $v \in V'$ at least one $\lambda_i$ takes value $1$ at $v$. Pick a smooth function $\eta \colon X \to [0,1]$ that is $1$ on $X \cap B_{r'}$ and $0$ on the complement of $X \cap B_r$ in $X$. This shall be used to control a homotopy of tangent maps, guaranteeing it is supported near the origin of $\bR^n$.

We will inductively modify the tangent maps over each $V_i$, and keep the submersion unchanged. After the $i$th step the tangential structure will be of the desired form on $X_b \cap B_{r'}$ for all $b \in \bigcup_{1 \leq j \leq i} \lambda_j^{-1}(1)$. During each step, if the tangential structure was already of the desired form on $X_b \cap B_{r'}$, it will remain so during the entire homotopy. Furthermore, over $\partial D^i$ we will remain in $\Psi_d^\mr{Diff,\upsilon,const}(\bR^n)$ at all times. At the end of this procedure, the tangential structure will be of the desired form on all of $X \cap (D^i \times B_{r'})$.

For the $i$th step, on the set $X_{i,r} \coloneqq X \cap (V_i \times \mr{int}(B_r))$, consider the homotopy $h \colon [0,1] \times X_{i,r} \to B$ given by
\[h(t,x) \coloneqq \begin{cases} \pi_1 \circ H_{2t}(\rho(x,s_i(\pi(x)))) & \text{if $t \leq 1/2$,} \\
\pi_2 \circ H_{2-2t}(\rho(x,s_i(\pi(x)))) & \text{if $t > 1/2$.}\end{cases}\]
This starts at $\phi_X$ and ends at the map $\phi_X(s_i(\pi(-)))$. Note that if for some $x \in X_{i,r}$ it is true that $\phi_X(x) = \phi_X(s_i(\pi(x)))$, this homotopy is constant at $x$ with value $\phi_X(x)$, because $H$ fixes the diagonal pointwise. We can produce a homotopy $\tilde{h} \colon [0,1] \times X \to B$ by cutting this off using $\eta$ and $\lambda_i$:
\[\tilde{h}(t,x) \coloneqq h\left[\eta(x)\lambda_i(\pi(x))t,x \right].\]
This may be covered by a homotopy of vector bundle maps starting at $\varphi_X$. At the end of this homotopy the tangential structure is of the desired form on $X_b \cap B_{r'}$ for $b \in \lambda_i^{-1}(1)$, in addition to those $b$ where it already was of the desired form. 

Now we zoom in: pick a smooth family of self-embeddings $\psi_t \colon \bR^n \to \bR^n$ such that $\psi_0$ is the identity and $\psi_1$ has image in $B_{r'}$ and create a map $D^i \times [0,1] \to \Psi^{\mr{Diff},\upsilon}_d(\bR^n)$ by taking $(\psi_t)^* \overline{X}$. This gives a homotopy through commutative diagrams to one where the tangent map is of the desired form everywhere, that is, the map $D^i \to \Psi^{\mr{Diff},\upsilon}_d(\bR^n)$ factors through the subspace $\Psi^\mr{Diff,\upsilon,const}_d(\bR^n)$.\end{proof}

\begin{proof}[Proof of Theorem \ref{thm.excisiontangentsmooth}] It suffices to prove that the map is a levelwise pointed weak equivalence. Since $\upsilon \mapsto \Psi_d^\mr{Diff,\upsilon}(\mathbb{R}^n)$ takes weak equivalences to weak equivalences by Lemma \ref{lem.vectorweq}, we may replace the base $B_k$ of $\upsilon_k$ by $|\mr{Sing}(B_\bullet)|$ and hence assume it is the geometric realisation of a simplicial set. Thus each $B_p$ may be assumed to be a CW complex (a geometric realisation of a simplicial set), as well as $||B_\bullet||$ (a geometric realisation of a semisimplicial simplicial set).   
	
	Now we apply Lemma \ref{lem.tangentconstsmooth} to $\Psi_d^\mr{Diff,\upsilon_p}(\mathbb{R}^n)$ to get weak equivalences \[\Psi_d^\mr{Diff,\upsilon_p,const}(\mathbb{R}^n) \overset{\simeq}\lra \Psi_d^\mr{Diff,\upsilon_p}(\mathbb{R}^n).\] Since thick geometric realisation preserves levelwise weak equivalences by \cite[Theorem 2.2]{rwebertsemi}, we get a weak equivalence $||\Psi_d^\mr{Diff,\upsilon_\bullet,const}(\mathbb{R}^n)|| \to ||\Psi_d^\mr{Diff,\upsilon_\bullet}(\mathbb{R}^n)||$. 
	
	The canonical map $\Delta^k \times \Psi_d^\mr{Diff,\upsilon_k}(\mathbb{R}^n) \to \Psi_d^\mr{Diff,||\upsilon_{\bullet}||}(\mathbb{R}^n)$ sends $\Delta^k \times \Psi_d^\mr{Diff,\upsilon_k,\mr{const}}(\mathbb{R}^n)$ into $\Psi_d^\mr{Diff,||\upsilon_{\bullet}||,const}(\mathbb{R}^n)$, so we get a commutative diagram
	\[\begin{tikzcd} {||\Psi_d^\mr{Diff,\upsilon_\bullet,const}(\mathbb{R}^n)||} \rar{\simeq} \dar & {||\Psi_d^\mr{Diff,\upsilon_\bullet}(\mathbb{R}^n)||} \dar \\
	\Psi_d^\mr{Diff,||\upsilon_{\bullet}||,const}(\mathbb{R}^n) \rar{\simeq} & \Psi_d^\mr{Diff,||\upsilon_{\bullet}||}(\mathbb{R}^n).\end{tikzcd}\]
	We showed above that the top map is a weak equivalence, and the bottom map is also a weak equivalence by another application of Lemma \ref{lem.tangentconstsmooth}. To prove the theorem it hence suffices to prove that the left map is a weak equivalence. It is in fact a homeomorphism. It is visibly a continuous bijection and verifying that its inverse is continuous is tedious (it will also be irrelevant when working with simplicial sets as in the topological case). To verify its inverse is continuous, we recall the construction of the topology on the spaces $\Psi_d^\mr{Diff,\upsilon}(-)$ in \cite[Section 2.2 \& 2.3]{grwmonoids}. The space $\Psi_d^\mr{Diff,\upsilon}(-)$ is a colimit as $r \to \infty$ of spaces $\Psi_d^\mr{Diff,\upsilon}(B_r \subset \bR^n)^\mr{cs}$, each of which in turn can be described a colimit as $\epsilon \to 0$ of spaces $\Psi_d^\mr{Diff,\upsilon}(\mr{int}(B_{r+\epsilon}))^\mr{cs}$. Each of these is homeomorphic to $\Psi_d^\mr{Diff,\upsilon}(\bR^n)^\mr{cs}$. All of these constructions have analogues with $\upsilon$-structure whose underlying map factors over a point. Since colimits commute, it suffices to prove that the continuous bijections \[||\Psi_d^\mr{Diff,\upsilon_\bullet,const}(\bR^n)^\mr{cs}|| \lra \Psi_d^\mr{Diff,||\upsilon_\bullet||,const}(\bR^n)^\mr{cs}\] are homeomorphisms. The topology $\Psi_d^\mr{Diff,\upsilon,const}(\bR^n)^\mr{cs}$ is constructed by demanding certain partially defined maps $\Gamma_c(NM) \times \mr{Bun}^\mr{Diff,const}(TM,\upsilon) \dashrightarrow \Psi_d^\mr{Diff,\upsilon,const}(\bR^n)^\mr{cs}$ are homeomorphisms onto open subsets. Since products commute with geometric realisation, it hence suffices to prove that the continuous bijections \[||\mr{Bun}^\mr{Diff,const}(TM,\upsilon_\bullet)|| \lra \mr{Bun}^\mr{Diff,const}(TM,||\upsilon_\bullet||)\] are homeomorphisms, which we already did in the proof of Lemma \ref{lem.smooththomgeomrel}.
\end{proof}

\section{Spaces of topological manifolds} \label{sec.spacestopmfds} We start by setting up spaces of topological manifolds, analogous to Section \ref{sec.smoothmanifolds}.

\subsection{Spaces of topological submanifolds} 

Analogous to the smooth simplices in $\Psi^{\mr{Diff}}_d(M)$ of Section \ref{sec.ssetsmooth}, we define a simplicial set of $d$-dimensional topological manifolds in $M$. We point out that in this paper, by Convention \ref{conv.locallyflat}, all embeddings and submanifolds are locally flat.

\begin{definition}Let $M$ be a topological manifold of dimension $m$. The \emph{space $\Psi^\mr{Top}_d(M)$ of $d$-dimensional manifolds in $M$} is the simplicial set with $k$-simplices given by the set of closed topological submanifolds $X \subset \Delta^k \times M$ such that $\pi \colon (\Delta^k \times M,X) \to \Delta^k$ is a relative topological submersion of relative dimension $(m,d)$.

A map $\theta \colon [k] \to [k']$ induces a map from $k'$-simplices to $k$-simplices by pulling back along the map $\theta_* \colon \Delta^k \to \Delta^{k'}$. That is, $X \subset \Delta^{k'} \times M$ is sent to $\theta^* X \subset \Delta^k \times M$ given by $\{(\vec{t},m) \mid (\theta_*(\vec{t}), m) \in X\}$.
\end{definition}

Unlike in the smooth case, we can glue and pull back topological submersions along a retraction $\Delta^k \to \Lambda^k_i$ to prove that the simplicial set $\smash{\Psi^\mr{Top}_d(M)}$ is Kan. Moreover, it extends to a \emph{quasitopological space} in the sense of \cite{spanierquasi}: its value on a topological space $Y$ is given by replacing $\Delta^k$ by $Y$ in the previous definition and through pullback, this construction is contravariantly functorial in all continuous maps. We shall use this extension to a quasitopological space, we turn homotopy-theoretic arguments into more geometric ones, e.g.\ representing elements of homotopy groups by $X \subset \partial D^i \times M$. In doing so, we can meaningfully talk about $X_b$ for any point $b \in \partial D^i$, and manipulate $X$ fibrewise; indeed, being a closed submanifold or a relative topological submersion are properties of $X$, not additional data.

Let us discuss the dependence of the simplicial set $\Psi^\mr{Top}_d(M)$ on $M$.

\begin{definition}\
	\begin{itemize}
		\item The category $\mr{sSet}$ has class of objects given by simplicial sets and morphisms from $X$ to $Y$ given by maps of simplicial sets $X \to Y$.
		\item The $\mr{sSet}$-enriched category $\cat{sSet}$ has class of objects given by simplicial sets and simplicial set of morphisms from $X$ to $Y$ given by $\mr{Map}(X,Y)$.
		\item The $\mr{sSet}$-enriched category $\cat{Mfd}^\mr{Top}$ has a set of objects given by the topological submanifolds of $\bR^\infty$, and the morphism space from $M$ to $N$ given by the simplicial set $\mr{Emb}^\mr{Top}(M,N)$ of codimension zero embeddings, as in Definition \ref{def.embeddings}.
	\end{itemize}
\end{definition}

\begin{definition}An \emph{invariant simplicial sheaf on topological manifolds} is a functor of $\mr{sSet}$-enriched categories $\Phi \colon (\cat{Mfd}^\mr{Top})^\mr{op} \to \cat{sSet}$, such that for any open cover $\{U_i\}_{i \in I}$ of a topological manifold $M$, the following is an equalizer diagram in $\mr{sSet}$:
	\[\begin{tikzcd} \Phi(M)\ar[r]& \prod_i \Phi(U_i) \rar[shift left=.5 ex] \rar[shift left=-.5 ex]& \prod_{i,j} \Phi(U_i \cap U_j).\end{tikzcd}\]
\end{definition}

Given an embedding $g \colon M \to M'$, there is a map $g^* \colon \Psi_d^\mr{Top}(M') \to \Psi_d^\mr{Top}(M)$ sending a $k$-simplex $X \subset \Delta^k \times M'$ to $(\mr{id}_{\Delta^k} \times g)^{-1}(X) \subset \Delta^k \times M$. The locality of the existence of submersion charts in Definition \ref{def.submersion} implies that this is again a submersion. Furthermore, as submersions can be glued along open covers of the domain, we have the following lemma.

\begin{lemma}The assignment $M \mapsto \Psi^\mr{Top}_d(M)$ extends to a functor $\Psi^\mr{Top}_d \colon (\cat{Mfd}^\mr{Top})^\mr{op} \to \cat{sSet}$ which is an invariant simplicial sheaf on topological manifolds.\end{lemma}

\subsection{Spaces of topological smanifolds with tangential structure} As in Section \ref{sec.smoothtangentialstruc}, we consider topological manifolds with tangential structures. Given a topological submersion $\pi \colon X \to \Delta^k$ of relative dimension $d$, the tangent microbundles of the fibres assemble into to the vertical tangent microbundle $T_\pi X$ described in Example \ref{exam.vertical}, a $d$-dimensional topological microbundle. Evidently this is natural in the morphisms of $\Delta$ up to canonical isomorphism. Though it may seem natural to let tangential structures be $d$-dimensional topological microbundles, we shall use instead \emph{topological $\bR^d$-bundles}, that is, fibre bundles with fibre $\bR^d$ and structure group $\mr{Top}(d)$. Compare this to the definition of a $d$-dimensional vector bundle, which is a fibre bundle with fibre $\bR^d$ and structure group $\mr{GL}_d(\bR)$. This choice is more convenient because the theory of fibre bundles is better developed in the literature, and comes at hardly any loss of generality due to the Kister-Mazur theorem (Theorem \ref{thm.kister}, which holds for base spaces which are ENR's or paracompact).

\begin{definition}A $d$-dimensional \emph{topological tangential structure} is a numerable topological $\bR^d$-bundle $\xi$.\end{definition}

\begin{definition}Let $M$ be a topological manifold and $\xi$ a $d'$-dimensional tangential structure with $d' \geq d$. The \emph{space $\Psi^\mr{Top,\xi}_d(M)$ of $d$-dimensional manifolds in $M$ with $\xi$-structure} is given by the following simplicial set: its $k$-simplices are given by those pairs $\overline{X} = (X,\phi_X)$ of a $k$-simplex $X \in \Psi_d^\mr{Top}(M)_k$ and a map of $d'$-dimensional microbundles $\varphi_X \colon T_\pi X \oplus \epsilon^{d'-d} \to \xi$. A morphism $\theta \colon [k] \to [k']$ in $\Delta$ sends $\overline{X} = (X,\varphi_X)$ to the pair $(\theta^* X,\varphi_X|_{T_\pi \theta^* X \oplus \epsilon^{d'-d}})$.\end{definition}

Like $\Psi^\mr{Top}_d(M)$, the simplicial set $\Psi^\mr{Top,\xi}_d(M)$ is Kan and extends to a quasitopological space. We again discuss the functoriality in $M$. Given a codimension zero embedding $g \colon M \to M'$, there is a map $g^* \colon \Psi^{\mr{Top},\xi}_d(M') \to \Psi^{\mr{Top},\xi}_d(M)$. It is given on the submersion as before, and if $\phi_X$ is the $\xi$-structure on $X$, we give $g^{-1}(X)$ the $\xi$-structure $\phi_X \circ Tg \colon T_\pi(g^{-1}(X)) \oplus \epsilon^{d'-d} \to \xi$, where $Tg \colon T_{\pi}(g^{-1}(X))\oplus \epsilon^{d'-d} \to T_\pi X\oplus \epsilon^{d'-d}$ is the map on total spaces induced by the homeomorphism $g^{-1}(X) \rightarrow X$. Having described the functoriality in $M$, the locality of the data and conditions implies:

\begin{lemma} The assignment $M \mapsto \Psi^{\mr{Top},\xi}_d(M)$ extends to a functor of $\mr{sSet}$-enriched categories $\Psi^{\mr{Top},\xi}_d(-) \colon (\cat{Mfd}^\mr{Top})^{\mr{op}} \to \cat{sSet}$, which is an invariant simplicial sheaf on topological manifolds.\end{lemma}

\subsection{The universal tangential structure} As discussed for $d$-dimensional vector bundles in Section \ref{sec.smoothuniversal}, there exist universal topological $\bR^d$-bundles. The associated $\bR^d$-bundle to the universal numerable principal $\mr{Top}(d)$-bundle over the bar construction $B\mr{Top}(d)$ provides an example \cite{mayclassifying}. There is also a Grassmannian model, analogous to the description of $\mr{Gr}^\mr{Diff}_d(\bR^{n+d})$ as $\mr{O}(n+d)/(\mr{O}(d) \times \mr{O}(n))$.

Let $\mr{Top}(n+d\, \mr{fix}\,d)$ denote the simplicial group of homeomorphisms of $\bR^{n+d}$ fixing $\bR^d$ pointwise, and $\mr{Top}(n+d\,\mr{pres}\,d)$ the simplicial group of homeomorphisms of $\bR^{n+d}$ preserving $\bR^d$ setwise, so that $\mr{Top}(n+d,\mr{pres}\,d) \cong \mr{Top}(d) \ltimes \mr{Top}(n+d\, \mr{fix}\,d)$.

\begin{definition}\label{def.topgrassmannian} The \emph{topological Grassmannian $\mr{Gr}^\mr{Top}_d(\bR^{n+d})$ of $d$-planes in $\bR^{n+d}$} is the simplicial set $\mr{Top}(n+d)/\mr{Top}(n+d\,\mr{pres}\,d)$. We let $\mr{Gr}^\mr{Top}_d(\bR^\infty)$ denote $\mr{colim}_{n \to \infty} \mr{Gr}^\mr{Top}_d(\bR^{n+d})$ and call this \emph{topological Grassmannian of $d$-planes}.\end{definition}

\begin{remark}We give two warnings. Firstly, unlike the smooth case, $\mr{Top}(n+d \,\mr{fix}\,d)$ is not homotopy equivalent to $\mr{Top}(n)$ in general \cite[Example 2]{millettnormal}. Secondly, there is a map of simplicial sets
\[\mr{Top}(n+d)/\mr{Top}(n+d\,\mr{pres}\,d) \lra \Psi_d^\mr{Top}(\bR^{n+d})\]
given by sending $f$ to $f(\bR^{d} \times \{0\})$. Its image is contained in the subsimplicial set with $k$-simplices given by those $X \subset \Delta^k \times \bR^{n+d}$ such that each fibre $X_b$ contains the origin and is homeomorphic to $\bR^d$, but this is not an isomorphism (e.g.\ when taking $d=1$, $n=2$, it does not hit the components of knotted $\bR \subset \bR^3$).\end{remark}

\begin{lemma}\label{lem.topgrassmannian} The topological Grassmannian $\mr{Gr}^\mr{Top}_d(\bR^\infty)$ is weakly equivalent to $B\mr{Top}(d)$.\end{lemma}

\begin{proof}Writing $\mr{Top}(n+d)/\mr{Top}(n+d\,\mr{pres}\,d)$ as $(\mr{Top}(n+d)/\mr{Top}(n+d\,\mr{fix}\,d))/\mr{Top}(d)$, we see that the action of $\mr{Top}(d)$ is free and hence it suffices to prove that the quotient $\mr{Top}(n+d)/\mr{Top}(n+d\,\mr{fix}\,d)$ becomes highly connected as $n \to \infty$. This follows by combining \cite[Proposition (t/pl)]{lashofembeddings} with \cite[Theorem 3]{haefligerwall}; the former says that if $n \geq 3$, $\mr{PL}(n+d)/\mr{PL}(n+d\,\mr{fix}\,d) \to \mr{Top}(n+d)/\mr{Top}(n+d\,\mr{fix}\,d)$ is a weak equivalence, and the latter says that $\mr{PL}(n+d)/\mr{PL}(n+d\,\mr{fix}\,d)$ is $n$-connected. These papers use germs of topological or PL embeddings instead of topological or PL automorphisms, that is, $\mr{Top}_{(0)}(d)$ instead of $\mr{Top}(d)$, etc. However, such simplicial sets of automorphisms are weakly equivalent to the corresponding simplicial set of germs of automorphisms, as a consequence of the Kister-Mazur theorem or its PL analogue, the Kuiper-Lashof theorem (Remark \ref{rem.germsweq}).\end{proof}

\begin{remark}Alternatively, that $\mr{Top}(n+d)/\mr{Top}(n+d\,\mr{fix}\,d)$ becomes highly connected as $n \to \infty$ may be deduced directly using \cite[Theorem 2.3]{rourkesandersontop}.\end{remark}

The map $\mr{Top}(n+d)/\mr{Top}(n+d\,\mr{fix}\,d) \to \mr{Top}(n+d)/ \mr{Top}(n+d\,\mr{pres}\,d)$ is a principal $\mr{Top}(d)$-bundle, and its geometric realisation has an associated topological $\bR^d$-bundle. As we let $n \to \infty$, this gives a universal numerable topological $\bR^d$-bundle. (If we had used topological groups instead of simplicial groups, we could have explicitly described it as $\mr{Top}(n+d)/\mr{Top}(n+d\,\mr{pres}\,d\,\mr{fix}\,\{e_1\})$ where $\mr{Top}(n+d\,\mr{pres}\,d\,\mr{fix}\,\{e_1\})$ is the topological group of homeomorphisms of $\bR^d$ fixing the origin and the point $e_1 \in \bR^d$.) Picking a universal numerable topological $\bR^d$-bundle $\xi^\mr{univ}$, we have an analogue of Lemma \ref{lem.universaltangential}. 

\begin{lemma}\label{lem.universaltangentialtop} The natural transformation $\Psi^\mr{Top,\xi^{univ}}_d(-) \to \Psi^\mr{Top}_d(-)$ is a natural weak equivalence.\end{lemma}

\begin{proof}Suppose we have a commutative diagram
	\[\begin{tikzcd}\partial D^i \rar \dar & \Psi^\mr{Top,\xi^{univ}}_d(M) \dar \\[-2pt]
	D^i \rar & \Psi^\mr{Top}_d(M).\end{tikzcd}\]
Then we need to provide a lift after a homotopy of commutative diagrams. This commutative diagram may be represented by a closed topological submanifold $X \subset D^i \times M$ such that $(D^i \times M,X) \to D^i$ is a relative topological submersion, and over $\partial D^i$ we have a map of $d$-dimensional topological microbundles $T_\pi X|_{\partial D^i} \to \xi$.

Using the Kister--Mazur theorem (\cref{thm.kister}) we may pick a topological $\bR^d$-bundle inside $T_\pi X|_{\partial D^i}$. The map $T_\pi X|_{\partial D^i} \to \xi$ then induces a fibrewise embedding of $\bR^d$ into $\bR^d$, and by Kister's weak equivalence $\mr{Emb}^\mr{Top}(\bR^d,\bR^d) \simeq \mr{Top}(d)$, we may deform this map to one of topological $\bR^d$-bundles. At this point, the proof of Lemma \ref{lem.universaltangential} goes through, replacing vector bundles with topological $\bR^d$-bundles throughout.\end{proof}

\subsection{Spectra of topological submanifolds} As in the smooth case, we will break up our computation of the homotopy type of the topological cobordism category into two steps. The first identifies it as an infinite loop space of the following spectrum, analogous to the one defined in Section \ref{sec.smoothspectra}:

\begin{definition}\label{def.topspectrum} The spectrum $\Psi^{\mr{Top},\xi}(d)$ has $n$th level given by $\Psi^{\mr{Top},\xi}_d(\bR^n)$, pointed at $\varnothing$, and structure maps $S^1 \wedge \Psi^{\mr{Top},\xi}_d(\bR^n) \to \Psi^{\mr{Top},\xi}_d(\bR^{n+1})$ given by identifying $S^1$ with the one-point compactification of $\bR$, and defining it to be
	\[(t,\overline{X}) \longmapsto \begin{cases} \varnothing & \text{if $(t,\overline{X}) = (\infty,\overline{X})$,} \\
	\iota (\overline{X}) + t \cdot e_1 & \text{otherwise,}\end{cases}\]
with $\iota \colon \bR^n \to \bR^{n+1}$ the inclusion on the last $n$ coordinates.\end{definition}

\section{The topological cobordism category} \label{sec:cobordismcat} In this section we define cobordism categories of topological manifolds.

\subsection{Restricted spaces of manifolds} As in the smooth case, we need to consider spaces of long topological manifolds. 

\begin{definition}For $0 \leq p \leq n$, we let $\psi^{\mr{Top},\xi}(d,n,p)$ be the subsimplicial set of $\Psi^{\mr{Top},\xi}_d(\bR^{n})$ where $k$-simplices are given by those $\overline{X} = (X,\varphi_X) \in \Psi^{\mr{Top},\xi}_d(\bR^n)_k$ such that $X \subset \Delta^k \times \bR^{p} \times (0,1)^{n-p}$. Let us denote ${\mr{colim}}_{n \to \infty}\, \psi^{\mr{Top},\xi}(d,n,p)$ by the colimit $\psi^{\mr{Top},\xi}(d,\infty,p)$ over the maps $\overline{X} \mapsto \iota(\overline{X})$.\end{definition}

When $p = 0$, the homotopy type of these spaces is more familiar. For a $d$-dimensional topological manifold $M$, the simplicial set $\mr{Bun}^\mr{Top}(TM,\xi)$ of microbundle maps from $TM$ to $\xi$ has an action of the simplicial group $\mr{Top}(M)$.

\begin{definition}We define $B\mr{Top}^\xi(M)$ to be the homotopy quotient $\mr{Bun}^\mr{Top}(TM,\xi) \sslash \mr{Top}(M)$.\end{definition}

This defines a homotopy type, and the homotopy quotient in this definition can be made explicit by picking a model, such as the thick geometric realisation of a two-sided bar construction, or as the quotient $E  \times_{\mr{Top}(M)} \mr{Bun}^\mr{Top}(TM,\xi)$ for a weakly contractible simplicial set $E$ with a free $\mr{Top}(M)$ action. The next lemma uses the latter.

\begin{lemma}\label{lem.stableobjectsbtop}  There is a weak equivalence
	\[\bigsqcup_{[M]} B\mr{Top}^\xi(M) \lra \psi^{\mr{Top},\xi}(d,\infty,0),\]
	where the indexing set runs over all homeomorphism types of compact $d$-dimensional  topological manifolds.\end{lemma}

\begin{proof}The fibres of a simplex of $\psi^{\mr{Top},\xi}(d,\infty,0)$ are closed in a bounded subset of a Euclidean space , hence compact. By Corollary \ref{cor.unionfibrebundle}, a closed topological submersion with compact fibres is a fibre bundle. Thus sending a pair of an embedding and a $\xi$-structure to their image gives an isomorphism of simplicial sets
	\[\bigsqcup_{[M]} \mr{Emb}^\mr{Top}(M,(0,1)^\infty) \times_{\mr{Top}(M)} \mr{Bun}^\mr{Top}(TM,\xi)\cong \psi^\mr{Top,\xi}(d,\infty,0).\]
	The simplicial set $\mr{Emb}^\mr{Top}(M,(0,1)^\infty)$ is weakly contractible by Lemma \ref{lem.weakwhitney} and the action of $\mr{Top}(M)$ by precomposition is free, so that the left hand side is weakly equivalent to $\bigsqcup_{[M]} B\mr{Top}^\xi(M)$.
\end{proof}

When $n \neq \infty$ and $d \neq 4$ the spaces $\psi^{\mr{Top},\xi}(d,n,0)$ approximate $\psi^{\mr{Top},\xi}(d,\infty,0)$; the map $\psi^{\mr{Top},\xi}(d,n,0) \to \psi^{\mr{Top},\xi}(d,n+1,0)$ is a highly-connected monomorphism of simplicial sets:

\begin{lemma}\label{lem.psi-connectivity} If $d \neq 4$, $n \geq 5$, and $n-d \geq 3$, the map $\psi^{\mr{Top},\xi}(d,n,0) \to \psi^{\mr{Top},\xi}(d,n+1,0)$ is $(n-2d-1)$-connected.\end{lemma}

\begin{proof}As in the proof of Lemma \ref{lem.stableobjectsbtop}, we identify $\psi^{\mr{Top},\xi}(d,n,0)$ with the disjoint union of the quotients $\mr{Emb}(M,(0,1)^n) \times_{\mr{Top}(M)} \mr{Bun}(TM,\xi)$. Since the action of $\mr{Top}(M)$ on $\mr{Emb}^\mr{Top}(M,(0,1)^n)$ and $\mr{Emb}^\mr{Top}(M,(0,1)^{n+1})$ is free, the connectivity of $\psi^{\mr{Top},\xi}(d,n,0) \to \psi^\mr{Top,\xi}(d,n+1,0)$ is equal to that of the maps $\mr{Emb}^\mr{Top}(M,(0,1)^n) \to \mr{Emb}^\mr{Top}(M,(0,1)^{n+1})$. By Theorem  \ref{thm.catwhitney}, the domain and codomain are $(n-2d-2)$-connected as long as $n \geq 5$ and $n-d \geq 3$, which gives the result.\end{proof}

\begin{remark}This proof fails for $d=4$ since we needed handle decompositions in our proof of the quantitative Whitney embedding theorem, and a topological 4-manifold admits a handle decomposition if and only if it admits a smooth structure.\end{remark}

\subsection{The definition}\label{sec.cobcatdefinition} We define topological cobordism categories following Definition \ref{def.smoothcob}. To do so, we explain what it means to take the product of an interval with a manifold with $\xi$-structure:

\begin{definition}Given $\overline{Y} = (Y,\varphi_Y) \in \psi^{\mr{Top},\xi}(d-1,n-1,0)$, the \emph{translational product} $(a,b) \times \overline{Y} \in \Psi_d^\mr{Top,\xi}((a,b) \times (0,1)^{n-1})$ is the $d$-dimensional manifold $(a,b) \times Y$ with $\xi$-structure defined as follows: the tangent microbundle of $(a,b) \times Y$ can be identified with $\epsilon \oplus TY$ by translation in the $(a,b)$-direction, and thus a $\xi$-structure on $Y$ gives a $\xi$-structure on $(a,b) \times Y$.\end{definition}

\begin{definition}\label{def.topcobordism} The \emph{topological cobordism category $\cat{Cob}^{\mr{Top},\xi}(d,n)$} is the non-unital category internal to $\mr{sSet}$ such that:
	\begin{itemize}
		\item the simplicial set of objects is given by $\psi^{\mr{Top},\xi}(d-1,n-1,0)$,
		\item the simplicial set of morphisms is given by the subsimplicial set of $\mr{Sing}((0,\infty)) \times \psi^{\mr{Top},\xi}(d,n,1)$ consisting of those $(t,\overline{X})$ such that there exists an $\epsilon > 0$ so that 
		\begin{align*}\qquad&\overline{X} \cap ((-\infty,\epsilon) \times (0,1)^{d-1}) = (-\infty,\epsilon) \times (\overline{X} \cap (\{0\} \times (0,1)^{n-1})), \quad \text{and} \\
		&\overline{X} \cap ((t-\epsilon,\infty) \times (0,1)^{n-1}) = (t-\epsilon,\infty)\times (\overline{X} \cap (\{t\} \times (0,1)^{n-1})),\end{align*}
		\item the source map sends $(t,\overline{X})$ to $\overline{X} \cap (\{0\} \times (0,1)^{n-1})$ and the target map sends it to $\overline{X} \cap (\{t\} \times (0,1)^{n-1})$,
		\item the composition map sends a pair of composable morphisms $(t_0,\overline{X}_0)$ and $(t_1,\overline{X}_1)$ to $(t_0+t_1,\overline{X})$ with $\overline{X}$ given by
		\[\qquad \overline{X} = \begin{cases}\overline{X}_0 \cap ((-\infty,t_0] \times (0,1)^{n-1}) & \text{in $(-\infty,t_0] \times (0,1)^{n-1}$,} \\
		\overline{X}_1 \cap ([0,\infty) \times (0,1)^{n-1})+t_0 \cdot e_1 & \text{in $[t_0,\infty) \times (0,1)^{n-1}$.}\end{cases}\]
	\end{itemize}\end{definition}

\begin{remark}Observe that though we wrote $t$, $t_0$ and $t_1$ in the previous definition these are actually continuous functions $\Delta^k \to \bR$.\end{remark}

The nerve of this non-unital category internal to $\mr{sSet}$ is a semisimplicial object in $\mr{sSet}$ and the thick geometric realisation of its semisimplicial direction is a simplicial set we denote $B\cat{Cob}^{\mr{Top},\xi}(d,n)$; this is the classifying space of the cobordism category $\cat{Cob}^{\mr{Top},\xi}(d,n)$.

In the remainder of this subsection we justify this definition by arguing that (i) as $n \to \infty$ the homotopy types of the objects, morphisms, and classifying spaces stabilise, and (ii) this nerve is a semi-Segal object. We start with (ii), which means that the map
\[\begin{tikzcd} N_p \cat{Cob}^{\mr{Top},\xi}(d,n) \dar \\[-5pt] \underbrace{N_1 \cat{Cob}^{\mr{Top},\xi}(d,n) \times^h_{N_0 \cat{Cob}^{\mr{Top},\xi}(d,n)} \cdots \times^h_{N_0 \cat{Cob}^{\mr{Top},\xi}(d,n)} N_1 \cat{Cob}^{\mr{Top},\xi}(d,n)}_{p}\end{tikzcd}\]
induced by the Segal morphisms in $\Delta_\mr{inj}$, is a weak equivalence. As it is a nerve, this formula holds with actual pullbacks in place of homotopy pullbacks, so it suffices to verify that the combined source and target maps of $\cat{Cob}^\mr{Top,\xi}(d,n)$ are Kan fibrations:

\begin{lemma}\label{lem.sourcetargetfib} The combined source and target map $(s,t) \colon \mr{mor}(\cat{Cob}^{\mr{Top},\xi}(d,n)) \to \psi^{\mr{Top},\xi}(d-1,n-1,0)^2$ is a Kan fibration.\end{lemma}

\begin{proof}Since a horn inclusion $\Lambda^k_i \hookrightarrow \Delta^k$ is homeomorphic to $\{0\} \times D^{k-1} \hookrightarrow [0,1] \times D^{k-1}$, we may consider the latter instead (and replace $k-1$ by $k$ for brevity). That is, we may suppose we are given a commutative diagram
	\[\begin{tikzcd} \{0\} \times D^k\rar{f} \dar & \mr{mor}(\cat{Cob}^{\mr{Top},\xi}(d,n)) \dar{(s,t)} \\[-2pt]
	{[0,1]} \times D^k \rar & \psi^{\mr{Top},\xi}(d-1,n-1,0)^2,\end{tikzcd}\]
and need to provide a lift. The bottom map is represented by a pair $\overline{Y}_0,\overline{Y}_1$ as follows: for $i=0,1$, $\overline{Y}_i = (Y_i,\varphi_{Y_i})$ consists of a closed topological submanifold $Y_i \subset [0,1] \times D^k \times (0,1)^{n-1}$ such that $\pi_i \colon ([0,1] \times D^k \times (0,1)^{n-1} \to [0,1],Y_i) \times D^k$ is a relative topological submersion of relative dimension $(n-1,d-1)$, together with a map $\varphi_{Y_i} \colon T_{\pi_i} Y_i \oplus \epsilon \to \xi$ of $d$-dimensional topological microbundles.

The top map  is represented by a continuous map $\tau \colon \{0\} \times D^k \to (0,\infty)$, and a $\overline{X} = (X,\varphi_X)$ as follows: $X$ is a closed topological submanifold of $\{0\} \times D^k \times \bR \times (0,1)^{n-1}$ such that $\pi \colon (\{0\} \times D^k \times \bR \times (0,1)^{n-1},X) \to \{0\} \times D^k$ is a relative topological submersion of relative dimension $(n,d)$, together with a map $\varphi_{X} \colon T_{\pi} X \to \xi$ of topological $\bR^d$-bundles. There should exist an $\epsilon > 0$ so that for each $(0,b) \in \{0\} \times D^k$, $\overline{X}_{(0,b)} \cap ((-\infty,\epsilon) \times (0,1)^{n-1}) = (-\infty,\epsilon) \times (\overline{Y}_0)_{(0,b)}$ and  $\overline{X}_{(0,b)} \cap ((\tau(b)-\epsilon,\infty) \times (0,1)^{n-1}) = (\tau(b)-\epsilon,\infty)\times (\overline{Y}_1)_{(0,b)}$.

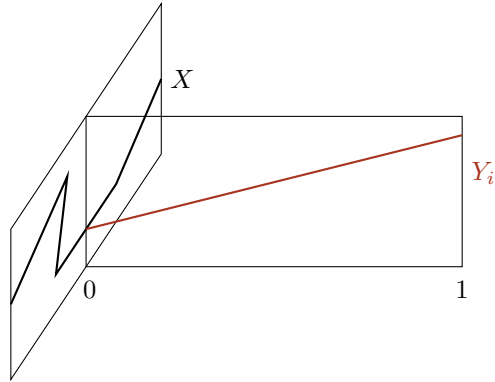
\begin{figure}
	\begin{tikzpicture}
	\draw (0,0) -- (0,2) -- (2,5) -- (2,3) -- cycle;
	\draw [xshift=1cm,yshift=1.5cm] (0,0) -- (5,0) -- (5,2) -- (0,2) -- cycle;
	\draw [thick] (0,1) -- (.75,2.7) -- (.6,1.4) -- (1.4,2.6) -- (2,4);
	\draw [thick, Mahogany] (1,2) -- (6,3.25);
	\node at (2,4) [right] {$X$};
	\node at (6,2.75) [right,Mahogany] {$Y_i$};
	\node at (6,1.2) {$1$};
	\node at (1.05,1.2) {$0$};
	\end{tikzpicture}
	\caption{We need to extend $X$ rightwards compatibly with $Y_i$. This will be done using isotopy extension.}
	\label{fig:kanfibrationinput}
\end{figure}

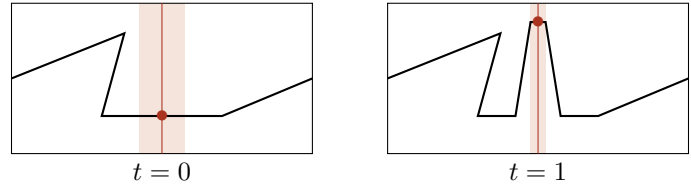
\begin{figure}
	\begin{tikzpicture}
		\begin{scope}
		\fill [Mahogany!10!white] (1.7,0) rectangle (2.3,2);
		\draw (0,0) -- (4,0) -- (4,2) -- (0,2) -- cycle;
		\draw [thick] (0,1) -- (1.5,1.6) -- (1.2,.5) -- (2.8,.5) -- (4,1);
		\node [Mahogany] at (2,.5) {$\bullet$};
		\draw [Mahogany] (2,0) -- (2,2);
		\node at (2,0) [below] {$t=0$};
		\end{scope}
		
		\begin{scope}[xshift=5cm]
		\fill [Mahogany!10!white] (1.9,0) rectangle (2.1,2);
		\draw (0,0) -- (4,0) -- (4,2) -- (0,2) -- cycle;
		\draw [thick] (0,1) -- (1.5,1.6) -- (1.2,.5) -- (1.7,.5) -- (1.9,1.75) -- (2.1,1.75) -- (2.3,.5) -- (2.8,.5) -- (4,1);
		\node [Mahogany] at (2,1.75) {$\bullet$};
		\draw [Mahogany] (2,0) -- (2,2);
		\node at (2,0) [below] {$t=1$};
		\end{scope}
	\end{tikzpicture}
	\caption{The starting and end point of the family obtained by applying the argument of Lemma \ref{lem.sourcetargetfib} to the Figure \ref{fig:kanfibrationinput}.}
	\label{fig:kanfibrationresult}	
\end{figure}

We claim that without loss of generality $\tau$ is identically $1$. To do so, pick a continuous family of homeomorphisms $\rho_t \colon \bR \to \bR$ for $t \in (0,\infty)$ with the following properties: \begin{enumerate}[\indent (i)]
	\item $\rho_t$ sends $(0,t)$ onto $(0,1)$, 
	\item $\rho_t$ is the identity near $0$, and 
	\item $\rho_t$ is a translation near $t$.
\end{enumerate} Now replace $f$ by $d \mapsto (\rho^{-1}_{\tau(d)} \times \mr{id}_{(0,1)^{n-1}})^* f(d)$.

Having made this simplifying assumption, we explain how to lift the relative submersions. A closed topological submersion with compact fibres is a fibre bundle (Corollary \ref{cor.unionfibrebundle}) and fibre bundles over contractible bases are trivialisable. Hence, for $i=0,1$, there is a $(d-1)$-dimensional closed topological submanifold $Z_i \subset (0,1)^{n-1}$ such that $Y_i$ is the graph of a family of topological embeddings $e_i \colon [0,1] \times D^k \to \mr{Emb}^\mr{Top}(Z_i,(0,1)^{n-1})$. By the isotopy extension theorem (Theorem \ref{thm.isotopyextensionmfd}) there is a family of homeomorphisms $\eta_i \colon [0,1] \times D^k \to \mr{Top}((0,1)^{n-1})$ such that $e_i$ is obtained by restricting each $\eta_i(t,d) \colon (0,1)^{n-1} \to (0,1)^{n-1}$ to $Z_i \subset (0,1)^{n-1}$. Let $\epsilon > 0$ be as above and let $\lambda \colon [0,1] \to [0,1]$ be a continuous function that is $1$ near $0$ and $0$ near $1$. Then we may define the submersion part of the lift as follows: the fibre over $(t,d) \in [0,1] \times D^d$ is
	\[X_{(t,d)} = \begin{cases} X_{0,d} & \text{in $[\epsilon,1-\epsilon] \times (0,1)^{n-1}$,} \\
	\eta_0(t \lambda(|s|/\epsilon),d)(Z_0) & \text{in $\{s\}\times (0,1)^{n-1}$, $s\in (-\epsilon,\epsilon]$,} \\
	 \eta_1(t \lambda(|1-s|/\epsilon),d)(Z_1) & \text{in $\{s\} \times (0,1)^{n-1}$, $s \in [1-\epsilon,1 +\epsilon)$.} \end{cases}\]
We finally explain how to add the $\xi$-structures. The $\xi$-structures on $X$ and the $Y_i$'s give a $\xi$-structure on the lift which is defined on all of the fibres over $\{0\} \times D^d$, and over $[0,1] \times D^d$ only on the part of the fibres inside $(-\epsilon',\epsilon') \times (0,1)^{n-1} \cup (1-\epsilon',1+\epsilon') \times (0,1)^{n-1}$ for some $\epsilon>\epsilon'>0$. We can then use the microbundle homotopy covering theorem to extend it (Theorem \ref{thm.microbundlecovering}).
\end{proof}

To investigate the dependence on $n$, we shall use the following definition.

\begin{definition}Given $t_0 < t_1$ and $\overline{Y}_0,\overline{Y}_1 \in \psi^{\mr{Top},\xi}(d-1,n-1,0)$, we define 
	\[\psi^{\mr{Top},\xi}(d,n,1)((t_0,\overline{Y}_0),(t_1,\overline{Y}_1))\]
	to be the colimit as $\epsilon \to 0$ of the subsimplicial sets of $\Psi^{\mr{Top},\xi}_d((t_0-\epsilon,t_1+\epsilon) \times \bR^{n-1})$ consisting of $\overline{X} = (X,\varphi)$ such that $\overline{X} \cap ((t_0-\epsilon,t_0+\epsilon) \times (0,1)^{n-1}) \cong (t_0-\epsilon,t_0+\epsilon) \times \overline{Y}_0$ and $\overline{X} \cap ((t_1-\epsilon,t_1+\epsilon) \times (0,1)^{n-1}) \cong (t_1-\epsilon,t_1+\epsilon) \times \overline{Y}_1$.\end{definition}

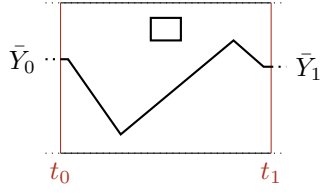
\begin{figure}
	\centering
	\begin{tikzpicture}
	\begin{scope}
	\draw [dotted] (-2,0) -- (1.2,0);
	\draw [dotted] (-2,2) -- (1.2,2);	
	\draw (-1.8,0) -- (1,0);
	\draw (-1.8,2) -- (1,2);
	\draw [Mahogany] (-1.8,0) -- (-1.8,2);
	\draw [Mahogany] (1,0) -- (1,2);
	
	\node at (-1.8,0) [below,Mahogany] {$t_0$};
	\node at (1,0) [below,Mahogany] {$t_1$};
	
	
	\draw [thick,dotted] (-2,1.25) -- (-1.8,1.25);
	\draw [thick,dotted] (1,1.15) -- (1.2,1.15);
	\draw [thick] (-1.8,1.25) -- (-1.7,1.25) -- (-1,0.25) -- (.5,1.5) -- (.9,1.15) -- (1,1.15);
	\draw [thick] (-.6,1.5) rectangle (-.2,1.8);
	\node at (-2,1.25) [left] {$\overline{Y}_0$};
	\node at (1.2,1.15) [right] {$\overline{Y}_1$};
	\end{scope}
	\end{tikzpicture}
	\caption{A $0$-simplex in $\psi^{\mr{Top},\xi}(1,2,1)((t_0,\overline{Y}_0),(t_1,\overline{Y}_1))$.}
	\label{fig:relpoint}
\end{figure}

As in Lemma \ref{lem.stableobjectsbtop}, as $n \to \infty$ the simplicial set $\psi^{\mr{Top},\xi}(d,n,1)((t_0,\overline{Y}_0),(t_1,\overline{Y}_1))$ approximates the disjoint union of the classifying spaces $B\mr{Top}^\xi_\partial(W)$ over all homeomorphism types of $d$-dimensional manifolds $W$ with boundary identified with $Y_0 \sqcup Y_1$.

\begin{lemma}Suppose $d \geq 5$ and $n-d \geq 3$, then the maps $B\cat{Cob}^{\mr{Top},\xi}(d,n) \to B\cat{Cob}^{\mr{Top},\xi}(d,n+1)$ are $(n-2d-1)$-connected.\end{lemma}

\begin{proof}
	By \cite[Lemma 2.4]{rwebertsemi} it suffices to show that the map on $p$-simplices of the nerves is $(n-2d-1)$-connected for all $p \geq 0$. As in the proof of Lemma \ref{lem.sourcetargetfib}, without loss of generality all $(0,\infty)$-parameters  are equal to $1$. By Lemma \ref{lem.sourcetargetfib}, there is a map of fibre sequences based at a sequence of objects $(\overline{Y}_0,\ldots,\overline{Y}_p)$ for $0 \leq i \leq p$,
	\[\begin{tikzcd} \prod_{i=0}^{p-1} \psi^{\mr{Top},\xi}(d,n,1)((0,\overline{Y}_i),(1,\overline{Y}_{i+1})) \rar \dar & \prod_{i=0}^{p-1} \psi^{\mr{Top},\xi}(d,n+1,1)((0,\overline{Y}_i),(1,\overline{Y}_{i+1})) \dar \\
	N_p \cat{Cob}^{\mr{Top},\xi}(d,n) \rar \dar & N_p \cat{Cob}^{\mr{Top},\xi}(d,n+1) \dar \\
	\psi^{\mr{Top},\xi}(d-1,n-1,0)^{p+1}\rar &  \psi^{\mr{Top},\xi}(d-1,n,0)^{p+1}.\end{tikzcd} \]
	It suffices to prove that the restriction of the maps to each of the path components of the bases and fibres is $(n-2d-1)$-connected, because as soon as $2d+2 < n$ the map on the base spaces is a bijection on path components.
	
	The map on bases is a $(p+1)$-fold product of the map $\psi^{\mr{Top},\xi}(d-1,n-1,0) \to \psi^{\mr{Top},\xi}(d-1,n,0)$ induced by the inclusion $(0,1)^{n-1} \hookrightarrow (0,1)^n$. This is $(n-2d)$-connected by Lemma \ref{lem.psi-connectivity}. The top-left corner is a product of spaces given by a disjoint union over all homeomorphism classes of compact $d$-dimensional topological manifolds $W$ with identifications $\partial W = Y_i \sqcup Y_{i+1}$ of the spaces
	\[\mr{Emb}^\mr{Top}((W,\partial W),([0,1] \times (0,1)^{n-1},\{0,1\} \times (0,1)^{n-1})) \times_{\mr{Top}_\partial(W)} \mr{Bun}_\partial(TW,\xi),\]
	where the embeddings are restricted to be the identity on $Y_i$ and $Y_{i+1}$ on $\partial W$. A similar statement holds for the top right corner, $n-1$ replaced by $n$. The map is induced by the inclusion $(0,1)^{n-1} \hookrightarrow (0,1)^n$ and this is $(n-2d-1)$-connected by the same argument as in Lemma \ref{lem.psi-connectivity}, using the version of Theorem \ref{thm.catwhitney} with boundary.
\end{proof}

\subsection{Comparison to long manifolds} The nerve of $\cat{Cob}^{\mr{Top},\xi}(d,n)$ is weakly equivalent to the following semi-Segal object consisting of long manifolds.

\begin{definition}Let $\psi_\cat{Cob}^{\mr{Top},\xi}(d,n,1)_\bullet$ be the semisimplicial simplicial set with simplicial set of $p$-simplices given by the subsimplicial set of $\mr{Sing}(\bR)^{p+1} \times \psi^{\mr{Top},\xi}(d,n,1)$ consisting of 
	\begin{itemize}
		\item a $(p+1)$-tuple of real valued functions $t_0 < \ldots < t_p$, and
		\item an $\overline{X}$ such that there exists an $\epsilon >0$ so that $\overline{X} \cap ((t_i-\epsilon,t_i+\epsilon) \times (0,1)^{n-1}) = (t_i-\epsilon,t_i+\epsilon)  \times (\overline{X} \cap (\{t_i\} \times (0,1)^{n-1}))$ for $0 \leq i \leq p$.
	\end{itemize} \end{definition}

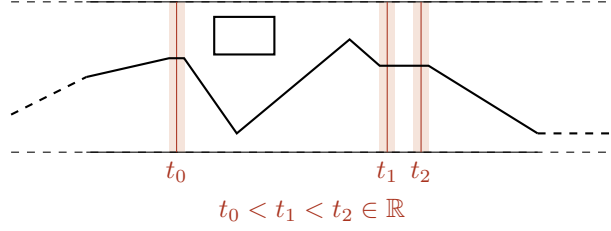
\begin{figure} 
	\centering
	\begin{tikzpicture}
	\fill [Mahogany!10!white] (-1.9,0) rectangle (-1.7,2);
	\fill [Mahogany!10!white] (0.9,0) rectangle (1.1,2);
	\fill [Mahogany!10!white] (1.35,0) rectangle (1.55,2);
	\draw [dashed] (-4,0) -- (4,0);
	\draw [dashed] (-4,2) -- (4,2);	
	\draw (-3,0) -- (3,0);
	\draw (-3,2) -- (3,2);
	\draw [Mahogany] (-1.8,0) -- (-1.8,2);
	\draw [Mahogany] (1,0) -- (1,2);
	\draw [Mahogany] (1.45,0) -- (1.45,2);
	
	\node at (-1.8,0) [below,Mahogany] {$t_0$};
	\node at (1,0) [below,Mahogany] {$t_1$};
	\node at (1.45,0) [below,Mahogany] {$t_2$};
	
	\node at (0,-.5) [below,Mahogany] {$t_0<t_1<t_2 \in \bR$};
	
	\draw [thick,dashed] (-4,.5) -- (-3,1);
	\draw [thick,dashed] (3,.25) -- (4,.25);
	\draw [thick] (-3,1) -- (-1.9,1.25) -- (-1.7,1.25) -- (-1,0.25) -- (.5,1.5) -- (.9,1.15) -- (1.55,1.15) -- (3,.25);
	\draw [thick] (-1.3,1.3) rectangle (-.5,1.8);
	
	\end{tikzpicture}
	\caption{A $0$-simplex in $\psi^\mr{Top,\xi}_\cat{Cob}(1,2,1)_2$, the red strips denoting the translationally constant regions.}
	\label{fig:longtop}
\end{figure}

The $t_i$'s as in the previous definition will be called \emph{slices for $\overline{X}$}; they demarcate sections of the long manifold $\overline{X}$ which are locally translational products, and are allowed to vary when moving over the base. Let us spell this out: a $k$-simplex of the simplicial set of $p$-simplices consists of a topological submanifold $X \subset \Delta^k \times \bR \times (0,1)^{n-1}$ such that $(\Delta^k \times \bR \times (0,1)^{n-1},X) \to \Delta^k$ is a relative topological submersion of relative dimension $(n,d)$. This comes with $\xi$-structure on the vertical tangent microbundle of $X$, and $p+1$ continuous functions $t_i \colon \Delta^k \to \bR$ such that for each $d \in \Delta^k$ the manifold $X_d$ as well as its $\xi$-structure are a product near $\{t_i(d)\} \times (0,1)^{n-1}$. See Figure \ref{fig:longtop} for an example when $k=0$.

There is a semisimplicial map 
\[r_\bullet \colon \psi_\cat{Cob}^{\mr{Top},\xi}(d,n,1)_\bullet \lra N_\bullet \cat{Cob}^{\mr{Top},\xi}(d,n)\]
sending $((t_0<\ldots<t_p),\overline{X})$ to the sequence of composable morphisms given by $(t_{i+1}-t_i,\overline{X}_i)$, where $\overline{X}_i$ is obtained from $\overline{X} \cap [t_i,t_{i+1}] \times (0,1)^{n-1}$ by translating it by $-t_i \cdot e_1$ and then extending it semi-infinitely by translation out to $-\infty$ and $\infty$.

\begin{lemma}\label{lem.longtocob} The map $||r_\bullet|| \colon ||\psi_\cat{Cob}^{\mr{Top},\xi}(d,n,1)_\bullet|| \to B\cat{Cob}^{\mr{Top},\xi}(d,n)$ is a weak equivalence.\end{lemma}

\begin{proof}By \cite[Theorem 2.2]{rwebertsemi} it suffices to prove that $r_\bullet$ is a levelwise weak equivalence. We use that $N_p \cat{Cob}^{\mr{Top},\xi}(d,n)$ may be identified with the subspace of $\psi_\cat{Cob}^{\mr{Top},\xi}(d,n,1)_p$ of those $((t_0,\ldots,t_p),(\overline{Y}_0,\ldots,\overline{Y}_p),\overline{X})$ such that $t_0 = 0$, $\overline{X} \cap ((-\infty,t_0) \times (0,1)^{n-1}) = (-\infty,t_0) \times \overline{Y}_0$, and $\overline{X} \cap ((t_k,\infty) \times (0,1)^{n-1}) = (t_k,\infty) \times \overline{Y}_k$.
	
	We claim that the inclusion $i_p$ of $N_p \cat{Cob}^{\mr{Top},\xi}(d,n)$ as a subspace of $\psi_\cat{Cob}^{\mr{Top}, \xi}(d,n,1)_p$ is a weak homotopy inverse to $r_p \colon \psi_\cat{Cob}^{\mr{Top},\xi}(d,n,1)_p \to N_p \cat{Cob}^{\mr{Top},\xi}(d,n)$. Clearly $r_p \circ i_p$ is the identity, so it suffices to show that $i_p$ is a weak equivalence. That is, we shall show that any commutative diagram
	\[\begin{tikzcd} \partial D^i \rar \dar & N_p \cat{Cob}^{\mr{Top},\xi}(d,n) \dar{i_p} \\[-2pt]
	D^i \rar & \psi_\cat{Cob}^{\mr{Top},\xi}(d,n,1)_p\end{tikzcd}\]
	can be homotoped through commutative diagrams to one where a lift exists. In other words, we are given a closed topological submanifold $X \subset D^i \times \bR \times (0,1)^{n-1}$ such that $(X,D^i \times \bR \times (0,1)^{n-1}) \to D^i$ is a relative topological submersion of relative dimension $(n,d)$ with a $\xi$-structure on the vertical tangent microbundle of $X$. Additionally, we have slices $t_i \colon D^i \to \mathbb{R}$ for $i=0,\ldots,p$ and, as in the proof of Lemma \ref{lem.sourcetargetfib} we can assume that the slice $t_0$ is constantly equal to 0. For $d \in \partial D^i$, the fibre $X_d$ is constant in $(-\infty,\epsilon) \times (0,1)^{n-1}$ and $(t_p(d)-\epsilon,\infty) \times (0,1)^{n-1}$, and we need to deform $X$ so that this is true for all $d \in D^i$.
	
	Pick a family of topological embeddings $\eta \colon D^i \times [0,1] \to \mr{Emb}^\mr{Top}(\bR,\bR)$ such that 
	\begin{enumerate}[(i)]
		\item $\eta(d,0)$ is the identity for all $d \in D^i$, 
		\item $\eta(d,t)$ is the identity on $(-\epsilon/2,t_p(d)+\epsilon/2)$, \item $\eta(d,1)$ has image in $(-\epsilon,t_p(d)+\epsilon)$.
	\end{enumerate} 
	Then we may produce a relative submersion over $D^i \times [0,1]$ with $\xi$-structure on the vertical tangent microbundle by taking as fibres
	\[(d,t) \longmapsto (\eta(t,d) \times \mr{id}_{(0,1)^{n-1}})^* \overline{X}_d.\] Slices for this fibre are provided by those for $(d,0)$. This gives the desired homotopy, the $\xi$-structures simply moved along during all the deformations.\end{proof}

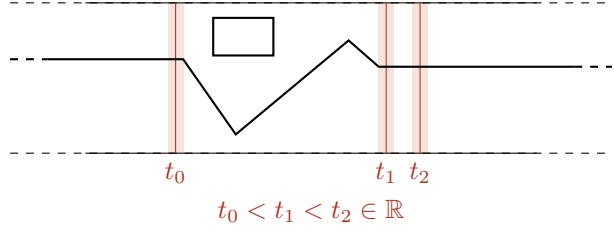
\begin{figure}
	\centering
	\begin{tikzpicture}
	\fill [Mahogany!10!white] (-1.9,0) rectangle (-1.7,2);
	\fill [Mahogany!10!white] (0.9,0) rectangle (1.1,2);
	\fill [Mahogany!10!white] (1.35,0) rectangle (1.55,2);
	\draw [dashed] (-4,0) -- (4,0);
	\draw [dashed] (-4,2) -- (4,2);	
	\draw (-3,0) -- (3,0);
	\draw (-3,2) -- (3,2);
	\draw [Mahogany] (-1.8,0) -- (-1.8,2);
	\draw [Mahogany] (1,0) -- (1,2);
	\draw [Mahogany] (1.45,0) -- (1.45,2);
	
	\node at (-1.8,0) [below,Mahogany] {$t_0$};
	\node at (1,0) [below,Mahogany] {$t_1$};
	\node at (1.45,0) [below,Mahogany] {$t_2$};
	
	\node at (0,-.5) [below,Mahogany] {$t_0<t_1<t_2 \in \bR$};
	
	\draw [thick,dashed] (-4,1.25) -- (-3.5,1.25);
	\draw [thick,dashed] (3.5,1.15) -- (4,1.15);
	\draw [thick] (-3.5,1.25) -- (-1.7,1.25) -- (-1,0.25) -- (.5,1.5) -- (.9,1.15) -- (3.5,1.15);
	\draw [thick] (-1.3,1.3) rectangle (-.5,1.8);
	
	\end{tikzpicture}
	\caption{The $0$-simplex of Figure \ref{fig:longtop}, after deforming it to lie in the image of $i_p$.}
	\label{fig:longtopdeformed}
\end{figure}

\subsection{Creating enduring transversality} In this section we prove two technical lemma's about deforming topological submanifolds with $\xi$-structure.

Let us fix an open subset $U \subset \bR^n$ and a closed topological submanifold $M \subset U$ with normal microbundle $\nu$. For a closed topological submanifold $X \subset B \times U$ and $b \in B$, let us recall the notation $X_b \coloneqq X \cap (\{b\} \times U)$ for the fibre over $b$. Let $\pitchfork$ denote microbundle transversality as in Definition \ref{def.microbundletransversality} (which suppresses the normal bundle data from the notation). The reader should keep in mind the example $U = \bR \times (0,1)^{n-1}$, $M = \{a\} \times (0,1)^{n-1}$ for $a \in \bR$, and $B$ some small ball in $D^i$.

In the statement and proof of the following lemma, we shall use the notation $\cO(-)$ for an open neighbourhood of the given subspace, which is produced during the proof and while giving the proof occasionally redefined. We shall also use $\cS(-)$ to denote a ``support'' open subset which one can specify beforehand.

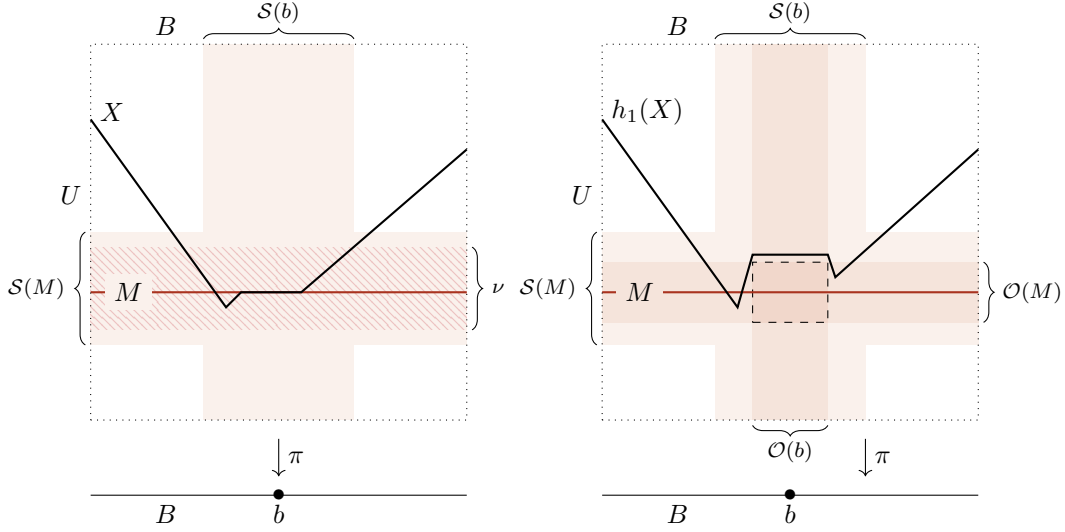
\begin{figure}
	\begin{tikzpicture}
	\begin{scope}
	\fill [Mahogany!5!white] (0,1) rectangle (5,2.5);
	\fill [Mahogany!5!white] (1.5,0) rectangle (3.5,5);
	\fill [pattern=north west lines, pattern color=Mahogany!20!white] (0,1.2) rectangle (5,2.3);
	\draw [decorate,decoration={brace,amplitude=4pt},xshift=-2pt,yshift=0pt]
	(0,1) -- (0,2.5) node [black,midway,xshift=-.65cm] {\footnotesize $\cS(M)$};
	\draw [decorate,decoration={brace,amplitude=4pt},xshift=2pt,yshift=0pt]
	(5,2.3) -- (5,1.2) node [black,midway,xshift=.35cm] {\footnotesize $\nu$};
	\draw [decorate,decoration={brace,amplitude=4pt},yshift=2pt,xshift=0pt]
(1.5,5) -- (3.5,5) node [black,midway,yshift=.35cm] {\footnotesize $\cS(b)$};
	\draw [thick,Mahogany] (0,1.7) -- (5,1.7);
	\node [fill=Mahogany!5!white] at (0.5,1.7) {$M$};
	\draw [dotted] (0,0) rectangle (5,5);
	\draw (0,-1) -- (5,-1);
	\draw [->] (2.5,-.25) -- (2.5,-.75);
	\node at (2.5,-.5) [right] {$\pi$};
	\node at (1,-1) [below] {$B$};
	\node at (2.5,-1) {$\bullet$};
	\node at (2.5,-1) [below] {$b$};
	\node at (1,5) [above] {$B$};
	\node at (0,3) [left] {$U$};
	\draw [thick] (0,4) -- (1.8,1.5) -- (2,1.7) -- (2.8,1.7) -- (5,3.6);
	\node at (0,4.1) [right] {$X$};
	\end{scope}
	
	\begin{scope}[xshift=6.8cm]
	\fill [Mahogany!5!white] (0,1) rectangle (5,2.5);
	\fill [Mahogany!5!white] (1.5,0) rectangle (3.5,5);
	\fill [Mahogany!10!white] (2,0) rectangle (3,5);		
	\fill [Mahogany!10!white] (0,1.3) rectangle (5,2.1);
	\draw [dashed,fill=Mahogany!15!white] (2,1.3) rectangle (3,2.1);
	\draw [decorate,decoration={brace,amplitude=4pt},xshift=-2pt,yshift=0pt]
	(0,1) -- (0,2.5) node [black,midway,xshift=-.65cm] {\footnotesize $\cS(M)$};
		\draw [decorate,decoration={brace,amplitude=4pt},xshift=2pt,yshift=0pt]
	(5,2.1) -- (5,1.3) node [black,midway,xshift=.65cm] {\footnotesize $\cO(M)$};
	\draw [decorate,decoration={brace,amplitude=4pt},yshift=2pt,xshift=0pt]
	(1.5,5) -- (3.5,5) node [black,midway,yshift=.35cm] {\footnotesize $\cS(b)$};
		\draw [decorate,decoration={brace,amplitude=4pt},yshift=-2pt,xshift=0pt]
	(3,0) -- (2,0) node [black,midway,yshift=-.35cm] {\footnotesize $\cO(b)$};
	\draw [thick,Mahogany] (0,1.7) -- (5,1.7);
	\node [fill=Mahogany!10!white] at (0.5,1.7) {$M$};
	\draw [dotted] (0,0) rectangle (5,5);
	\draw (0,-1) -- (5,-1);
	\draw [->] (3.5,-.25) -- (3.5,-.75);
	\node at (3.5,-.5) [right] {$\pi$};
	\node at (1,-1) [below] {$B$};
	\node at (2.5,-1) {$\bullet$};
	\node at (2.5,-1) [below] {$b$};
	\node at (1,5) [above] {$B$};
	\node at (0,3) [left] {$U$};
	\draw [thick] (0,4) -- (1.8,1.5) -- (2,2.2) -- (3,2.2) -- (3.1,1.9) -- (5,3.6);
	\node at (0,4.1) [right] {$h_1(X)$};
	\end{scope}
	\end{tikzpicture}
	\caption{On the right, the input of Lemma \ref{lem.slicedeform}, and on the left its output. Note that when $d=0$, then being transverse to $M$ is the same as being disjoint from it.}
	\label{fig:slicedeform}
\end{figure}

\begin{lemma}\label{lem.slicedeform} Fix a topological manifold $B$ of dimension $p$, and suppose we are given the following data:
	\begin{itemize}
		\item  a closed topological submanifold $X \subset B \times U$ such that $(B \times U,X) \to B$ is a relative topological submersion of relative dimension $(n,d)$,
		\item a submanifold $M \subset U$ with normal microbundle $\nu$,
		\item a point $b_0 \in B$, and
		\item open neighbourhoods $\cS(b_0) \subset B$ and $\cS(M) \subset U$ of $M$.
	\end{itemize}  
Moreover, let us denote the closure of $\cS(M)$ in $U$ by $\mr{cl}_U(\cS(M))$.  Then, if $M$ is compact or $\mr{cl}_U(\cS(M)) \cap X_{b_0}$ is compact, we can find:
	\begin{enumerate}[\indent (a)]
		\item (\emph{transversality over $b_0$}) an ambient isotopy $h_t$ of $U$ (i.e., a homotopy of homeomorphisms $h_t \colon B\times U \rightarrow B\times U$  commuting with the projection onto $B$) supported in $\cS(b_0) \times \cS(M)$ such that $h_1(X)_{b_0} \pitchfork M$ (i.e., the fibre of the map $h_1(X) \hookrightarrow B \times U \to B$ over $b_0$ is microbundle transverse to $M$),
		\item (\emph{locally constant}) open neighbourhoods $\cO(b_0) \subset \cS(b_0)$ and $\cO(M) \subset \cS(M)$ of $b_0$ and $M$, so that $h_1(X)_{b} \cap \cO(M) = h_1(X)_{b_0} \cap \cO(M)$ for all $b \in \cO(b_0)$.
	\end{enumerate}
Necessarily $h_1(X)_{b} \pitchfork M$ for all $b \in \cO(b_0)$ and in fact $h_1(X) \cap (\cO(b_0) \times \cO(M)) \pitchfork (\cO(b_0) \times M)$.
\end{lemma}

\begin{proof}By picking a chart in $B$ which is contained in $\cS(b_0)$, we can assume without loss of generality that $B$ is the closed unit ball $B_1$ in $\mathbb{R}^p$, $\cS(b_0) = \mr{int}(B_1)$, and that $b_0$ is the origin in $\mathbb{R}^p$. Eventually, $\cO(b_0)$ will be a small ball around the origin. Throughout the rest of this proof, even though we are identifying $b_0$ with the origin in $\mathbb{R}^p$, we will continue to denote this point by $b_0$. 
	
\medskip

\noindent \emph{Step 1: arranging (a).} Let us first apply microbundle transversality (Theorem \ref{thm.microbundletransversality}) with $X = X_{b_0} \subset U$, $N = U$, $C = \varnothing$, and $D$ some closed neighbourhood of $M$ in $U$ contained in $\cS(M)$. The result is an ambient isotopy $k_t \colon U \rightarrow U$ of $U$ supported in $\cS(M)$ such that $k_1(X_{b_0}) \pitchfork M$.

Pick a continuous function $\lambda \colon [0,1] \to [0,1]$ which is $0$ on $[1/2,1]$ and $1$ on $[0,1/4]$. We now isotope $X \subset B_1 \times U$ (while preserving the map to $B_1$) via an ambient isotopy $\tilde{k}_t \colon B_1 \times U \rightarrow B_1 \times U$ defined as follows:
\[\tilde{k}_t(b,x) = \begin{cases} (b,x) & \text{if $|b|>1/2$,} \\
(b,k_{t \lambda(|b|)}(x)) & \text{if $|b| \leq 1/2$}.\end{cases}\]
For $t=0$, the composition $\tilde{k}_0(X) \hookrightarrow B_1\times U \rightarrow B_1$ is the original topological submersion. For $t=1$, the fibre of the composition $\tilde{k}_1(X) \hookrightarrow B_1\times U \rightarrow B_1$ over $b_0$ is equal to $k_1(X_{b_0})$. In particular, this fibre is microbundle transverse to $M$ and we have thus arranged (a), \emph{transversality over $b_0$}. Since we only modify $X_{b_0}$ within compact subset of $\cS(M)$, this preserves the property that $\mr{cl}_U(\cS(M)) \cap X_{b_0}$ is compact.

\medskip

\noindent \emph{Step 2: arranging (b).} For simplicity, let us relabel the outcome $\tilde{k}_1(X)$ of step 1 to $X$. We next arrange (b) (i.e.\ make the fibres of $X \hookrightarrow B_1\times U \rightarrow B_1$ \emph{locally constant} near $b_0$) under the assumption that $X_{b_0} \pitchfork M$. By this transversality assumption and the hypotheses, $X_{b_0} \cap M$ is a compact submanifold of $X_{b_0}$. We may take a normal microbundle of $X_{b_0} \cap M$ in $X_{b_0}$ whose fibres coincide with those of $\nu$ and whose total space is a compact topological submanifold with boundary $C \subset X_{b_0} \cap \cS(M)$ (necessarily containing $X_{b_0} \cap M$). 

By the union lemma for submersion charts (Theorem \ref{thm.unionsubmersion}) there is an $r \in (0,1/2)$, an open neighbourhood $\cO(C) \subset X_{b_0}$, and a map $g \colon B_r \times \cO(C) \to X$ over the projection to $B_1$ which is an isomorphism onto its image. Moreover, we can assume that the restriction of $g$ on the fibre $\{b_0\}\times \cO(C)$ agrees with the natural inclusion $\cO(C) \hookrightarrow X_{b_0} \hookrightarrow X$. Restricting the composition 
$B_r \times \cO(C)  \xrightarrow{g}X \hookrightarrow B_1 \times U$  to $B_r \times C$ gives us an isotopy $e \colon B_r\times C \rightarrow B_r\times U$ of $C$ in $U$ parametrised by $B_r$. Note that this isotopy $e$ has the property that its restriction to $\{b_0\}\times C$ is equal to the natural inclusion 
$C \hookrightarrow X_{b_0} \hookrightarrow X$ after identifying $C$ with $\{b_0\}\times C$. Consequently, we can find a path of isotopies $\tilde{e}_t \colon B_r\times C \rightarrow B_r\times U$ of $C$ in $U$ over $B_r$ such that $\tilde{e}_1 = e$ and $\tilde{e}_0$ is the standard inclusion 
$B_r \times C \hookrightarrow B_r \times U$. By shrinking $r$ if necessary, we may assume there is a closed subspace $C(M) \subset \cS(M)$ such that the image of each isotopy $\tilde{e}_t$ 
is contained in $B_r \times \mathrm{int}(C(M))$. By the isotopy extension theorem (Theorem \ref{thm.isotopyextensionmfd}), there hence exists a path of ambient isotopies $f_t \colon B_r \times U \rightarrow B_r \times U$ of $U$ over $B_r$ which is supported in $\cS(M)$ 
 and such that $f_t$ extends $\tilde{e}_t$ for each $t$ in $[0,1]$. Moreover, since $\tilde{e}_0$ was the standard inclusion $B_r \times C \hookrightarrow B_r \times U$, we may assume that $f_0$ is the identity homeomorphism from $B_r \times U$ to itself. 
 
Pick a continuous function $\eta \colon [0,1] \to [0,1]$ which is $1$ near $0$ and $0$ on $[r/2,r]$, and consider the ambient isotopy 
$h_t: B_1 \times U \rightarrow B_1\times U$ of $U$ over $B_1$ given by the formula
\[h_t(b,x) = \begin{cases} (b,x) & \text{if $|b|>r$,} \\
f_{\eta(|b|)t}^{-1}(b,x) & \text{if $|b| \leq r$}.\end{cases}\]
For $t=0$, we have that $h_0(X) = X$. On the other hand,  for $t=1$, the intersection $X_{b_0} \cap M$ (which is a subspace of $U$) is contained in the fibre $h_1(X)_b$ if $|b|$ is sufficiently small. Also, $h_1(X)_b$ coincides with $C$ near $X_{b_0} \cap M$ when $b$ is sufficiently close to $b_0$. There may be other intersections of $h_1(X)_b$ with $M$. But for $b$ even closer to $b_0$ this cannot occur: by the assumption that $M$ or $\mr{cl}_U(\cS(M)) \cap X$ is compact, there is an open neighbourhood $U'$ of $M$ such that $X_{b_0} \cap U' = C \cap U'$, which implies that $h_1(X)_b$ is disjoint from $M$ except at $X_{b_0} \cap M$ for $b \in B_1$ with Euclidean norm $|b|$ sufficiently small.\end{proof}

If $X$ came with a $\xi$-structure on its vertical tangent microbundle, then so does the end result $h_1(X)$ of the previous lemma; one can move the $\xi$-structure along the isotopy using the microbundle homotopy covering theorem, here Theorem \ref{thm.microbundlecovering}. However, this $\xi$-structure need \emph{not} be independent of $b' \in \cO(b)$ near $h_1(X)_{b'} \cap \cO(M)$. To ameliorate this, apply to the output the following lemma: take $B$ the output $\cO(b)$, $U$ the output $\cO(M)$, and $\cS(b)$ an open neighbourhood of $b$ such that $\mr{cl}(\cS(b)) \subset \cO(b)$.

\begin{lemma}\label{lem.slicedeformtangent} Suppose we are given the following data:
	\begin{itemize}
		\item a closed topological submanifold $X_0 \subset U$ (so that $(B \times U,B \times X_0) \to B$ is trivially a relative topological submersion of relative dimension $(n,d)$ and we write $X \coloneqq B \times X_0$),
		\item a $\xi$-structure $\varphi_X \colon T_\pi X \to \xi$ on the vertical tangent microbundle of $X$,
		\item a closed subset $K \subset X_0$,
		\item a point $b \in B$, and 
		\item open neighbourhoods $\cS(b) \subset B$ and $\cS(K) \subset X_0$.
	\end{itemize}
Then we can find:
\begin{enumerate}[\indent (a)]
	\item a homotopy $h_t$ of $\xi$-structures supported in $\cS(b) \times \cS(K)$,
	\item open subsets $\cO(b) \subset \cS(b)$ and $\cO(K) \subset \cS(K)$ such that $h_1|_{X_{b'} \cap \cO(K)} = h_1|_{X_{b} \cap \cO(K)}$ for all $b' \in \cO(b)$. 
\end{enumerate}
\end{lemma}

\begin{proof}As in the previous lemma, without loss of generality we can assume that $B = B_1$ and $b$ is the origin. $\cO(b)$ will eventually be a small ball around the origin. By normality, there exist open subsets $W_0$ and $W_1$ of $X_0$ such that $K \subset W_0$ and $X_0 \setminus \cS(K) \subset W_1$. Then we may construct a $\xi$-structure on the subset of $[0,1] \times X$ given by $\{0\} \times X \cup [0,1] \times B_1 \times (W_0 \cup W_1) $ as follows: on $\{0\} \times X$ it is $\varphi_X$, on $[0,1] \times B_1 \times W_1$ it is $\varphi_X \circ p_{\hat{1}}$ (where $p_{\hat{1}}$ is the projection which forgets the first factor), but on $[0,1] \times B_1 \times W_0$ it is given at $(t,b,x)$ by $\varphi_X$ at $(\lambda_t(b),x)$, where $\lambda_t$ is a family of maps $B_1 \to B_1$ such that (i) it is the identity outside $B_{1/2}$, (ii) is the identity for $t=0$, and (iii) for $t=1$ it collapses a neighbourhood of the origin onto the origin. We may extend this $\xi$-structure to all of $[0,1] \times X$ using the microbundle homotopy covering theorem (\cref{thm.microbundlecovering}).
\end{proof}

\subsection{Comparing long manifolds to the cobordism category} 

There is a zigzag
\[\psi^{\mr{Top},\xi}(d,n,1)  \overset{\epsilon}{\longleftarrow} ||\psi^\mr{Top,\xi}_\cat{Cob}(d,n,1)_\bullet|| \overset{\simeq}{\longrightarrow} B\cat{Cob}^\mr{Top,\xi}(d,n),\]
where the left map is induced by the augmentation $\epsilon_\bullet \colon \psi^\mr{Top,\xi}_\cat{Cob}(d,n,1)_\bullet \to \psi^\mr{Top,\xi}(d,n,1)$ which forgets all but $\overline{X}$. We showed that the right map is a weak equivalence in Lemma \ref{lem.longtocob}. 

In this section we will prove that the left map is also a weak equivalence, using the results of the previous subsection. This involves an application of microbundle transversality, and thus requires the existence of some normal bundles. In our case, one of the two submanifolds will be of the form $\{t_0\} \times (0,1)^{n-1} \subset \bR \times (0,1)^{n-1}$, which has a standard normal microbundle $\nu \coloneqq (\{t_0\} \times (0,1)^{n-1}, \bR \times (0,1)^{n-1},\mr{id},\pi)$ with $\pi(t,x) \coloneqq (t_0,x)$.

\begin{proposition}\label{prop.longworestrictionstocob} The map $\epsilon$ is weak equivalence.\end{proposition}

\begin{proof}Suppose we are given a commutative diagram
	\[\begin{tikzcd} \partial D^i \dar \rar & {||\psi^\mr{Top,\xi}_{\cat{Cob}}(d,n,1)_\bullet||} \dar{\epsilon} \\[-2pt]
	D^i \rar & \psi^{\mr{Top},\xi}(d,n,1). \end{tikzcd}\]
Then we must homotope it through commutative diagrams to one where there exists a lift. Without loss of generality, there is a lift over an open neighbourhood $U'$ of $\partial D^i$. Thus, we have a closed topological submanifold $X \subset D^i \times \bR \times (0,1)^{n-1}$ such that the map $\pi \colon (D^i \times \bR \times (0,1)^{n-1},X) \to D^i$ is a relative topological submersion of relative dimension $(n,d)$, with a $\xi$-structure on the vertical tangent microbundle of $X$. We additionally have a collection of $(U_k,t_k,s_k)$, where the $U_k$ are open subsets in an open cover $\{U_k\}$ of $U'$, and for each $U_k$ we have a slice $t_k \colon U_k \to \bR$ and a compactly supported continuous weight function $s_k \colon U_k \to [0,1]$ such that for each $b \in s_k^{-1}((0,1])$ the fibre $X_b$ is microbundle transverse to $\{t_k(b)\} \times (0,1)^{n-1}$ at its standard normal bundle. Let $T \subset \bR$ denote the union of the images of the maps $t_k$, which without loss of generality we may assume to be bounded by shrinking $U'$ if necessary.

\vspace{.5em}

\noindent \textbf{Claim.} \emph{We may assume that for all $b \in D^i \setminus U'$ there exists an element $t \in \bR \setminus T$ and an open neighbourhood $W$ of $\{t\} \times (0,1)^{n-1}$ such that $X_b \pitchfork \{t\} \times (0,1)^{n-1}$ and $\overline{X}_b \cap W = \overline{X}_{b'} \cap W$ for $b'$ near $b$.}

\begin{proof}Pick $(i+2)$ pairwise disjoint intervals $\{[a_j-\eta_j,a_j+\eta_j]\}_{0 \leq j \leq i+1}$ in $\bR \setminus T$, which is possible since $T$ is bounded. Pick for each $b \in \mr{int}(D^i)$ an $i$-dimensional embedded ball $B_1(b) \subset \mr{int}(D^i)$ with origin mapping to $b$.
	
	For each $b \in \mr{int}(D^i)$ and $0 \leq j \leq i+1$, apply Lemmas \ref{lem.slicedeform} and \ref{lem.slicedeformtangent} to the family $F$ obtained by restricting $\overline{X}$ to $U = \bR \times (0,1)^{n-1}$, $B = B_1(b)$, letting $\cS(b) = B_{1/2}(b)$, taking $M = \{a_j\} \times (0,1)^{n-1}$ and finally letting $\cS(M) = (a_j-\eta_j,a_j+\eta_j) \times (0,1)^{n-1}$. The conclusion is that we can isotope $X$ and change its $\xi$-structure using a homotopy supported in $\cS(b) \times \cS(M)$ such that there is an open neighbourhood $U_j(b)$ of $b$ over which the end result is constant near $\{a_j\} \times (0,1)^{n-1}$ and transverse to $\{a_j\} \times (0,1)^{n-1}$. Since there are finitely many $j$, we may without loss of generality assume that $U_j(b) = U_{j'}(b)$ and thus drop the subscripts.
	
	Since $D^i$ is an $i$-dimensional manifold, it has Lebesgue covering dimension $i$. This means that there are open subsets $\tilde{U}(b) \subset U(b)$ and $\tilde{U} \subset U'$ which (a) cover $D^i$, and (b) have the property that each $d \in D^i$ is contained in at most $i+1$ of these subsets. Since $D^i$ is compact, we may assume that only finitely many of these are non-empty: $\tilde{U}$ and $\tilde{U}(b_k)$ for $1 \leq k \leq K$. Since $D^i$ is paracompact we may pick a partition of unity subordinate to $\tilde{U}$ and $\tilde{U}(b_k)$. By normalization we may obtain from this $[0,1]$-valued continuous functions $\sigma_{\tilde{U}}$ and $\{\sigma_{\tilde{U}(b_k)}\}_{1 \leq k \leq K}$ supported in $\tilde{U}$ and $\tilde{U}(b_k)$ respectively, such that for all $d \in D^i$ at least one of these functions has value $1$ on a neighbourhood of $d$.
	
	The finite graph with a vertex for each element of this cover and an edge whenever two elements have non-empty intersection, has vertices of valence $\leq i+1$. Hence its chromatic number is $\leq i+2$, and we can assign to each $1 \leq k \leq K$ a number $0 \leq c(k) \leq i+1$ such that $c(k) \neq c(k')$ if $k \neq k'$ and $\tilde{U}(b_k) \cap \tilde{U}(b_{k'}) \neq \varnothing$.

         Now, by our choice of the open sets $\tilde{U}(b_k)$, there exists for each $1 \leq k \leq K$ an isotopy $H_t^k$ of $\overline{X}$ over $\tilde{U}(b_k)$ which is supported in $(a_{c(k)}-\eta_{c(k)},a_{c(k)}+\eta_{c(k)}) \times (0,1)^{n-1}$ and such that the fibres of the end result are constant near $\{ a_{c(k)} \}\times (0,1)^{n-1}$ and transverse to $\{ a_{c(k)} \}\times (0,1)^{n-1}$. From this we can construct an isotopy over the whole disc $D^i$ by taking 
         \[\widetilde{H}^k_t \coloneqq H^k_{t\cdot \sigma_{\tilde{U}(b_k)}}.\] 
         That is, we use the function $\smash{\sigma_{\tilde{U}(b_k)}}$ to cut off $H_t^k$. It is possible to do all of the isotopies $\smash{\widetilde{H}^k_t}$ simultaneously, because whenever there is a non-empty intersection $\tilde{U}(b_k) \cap \tilde{U}(b_{k'}) \neq \varnothing$, these isotopies have disjoint support in the codomain. The result is a new family $\overline{X}$ such that for each $d \in D^i \setminus U$ there is at least one $t \in \bR$ such that near $b$ the fibre $\overline{X}_{b'}$ is independent of $b'$ near $\{t\} \times (0,1)^{n-1}$ and transverse to this submanifold.   
\end{proof}

Given the claim above, we may cover $D^i$ by $U'$ and open subsets $V_\ell$ for $1 \leq \ell \leq L$ for which we can find slices $t_\ell \colon V_\ell \to \bR \setminus T$. Pick a subordinate partition of unity given by $\eta$ supported in $U'$ and $\eta_\ell$ supported in $V_\ell$, and replace $(U_k,t_k,s_k)$ by $(U_k,t_k,\eta s_k)$ and add $(V_\ell,t_\ell,\eta_\ell)$ for $1 \leq \ell \leq L$.
\end{proof}

\begin{remark}The claim in the above proof also shows that an alternative definition of $\cat{Cob}^\mr{Top,\xi}_d(n)$ with discrete objects has a weakly equivalent classifying space, as the slices $t_\ell$ produced at the very end of the previous proof will be constant functions.\end{remark}

\section{Microflexibility for spaces of topological manifolds} \label{sec.microflexibility} In this section we give a proof of the variant of Theorem \ref{thm.cobinfiniteloop} for topological manifolds. Though we believe there are no difficulties in applying the iterated delooping argument, inspired by \cite{rwembedded} instead we prove microflexibility of a certain sheaf and apply Gromov's $h$-principle.

\subsection{Gromov's $h$-principle} In \cite{gromovhp}, Gromov gave a general theory of $h$-principles on open manifolds. We describe it in this subsection, and later this section apply it to spaces of topological manifolds. The objects studied in this theory are invariant topological or simplicial sheaves on manifolds, such as the examples $\Psi^{\mr{Diff},\upsilon}_{d}(-)$ or $\Psi^{\mr{Top},\xi}_d(-)$. A crucial role in Gromov's theory is played by the notion of a Serre microfibration. This notion makes sense for simplicial sets that extend to quasitopological spaces, such as our $\Psi_d^\mr{Top,\xi}(M)$, which is exactly the setting in which Gromov works, cf.~\cite[Section 1.4.1]{gromovhp}.

\begin{definition}A map $f \colon E \to B$ of topological spaces or quasitopological spaces is said to be a \emph{Serre microfibration} if in each commutative diagram
	\[\begin{tikzcd} D^i \rar \dar & E \dar{f} \\[-2pt]
	{[0,1]} \times D^i \rar & B.\end{tikzcd}\]
there exists an $\epsilon >0$ and a partial lift over $[0,\epsilon] \times D^i$.\end{definition}

\begin{definition}Let $\Phi$ be an invariant simplicial sheaf on topological manifolds, whose values extend to quasitopological spaces.
	\begin{itemize}
		\item $\Phi$ is said to be \emph{microflexible} if for all pairs $L \subset K$ of compact subsets of all topological manifolds $M$ the following restriction map is a Serre microfibration
		\[\Phi(K \subset M) \lra \Phi(L \subset M).\]
		\item $\Phi$ is said to be \emph{flexible} if for all pairs $L \subset K$ of compact subsets of all topological manifolds $M$ the following restriction map is a fibration
		\[\Phi(K \subset M) \lra \Phi(L \subset M).\]
	\end{itemize}
\end{definition}

\begin{remark}A sheaf $\Phi$ is microflexible (resp.\ flexible) with respect to all pairs $(K,L)$ of compact subsets if and only if it is microflexible (resp.\ flexible) with respect to pairs of the form $(H \cup (D^r \times D^{k-r}),H)$ where $H \subset M$ is an embedded finite handlebody, see \cite[Theorem V.A.1]{kirbysiebenmann} or \cite[Lemma 3.3]{kupershp}. Using the invariance of $\Phi$, this may further be reduced to the case $(K,L) = (D^r \times D^{k-r},S^{r-1} \times D^{k-r})$ in $M = \bR^k$.\end{remark}

Gromov's theory of microflexible sheaves used the existence of triangulations or handlebodies. This might cause a problem for topological 4-manifolds, which may not necessarily admit such structures, but in \cite[Appendix V.A]{kirbysiebenmann}, Siebenmann described a proof of Gromov's $h$-principle that avoids them. Up to this caveat, the following result about flexible sheaves is proven in \cite[Section 2.2.1]{gromovhp}.

\begin{proposition}[Gromov] There exists an initial flexible approximation $\Phi \to \Phi^f$ for any microflexible sheaf $\Phi$. Given a $k$-dimensional topological manifold $M$, and a topological $\bR^k$-bundle $\tau_M$ in the tangent microbundle $TM$ (which exists using Theorem \ref{thm.kister}), the value of $\Phi^f$ on $M$ is weakly equivalent to the space of sections of the associated bundle
	\[T\Phi(\bR^k) \coloneqq \tau_M \times_{\mr{Top}(k)} \Phi(\bR^k).\]
Furthermore, given a codimension zero embedding $N \subset M$, the following diagram commutes up to homotopy
\[\begin{tikzcd}\Phi^f(M) \rar \dar[swap]{\simeq} & \Phi^f(N) \dar{\simeq} \\[-4pt]
\Gamma(M,T\Phi(\bR^k)) \rar & \Gamma(N,T\Phi(\bR^k)).\end{tikzcd}\]\end{proposition}

That the above diagram only commutes up to homotopy is because we used a bundle with fibre $\Phi(\bR^k)$ instead of a bundle of germs: we need to choose a map $\tau_M \to M$ which is an embedding on each fibre, and this may need to be rechosen when restricting. 

Up to caveat mentioned above, the following result about microflexible sheaves is proven in \cite[p.~260--261]{kirbysiebenmann} or \cite[Section 2.2.2]{gromovhp}:

\begin{theorem}[Gromov-Siebenmann] \label{thm.hprinciple} If $\Phi$ is a microflexible invariant topological sheaf on topological manifolds, then the map
	\[\Phi(M \rel A) \lra \Phi^f(M \rel A)\]
is a weak equivalence for all $A \subset M$ closed such that no component of $ M \setminus A$ has compact closure.\end{theorem}

\subsection{Applying the $h$-principle} The following observation appears in \cite{rwembedded,dottohprinciple} for the smooth case, where the proof is easier because being an embedding is an open condition. Our proof instead relies on respectful isotopy extension.

\begin{proposition}The invariant topological sheaf $\Psi^{\mr{Top},\xi}_d(-)$ is microflexible.\end{proposition}

\begin{proof}We must find an $\epsilon>0$ and a partial lift over $[0,\epsilon] \times D^i$ in each commutative diagram
\[\begin{tikzcd} D^i \ar[r] \ar[d] & \Psi^{\mr{Top},\xi}_d(D^r \times D^{k-r} \subset \bR^k) \ar[d] \\[-2pt]
{[0,1] \times D^i} \ar[r] & \Psi^{\mr{Top},\xi}_d(S^{r-1} \times D^{k-r} \subset \bR^k).\end{tikzcd}\]

For the benefit of the reader we first prove the case $i=0$; the proof in the case $i>0$ is exactly analogous but involves more notation. In the case $i=0$, the top map is represented by the following data: 
\begin{itemize}
	\item an open subset $\cO(D^r \times D^{k-r}) \subset \bR^k$,
	\item a closed topological submanifold $X \subset \cO(D^r \times D^{k-r})$, and
	\item a $\xi$-structure on its tangent microbundle $TX$.
\end{itemize}
The bottom map is represented by similar data:
\begin{itemize}
	\item open subsets $\cO([0,1]) \subset \bR$ and $\cO(S^{r-1} \times D^{k-r}) \subset \bR^k$,
	\item a closed topological submanifold $X' \subset \cO([0,1]) \times \cO(S^{r-1} \times D^{k-r})$ such that $\pi' \colon (\cO([0,1]) \times \cO(S^{r-1} \times D^{k-r}),X') \to \cO([0,1])$ is a relative topological submersion of relative dimension $(r,d)$,
	\item a $\xi$-structure on its vertical tangent microbundle $T_{\pi'} X'$.
\end{itemize}
We use the abbreviations $\overline{X} = (X,\varphi_X)$ and $\overline{X}' = (X',\varphi_{X'})$.  Since the diagram commutes, without loss of generality we can assume $\cO(S^{r-1} \times D^{k-r}) \subset \cO(D^{r} \times D^{k-r})$, and $\overline{X}'_0 = \overline{X} \cap \cO(S^{r-1} \times D^{k-r})$ (where $\overline{X}'_0$ is the fibre of $\overline{X}'$ over $0 \in \cO([0,1])$ together with its $\xi$-structure, as always). Our goal is to show that we may assume that $X' = \cO([0,1]) \times X'_0$, by applying some homeomorphisms and shrinking some of the open subsets. Given this claim, it is easy to extend $X$ compatibly with $X'$: we can take the extension to be $\cO([0,1])\times X$. By construction, this comes with a $\xi$-structure over $\{0\} \times X \cup X'$. As $[0,1] \times X$ deformation retracts onto $\{0\} \times X \cup X'$ (e.g.~combine Proposition A.6.7 and Theorem A.6.9 of \cite{fritschpiccinini}), we may use the microbundle homotopy covering theorem to extend the $\xi$-structure. This completes the proof in the case $i=0$. 

\begin{figure}
	\begin{tikzpicture}
		\draw [dotted] (0,0) -- (0,2) -- (2,5) -- (2,3) -- cycle;
		\draw [dotted] (0,0) -- (0,2) -- ({0+.5},{2+3/2*.5}) -- ({0+.5},{3/2*.5}) -- cycle;
		\draw [dotted,xshift=5cm] (0,0) -- (0,2) -- ({0+.5},{2+3/2*.5}) -- ({0+.5},{3/2*.5}) -- cycle;
		\draw [dotted,yshift=3cm,xshift=2cm] (0,0) -- (0,2) -- ({0+.5},{2+3/2*.5}) -- ({0+.5},{3/2*.5}) -- cycle;
		\draw [thick,dashed,fill=black!5!white] (0,1) -- (5,1) -- (5.5,1.25) -- (.5,2);
		\draw [thick,dashed,fill=black!5!white] (2,3.8) -- (7,4) -- (7.5,4.25) -- (2.5,4.5);
		\draw [thick,dashed,fill=black!5!white] (4,5.5) -- (5.5,5) -- (7,4.65) -- (7.5,4.8) -- (6.5,5.75) -- cycle;
		\draw [dotted,yshift=3cm,xshift=7cm] (0,0) -- (0,2) -- ({0+.5},{2+3/2*.5}) -- ({0+.5},{3/2*.5}) -- cycle;
		\draw [dotted] (0,0) -- (5,0) -- (5,2) -- (0,2);
		\draw [dotted,xshift=.5cm,yshift=.75cm] (0,0) -- (5,0) -- (5,2) -- (0,2);
		\draw [dotted,yshift=3cm,xshift=2cm] (0,0) -- (5,0) -- (5,2) -- (0,2);
		\draw [dotted,xshift=2.5cm,yshift=3.75cm] (0,0) -- (5,0) -- (5,2) -- (0,2);
		\draw [decorate,decoration={brace,amplitude=4pt},xshift=-1pt,yshift=2pt]
		(0,2) -- (2.5,{5+3/2*.5}) node [black,midway,xshift=-1.05cm,yshift=.1cm] {\footnotesize $\cO(D^1 \times D^0)$};
		\draw [thick] (0,1) -- (1,3) -- (.9,2.4) -- (2.5,4.5);
		\node at (1.35,3.5) {$X$};
		\node at (3,1.3) {$X'$};
		\node at (4,4.15) {$X'$};
		\node at (5,-.25) {$1$};
		\node at (0,-.25) {$0$};
	\end{tikzpicture}
	\label{fig:microflexible}
	\caption{The submanifolds $X$ and $X'$ for $i=0$, $d=1$, $r=1$ and $k=1$. The goal is to extend $X$ rightwards a little bit. The strategy is to straighten out $X'$ near $0$, and extend $X$ by translation.}
\end{figure}
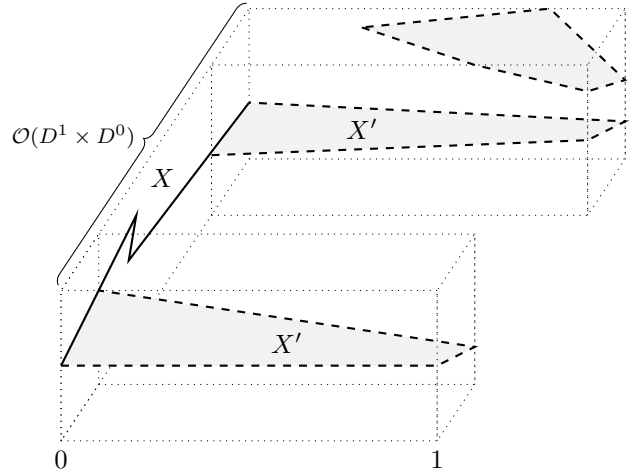

\vspace{.5em}

\noindent \textbf{Claim 1.} \emph{In the case $i=0$, we may assume that $X' = \cO([0,1]) \times X'_0$.}

\begin{proof}Since $X'$ is closed, so is the fibre $X'_0$ over $0 \in \cO([0,1])$ and hence $X'_0 \cap (S^{r-1} \times D^{k-r})$ is compact. Thus there is a finite collection of relative submersion charts for $(\cO([0,1]) \times \cO(S^{r-1} \times D^{k-r}),X') \to \cO([0,1])$ covering $X'_0 \cap (S^{r-1} \times D^{k-r})$. Using the union lemma (Theorem \ref{thm.unionsubmersion} and the remarks following it) with $B = \cO([0,1])$, $b_0  = 0$, and $(E,E') = (\cO([0,1]) \times \cO(S^{r-1} \times D^{k-r}),X)$, we can combine these into a single relative submersion chart. By shrinking its domain we may assume it is a product, and obtain a closed interval $I \subset \cO([0,1])$ containing $\{0\}$ in its interior and an open subset $V \subset \cO(S^{r-1} \times D^{k-r})$ containing $S^{r-1} \times D^{k-r}$, together with an open embedding 
\[
	\varphi \colon I \times V \lra I \times \cO(S^{r-1} \times D^{k-r})
\] 
over $I$ so that $\varphi^{-1}(X') = I \times (V \cap X'_0)$ and $\varphi$ is the inclusion over $\{0\}$.

Pick a compact subset $L$ of $V$ containing $S^{r-1} \times D^{k-r}$ in its interior, e.g.\ an annulus. Then let $K \coloneqq L \cap X'_0$, a compact subset of $X'_0$. We apply the form of the isotopy extension theorem in \cref{thm.isotopyextension} with $k=1$ by identifying $I$ with $[0,1]$, $t$ corresponding to $0$ under this identification, $e = \varphi$, $V = V$, $N = \cO(S^{r-1} \times D^{k-r})$, and $K \subset V$. The conclusion is that there exists a homeomorphism $\Phi \colon I \times \cO(S^{r-1} \times D^{k-r}) \to I \times \cO(S^{r-1} \times D^{k-r})$ over $I$, which is the identity over $0$ and compactly-supported in the $\cO(S^{r-1} \times D^{k-r})$-direction, such that $\Phi$ agrees with $\phi$ near $S^{r-1} \times D^{k-r}$. Since it is compactly-supported in the $\cO(S^{r-1} \times D^{k-r})$-direction, we may extend it by the identity to $I \times \cO(D^r \times D^{k-r})$.

Upon applying $\Phi^{-1}$ to $X'$ and $\Phi_0^{-1}$ to $X$, we may assume that $X'$ contains the product $U \times K$. As $(I \times L) \cap X'$ is compact because $X'$ is closed and $K = X'_0 \cap L$, there is an open neighbourhood $U'$ of $0$ in $I$ and an open neighbourhood $V'$ of $S^{r-1} \times D^{k-r}$ in $\mathrm{int}(L)$ such that $(U' \times V') \cap X' = U' \times (V' \times K)$. Reparametrizing the interval, taking $\cO([0,1]) \coloneqq U'$ and $\cO(S^{r-1},D^{k-r}) = V'$, we have achieved our stated goal.\end{proof}

\vspace{.5em}

Let us now do the case $i > 0$. The main difference is that we now need to take care of the $D^i$'s. This involves an additional open subset of $\bR^i$ containing $D^i$, which we think of not as additional parameters but as a providing a collection of closed subsets to be respected (we comment on this in \cref{rem.why-respect}). The top map is represented by the following data: 
\begin{itemize}
	\item open subsets $\cO(D^i) \subset \bR^i$ and $\cO(D^r \times D^{k-r}) \subset \bR^k$,
	\item  a closed topological submanifold $X \subset \cO(D^i) \times \cO(D^{r} \times D^{k-r})$ such that $\pi \colon (\cO(D^i) \times \cO(D^{r} \times D^{k-r}),X) \to \cO(D^i)$ is a relative topological submersion of relative dimension $(r,d)$, and
	\item a $\xi$-structure $\varphi_X$ on its vertical tangent microbundle $T_\pi X$.
\end{itemize}

The bottom map is represented by similar data: 
\begin{itemize}
	\item open subsets $\cO([0,1]) \subset \bR$, and $\cO(S^{r-1} \times D^{k-r}) \subset \bR^k$,
	\item a closed topological submanifold $X' \subset \cO([0,1])\times \cO(D^i) \times \cO(S^{r-1} \times D^{k-r})$ such that $\pi' \colon (\cO([0,1])\times \cO(D^i) \times \cO(S^{r-1} \times D^{k-r}),X') \to \cO([0,1]) \times \cO(D^i)$ is a relative topological submersion of relative dimension $(r,d)$, and
	\item a $\xi$-structure $\varphi_{X'}$ on its vertical tangent microbundle $T_{\pi'} X'$.
\end{itemize}

We again use the abbreviations $\overline{X} = (X,\varphi_X)$ and $\overline{X}' = (X',\varphi_{X'})$. Without loss of generality $\cO(S^{r-1} \times D^{k-r}) \subset \cO(D^r \times D^{k-r})$, and $\overline{X}'_0 = \overline{X} \cap \cO(D^i) \times \cO(S^{r-1} \times D^{k-r})$. Here we retain the use of $\overline{X}'_0$ for $\overline{X} \cap (\{0\} \times  \cO(D^i) \times \cO(S^{r-1} \times D^{k-r}))$, which coincides with $(\pi')^{-1}(\{0\} \times \cO(D^i))$. We will mimic the strategy in the case $i=0$ and our goal is to show that we may assume that $X' = \cO([0,1]) \times X'_0$. Given this, we can extend by $\cO([0,1]) \times X$ as before, and the $\xi$-structures may be taking care of as before.

\vspace{.5em}

\noindent \textbf{Claim 2.} \emph{In the case $i>0$, we may assume that $X' = \cO([0,1]) \times X'_0$.}

\begin{proof}To do so, consider the composite projection 
\[q \colon (\cO([0,1]) \times \cO(D^i) \times \cO(S^{r-1} \times D^{k-r}),X') \lra \cO([0,1]) \times \cO(D^i) \lra \cO([0,1]).\]
There is a collection $\mathscr{C}$ of closed subsets in $\cO([0,1]) \times \cO(D^i) \times \cO(S^{r-1} \times D^{k-r}$ given by $(\cO([0,1]) \times \{x\} \times \cO(S^{r-1} \times D^{k-r}))$ with $x \in \cO(D^i)$. Relative submersion charts for $X' \to \cO([0,1]) \times \cO(D^i)$ give rise to relative submersion charts for $q$ respecting $\mathscr{C}$, because respecting $\mathscr{C}$ means preserving the projection to $\cO(D^i)$. As before, we combine these using \cref{thm.unionsubmersion} into a single respectful relative submersion chart for $q$ and upon shrinking its domain obtain a closed interval $I \subset \cO([0,1])$ containing $\{0\}$ in its interior and an open subset $V \subset \cO(S^{r-1} \times D^{k-r})$ containing $S^{r-1} \times D^{k-r}$, together with an open embedding
\[
	\varphi \colon I \times (X'_0 \cap \cO(D^i) \times V) \lra X' \cap (\cO([0,\epsilon]) \times \cO(D^i) \times \cO(S^{r-1} \times D^{k-r})
\]
over $I \times \cO(D^i)$ so that $\varphi^{-1}(X') = I \times ((\cO(D^i) \times V) \cap X'_0)$ and $\varphi$ is the inclusion over $\{0\}$.

Pick a compact subset $L \subset V$ containing $S^{r-1} \times D^{k-r}$ in its interior and let $K \coloneqq (\cO(D^i) \times L) \cap X'_0$. Using respectful isotopy extension as in Theorem \ref{thm.isotopyextension} with respect to closed subsets $\cO([0,1]) \times \{x\} \times \cO(S^{r-1} \times D^{k-r})$, we obtain a homeomorphism $\Phi \colon I \times \cO(D^i) \times \cO(S^{r-1} \times D^{k-r}) \to I \times \cO(D^i) \times \cO(S^{r-1} \times D^{k-r})$ over $I \times \cO(D^i)$ which is compactly supported in the $\cO(S^{r-1} \times D^{k-r})$-direction. Applying $\Phi^{-1}$, restricting and renaming, we again have achieved our stated goal.\end{proof}

Having proved both the cases $i=0$ and $i>0$, we are done.
\end{proof}

\begin{remark}\label{rem.why-respect} Let us comment on the reason for considering the $D^i$-parameter as a collection of closed subsets to be respected instead of an additional set of parameters. Firstly, \cref{thm.unionsubmersion} is a statement about charts for a subset of a \emph{single} fibre. We are thus forced to think of the intersection of $X'$ with $\{0\} \times \cO(D^i) \times \cO(S^{r-1} \times D^{k-r})$ as a fibre of the map to $\cO([0,1])$. To keep track of $\cO(D^i)$-parameters we recast them as fibres of the projection to $\cO(D^i)$. Secondly, at the end of the proof, isotopy extension is used. Ordinary versions of this require as input an isotopy of a \emph{fixed} subset, but the fibres of $K$ change while moving over $\cO(D^i)$. Hence we think of $X'$ as obtained by single-parameter isotopy of $K$ respecting the closed subsets by fibres of the projection to $\cO(D^i)$.\end{remark}

A consequence of this proposition and Theorem \ref{thm.hprinciple} is the following:

\begin{proposition}There is a weak equivalence $B\cat{Cob}^{\mr{Top},\xi}(d,n) \simeq \Omega^{n-1} \Psi^{\mr{Top},\xi}_d(\bR^n)$ natural in $\xi$. Furthermore, for all $n \geq 1$ there is a zigzag of commutative diagrams
\[\begin{tikzcd} B\cat{Cob}^{\mr{Top},\xi}(d,n) \dar & \lar[swap]{\simeq} \cdots \dar \rar{\simeq}  & \Omega^{n-1} \Psi^{\mr{Top},\xi}_d(\bR^n) \dar \\[-4pt]
B\cat{Cob}^{\mr{Top},\xi}(d,n+1) & \lar[swap]{\simeq} \cdots \rar{\simeq} & \Omega^{n} \Psi^{\mr{Top},\xi}_d(\bR^{n+1}).\end{tikzcd}\]\end{proposition}

\begin{proof}Apply Theorem \ref{thm.hprinciple} to $M \coloneqq \bR \times D^{n-1}$ with boundary condition $\varnothing$ near $\partial M = \bR \times \partial D^{n-1}$. Then we obtain that 
\[\Psi^{\mr{Top},\xi}(d,M \rel \partial M) \simeq \mr{Map}_\partial(M,\Psi^{\mr{Top},\xi}(d,\bR^n)) \simeq \Omega^{n-1} \Psi^{\mr{Top},\xi}(d,\bR^n)\]
where in the first weak equivalence we used that $\bR \times D^{n-1}$ is parallellizable. The map $\iota \colon \bR^n \hookrightarrow \bR^{n+1}$ induces compatible vertical maps; for the right hand-side we use that we can write down compatible maps $T\bR^n \to \bR^n$ and $T\bR^{n+1} \to \bR^{n+1}$ which are an embedding on each fibre.
\end{proof}

The following corollary is obtained from the previous proposition by letting $n \to \infty$, and is phrased using the spectrum of Definition \ref{def.topspectrum}.

\begin{corollary}\label{cor.topbordism} We have that $B\cat{Cob}^\mr{Top,\xi}(d) \simeq \Omega^{\infty-1}\Psi^\mr{Top,\xi}(d)$.
\end{corollary}

\section{Smoothing spectra of topological manifolds} \label{sec.smoothing} 

Given a $d$-dimensional vector bundle $\upsilon$, which is in particular a topological $\bR^d$-bundle, there is heuristically a map of spectra $\Psi^{\mr{Diff},\upsilon}(d) \to \Psi^{\mr{Top},\upsilon}(d)$ obtained by forgetting the smooth structures. Of course the former is a spectrum of topological spaces and the latter of simplicial sets, so this does not make sense without the following convention:

\begin{convention}\label{conv.smooth-simplices} In this section, whenever we talk about smooth manifolds or spaces derived from them, \emph{we take their smooth singular simplicial sets} as in Definition \ref{def.smoothsimplices} without changing the notation.\end{convention}

Having established this, there \emph{is} a map of spectra of simplicial sets
\[\Psi^{\mr{Diff},\upsilon}(d) \lra \Psi^{\mr{Top},\upsilon}(d),\]
and in this section we prove that this map is nearly always a weak equivalence:

\begin{theorem}\label{thm.smoothspectra} If $d \neq 4$, the map $\Psi^{\mr{Diff},\upsilon}(d) \to \Psi^{\mr{Top},\upsilon}(d)$ is a weak equivalence.
\end{theorem}

\subsection{Smoothing theory with tangential structures} Before proving Theorem \ref{thm.smoothspectra}, we need to understand smoothing theory with tangential structures. Let $\upsilon$ be a $d$-dimensional vector bundle and $M$ be a topological manifold. Fix germs of a smooth structure and of a map of vector bundles $TM \to \upsilon$ near $\partial M$. For each smooth structure $\sigma$ extending the germ of smooth structure there is a map $\mr{Diff}_\partial(M_\sigma) \to \mr{Top}_\partial(M)$. The simplicial set $\mr{Bun}^\mr{Diff}_\partial(TM_\sigma,\upsilon)$ of vector bundle maps admits an action of $\mr{Diff}_\partial(M_\sigma)$, and this may be extended to an $\mr{Top}_\partial(M)$-action on the simplicial set $\mr{Bun}^\mr{Top}_\partial(TM,\upsilon)$ of microbundle maps. Thus there is a map of homotopy quotients
\[\begin{tikzcd}
	B\mr{Diff}^\upsilon_\partial(M_\sigma) = \mr{Bun}^\mr{Diff}_\partial(TM_\sigma,\upsilon)\sslash \mr{Diff}_\partial(M_\sigma) \dar \\[-4pt]
	 B\mr{Top}^\upsilon_\partial(M) = \mr{Bun}^\mr{Top}_\partial(TM,\upsilon)\sslash \mr{Top}_\partial(M).
 \end{tikzcd}\]

The target admits an alternative description, analogous to the smooth setting in Section \ref{sec.asidebundles}. Let us pick a map $\upsilon \to \upsilon^\mr{univ}$ of $d$-dimensional vector bundles, and $\upsilon^\mr{univ} \to \xi^\mr{univ}$  of topological $\bR^d$-bundles, both of which without loss of generality have underlying maps which are Hurewicz fibrations. Let $\cX_\partial^\upsilon(M)$ denote the space of commutative diagrams of maps of topological $\bR^d$-bundles
\[\begin{tikzcd} &[10pt] \upsilon \dar \\[-4pt]
& \upsilon^\mr{univ} \dar \\[-4pt]
TM \arrow{ru} \arrow{ruu} \rar  & \xi^\mr{univ},\end{tikzcd}\]
that are equal to the given germ near $\partial M$. Then there is a $\mr{Top}_\partial(M)$-equivariant map $\mr{Bun}^\mr{Top}_\partial(TM,\upsilon) \to \cX^\upsilon_\partial(M)$, given by composing with the maps $\upsilon \to \upsilon^\mr{univ}$ and $\upsilon^\mr{univ} \to \xi^\mr{univ}$. Because the top map uniquely determines the remaining maps, this is an isomorphism of simplicial sets and hence $B\mr{Top}^\upsilon_\partial(M) \cong \cX_\partial^\upsilon(M) \sslash \mr{Top}_\partial(M)$.

\begin{proposition}\label{prop.smoothingtangential} Let $d \neq 4$, $\upsilon$ be a $d$-dimensional vector bundle, $M$ be a topological manifold, and fix germs of a smooth structure and of a map of vector bundles $TM \to \upsilon$ near $\partial M$. Then the map 
	\[\bigsqcup_{[\sigma]} B\mr{Diff}^\upsilon_\partial(M_\sigma) \lra B\mr{Top}^\upsilon_\partial(M)\]
is a weak equivalence, where $[\sigma]$ ranges over diffeomorphism classes of smooth manifolds homeomorphic to $M$ rel a neighbourhood of $\partial M$.\end{proposition}

\begin{proof}The maps $\mr{Bun}^\mr{Diff}_\partial(TM_\sigma,\upsilon) \to \cX^\upsilon_\partial(M)$ induce a map
\[\bigsqcup_{[\sigma]} \mr{Bun}^\mr{Diff}_\partial(TM_\sigma,\upsilon) \sslash \mr{Diff}_\partial(M_\sigma) \lra \cX_\partial^\upsilon(M) \sslash \mr{Top}_\partial(M).\]
Let $\cX_\partial(M)$ denote the space of commutative diagrams of topological $\bR^d$-bundle maps
	\[\begin{tikzcd} & \upsilon^\mr{univ} \dar \\
	TM \arrow{ru} \rar  & \xi^\mr{univ}.\end{tikzcd}\]
Then we have a map
\[\bigsqcup_{[\sigma]} \mr{Bun}^\mr{Diff}_\partial(TM_\sigma,\upsilon^\mr{univ}) \sslash \mr{Diff}_\partial(M_\sigma) \lra \cX_\partial(M) \sslash \mr{Top}_\partial(M)\]
which fits in a commutative diagram
\begin{equation}\label{eqn.squaresmoothing} \begin{tikzcd}\bigsqcup_{[\sigma]} \mr{Bun}^\mr{Diff}_\partial(TM_\sigma,\upsilon) \sslash \mr{Diff}_\partial(M_\sigma) \dar \rar & \cX_\partial^\upsilon(M) \sslash \mr{Top}_\partial(M) \dar \\[-2pt]
\bigsqcup_{[\sigma]} \mr{Bun}^\mr{Diff}_\partial(TM_\sigma,\upsilon^\mr{univ}) \sslash \mr{Diff}_\partial(M_\sigma) \rar & \cX_\partial(M) \sslash \mr{Top}_\partial(M),\end{tikzcd}\end{equation}
with bottom map a weak equivalence by smoothing theory, Corollary \ref{cor.smoothingx}. To prove that the top map is a weak equivalence, and hence the proposition, it suffices to prove that the diagram is homotopy cartesian.
	
Since $\mr{Diff}_\partial(M_\sigma)$ is a group, \cite[Proposition 1.6]{segalcategories} applies to tell us that
\[\begin{tikzcd} \bigsqcup_{[\sigma]} \mr{Bun}^\mr{Diff}_\partial(TM_\sigma,\upsilon) \rar \dar & \bigsqcup_{[\sigma]} \mr{Bun}^\mr{Diff}_\partial(TM_\sigma,\upsilon) \sslash \mr{Diff}_\partial(M_\sigma) \dar \\[-2pt]
\bigsqcup_{[\sigma]} \mr{Bun}^\mr{Diff}_\partial(TM_\sigma,\upsilon^\mr{univ}) \rar & \bigsqcup_{[\sigma]} \mr{Bun}^\mr{Diff}_\partial(TM_\sigma,\upsilon^\mr{univ}) \sslash \mr{Diff}_\partial(M_\sigma)\end{tikzcd}\]
is homotopy cartesian. Similarly, the following is homotopy cartesian:
\[\begin{tikzcd} \cX_\partial^\upsilon(M) \rar \dar & \cX_\partial^\upsilon(M) \sslash \mr{Top}_\partial(M) \dar \\[-2pt]
\cX_\partial(M) \rar & \cX_\partial(M) \sslash \mr{Top}_\partial(M).\end{tikzcd}\]
Because $\mr{Bun}^\mr{Diff}_\partial(TM_\sigma,\upsilon^\mr{univ})$ is weakly contractible by the proof of Lemma \ref{lem.universaltangential}, the homotopy fibre of the left vertical map of \eqref{eqn.squaresmoothing} over the component $\sigma$ of the left-bottom corner is $\mr{Bun}^\mr{Diff}_\partial(TM_\sigma,\upsilon)$.

On the other hand, we claim that the map $\cX_\partial^\upsilon(M) \to \cX_\partial(M)$ is a fibration. This is a consequence of the underlying map of $\upsilon \to \upsilon^\mr{univ}$ being a fibration and the bundle homotopy covering theorem. Thus the homotopy fibre of the right vertical map of \eqref{eqn.squaresmoothing} over a fixed topological $\bR^d$-bundle map $\ell \colon TM \to \upsilon^\mr{univ}$ is weakly equivalent to the fibre, which is the space of lifts of maps of topological $\bR^d$-bundles
	\[\begin{tikzcd} & \upsilon\dar \\ 
TM \arrow{ru} \rar{\ell}  & \upsilon^\mr{univ}.\end{tikzcd}\]

The map $\ell$ endows $TM$ with a vector bundle structure, which by smoothing theory comes from some smooth structure $\sigma$, unique up to isotopy. The condition that a topological $\bR^d$-bundle map $TM \to \upsilon$ is a vector bundle map with respect to this vector bundle structure is exactly that it covers $\ell$. Thus the natural inclusion of $\mr{Bun}^\mr{Diff}_\partial(TM_\sigma,\upsilon)$ into the homotopy fibre is an weak equivalence. We have thus shown that both vertical homotopy fibres of \eqref{eqn.squaresmoothing} are weakly equivalent to $\mr{Bun}^\mr{Diff}_\partial(TM_\sigma,\upsilon)$, and the map between them is weakly equivalent to the identity.
\end{proof}
	
We shall apply this to spaces of submanifolds of subsets of Euclidean space. In addition to the data fixed above, we fix a compact codimension $0$ smooth submanifold $N \subset \bR^n$ with smooth boundary $\partial N$. We also fix a smooth embedding $\partial M \hookrightarrow \partial N$ and a germ of an extension of this to a smooth embedding of a neighbourhood of $\partial M$ into $N$. 

\begin{proposition}\label{prop.embeddedsmoothing} let $d \neq 4$, the map
\[\begin{tikzcd} \bigsqcup_{[\sigma]} \mr{Bun}^\mr{Diff}_\partial(TM_\sigma,\upsilon) \times_{\mr{Diff}_\partial(M_\sigma)} \mr{Emb}^\mr{Diff}_\partial(M_\sigma,N) \dar \\[-4pt]
 \mr{Bun}^\mr{Top}_\partial(TM,\upsilon) \times_{\mr{Top}_\partial(M)} \mr{Emb}^\mr{Top}_\partial(M,N)\end{tikzcd}\]
is $(2n-3d-3)$-connected.\end{proposition}

\begin{proof}
The map given in this proposition fits in the following commutative diagram
\[\begin{tikzcd}\bigsqcup_{[\sigma]} \mr{Bun}^\mr{Diff}_\partial(TM_\sigma,\upsilon) \times_{\mr{Diff}_\partial(M_\sigma)} \mr{Emb}^\mr{Diff}_\partial(M_\sigma,N) \rar \dar & \dar{\simeq} \bigsqcup_{[\sigma]} B\mr{Diff}^\upsilon_\partial(M_\sigma) \\[-2pt]
 \mr{Bun}^\mr{Top}_\partial(TM,\upsilon) \times_{\mr{Top}_\partial(M)} \mr{Emb}^\mr{Top}_\partial(M,N) \rar & B\mr{Top}^\upsilon_\partial(M),\end{tikzcd}\]
where the right vertical map is a weak equivalence by Proposition \ref{prop.smoothingtangential}. Thus the left map is $(2n-3d-3)$-connected if and only if each map $\mr{Emb}^\mr{Diff}_\partial(M_\sigma,N) \to \mr{Emb}^\mr{Top}_\partial(M,N)$ on horizontal homotopy fibres is. This is the Theorem on the bottom of page 147 of \cite{lashofembeddings}.
\end{proof}

\subsection{The proof of Theorem \ref{thm.smoothspectra}}

We shall prove that the map $\Psi^{\mr{Diff},\upsilon}_d(\bR^n) \to \Psi^{\mr{Top},\upsilon}_d(\bR^n)$ between $n$th levels of the spectra is approximately $2n$-connected. In the first half of the proof, we create smooth ``radial slices'' and ``radial products.'' We introduce some notation for the latter.

\begin{definition}Given an interval $(a,b) \subset (0,\infty)$, $r \in (a,b)$, and submanifold $Y$ of the sphere $S^{n-1}_r \coloneqq \{x \in \bR^n \mid ||x||=r\}$, we let the \emph{radial product} $(a,b) \times^{\mr{rad}} Y$ be the submanifold of $S^{n-1}_{a,b} \coloneqq \{x \in \bR^n \mid ||x|| \in (a,b)\}$ consisting of those points $(s,\theta) \in (0,\infty) \times S^{n-1}$ in radial coordinates such that $s \in (a,b)$ and $r \cdot \theta \in Y$.\end{definition}

We now use these in two definitions, which differ in which parts of the manifolds involved have a smooth structure:

\begin{definition}Let $\Psi^{\mr{Diff},\upsilon}_d(\bR^n)_\bullet$ be the semisimplicial simplicial set with $p$-simplices given by a disjoint union over $(p+1)$-tuples $0 < r_0 < \ldots < r_p$ of positive real numbers of the subsimplicial set\footnote{Recall Convention \ref{conv.smooth-simplices}.} of $\Psi^{\mr{Diff},\upsilon}_d(\bR^n)$ of $\overline{X}$ such that there exists an $\epsilon>0$ so that \begin{enumerate}[(i)]
		\item $X_d \pitchfork S^{n-1}_{r_i}$ for all $d \in \Delta^k$, and
		\item  $\overline{X}_d \cap S^{n-1}_{r_i -\epsilon,r_i+\epsilon} = (r_i -\epsilon,r_i+\epsilon) \times^\mr{rad} (\overline{X}_d \cap S^{n-1}_{r_i})$ for all $d \in \Delta^k$.
	\end{enumerate}\end{definition}

\begin{definition}Let $\Psi^{\mr{Top},\upsilon}_d(\bR^n)_\bullet$ be the semisimplicial simplicial set with $p$-simplices given by a disjoint union over $(p+1)$-tuples $0 < r_0 < \ldots < r_p$ of positive real numbers of the subsimplicial set of $\Psi^{\mr{Top},\upsilon}_d(\bR^n)$ of those $k$-simplices $\overline{X} = (X,\varphi_X)$ such that there exists an $\epsilon>0$ so that 
	\begin{enumerate}[(i)]
		\item $\overline{X} \cap (\Delta^k \times S^{n-1}_{r_i -\epsilon,r_i+\epsilon})$ is a $k$-simplex of $\Psi_d^\mr{Diff,\upsilon}(S^{n-1}_{r_i -\epsilon,r_i+\epsilon})$, 
		\item $X_d \pitchfork S^{n-1}_{r_i}$ for all $d \in \Delta^k$, 
		\item $\overline{X}_d \cap  S^{n-1}_{r_i -\epsilon,r_i+\epsilon} = (r_i-\epsilon,r_i+\epsilon) \times^\mr{rad} (\overline{X}_d \cap S^{n-1}_{r_i})  $ for all $d \in \Delta^k$.
	\end{enumerate}\end{definition}

Note that in the previous two definitions, the radii $r_i$ are given the discrete topology. There is an augmentation $\Psi^{\mr{Top},\upsilon}_d(\bR^n)_\bullet \to \Psi^{\mr{Top},\upsilon}_d(\bR^n)$ given by sending $((r_0,\ldots,r_p),\overline{X})$ to $\overline{X}$. 

\begin{lemma}\label{lem.sphericaltop} The map $||\Psi^{\mr{Top},\upsilon}_d(\bR^n)_\bullet|| \to \Psi^{\mr{Top},\upsilon}_d(\bR^n)$ is a pointed weak equivalence when $n\geq 5$ and $n \geq \max(d+3,2d+1)$.\end{lemma}

\begin{proof}We shall follow the outline of the proof of Proposition \ref{prop.longworestrictionstocob}, with $S^{n-1}_r$ replacing $\{t\} \times (0,1)^{n-1}$. Suppose we are given a commutative diagram
	\[\begin{tikzcd} \partial D^i \dar \rar & {||\Psi^{\mr{Top},\upsilon}_d(\bR^n)_\bullet||} \dar{\epsilon} \\[-2pt]
	D^i \rar & \Psi^{\mr{Top},\upsilon}_d(\bR^n).\end{tikzcd}\]
	Then we must homotope it through commutative diagrams to one where there exists a lift, fixing $\varnothing$. Without loss of generality, such a lift exists over an open neighbourhood $U$ of $\partial D^i$. That is, we have a closed topological submanifold $X \subset D^i \times \bR^n$ such that the map $\pi \colon (D^i \times \bR^n,X) \to D^i$ is a relative topological submersion of relative dimension $(n,d)$, with a $\upsilon$-structure on its vertical tangent microbundle. We additionally have a collection of triples $(U_k,r_k,s_k)$, where the $U_k \subset U$ form an open cover of $U$, and for each $U_k$ we are also given an $r_k \in (0,\infty)$ and a compactly supported continuous weight function $s_k \colon U_k \to [0,1]$ such that for $d \in s_k^{-1}((0,1]) \subset U_k$ the intersection $S^{n-1}_{r_k} \cap X_d$ is smooth and $\overline{X}_d$ is a radial product near $S^{n-1}_{r_k}$. Let $R$ denote the collection of $r_k$'s, which without loss of generality we may assume to be finite.
	
	\vspace{.5em}
	
	\noindent \textbf{Claim.} \emph{We may assume that for all $d \in D^i \setminus U$ there exists an $r \in (0,\infty) \setminus R$, an $\epsilon>0$, and an open neigborhood $V$ of $d$ in $D^i$, such that 
	\begin{enumerate}[(i)] 
		\item $\overline{X} \cap (V \times S^{n-1}_{r-\epsilon,r+\epsilon})$ is a closed smooth submanifold with map to $V$ a smooth submersion of relative dimension $d$, together with a $\upsilon$-structure on its vertical tangent bundle, 
		\item $X_d \pitchfork S^{n-1}_r$ for all $d \in V$, 
		\item $\overline{X}_d \cap S^{n-1}_{r-\epsilon,r+\epsilon}  = (r-\epsilon,r+\epsilon)\times^\mr{rad}(\overline{X}_d \cap S^{n-1}_r)  $, and
		\item $\overline{X}_d \cap S^{n-1}_{r-\epsilon,r+\epsilon} = \overline{X}_{d'} \cap S^{n-1}_{r-\epsilon,r+\epsilon}$ for $d' \in V$.
	\end{enumerate}}
	
	\begin{proof} Pick $(i+2)$ pairwise disjoint intervals $\{[a_j-\eta_j,a_j+\eta_j]\}_{0 \leq j \leq i+1}$ in $(0,\infty) \setminus R$. Pick for each $d \in \mr{int}(D^i)$ an embedded ball $B^{i}(d) \subset \mr{int}(D^i)$ with origin mapping to $d$.
		
		For each $d \in \mr{int}(D^i)$ and $0 \leq j \leq i+1$, we first apply Lemmas \ref{lem.slicedeform} and \ref{lem.slicedeformtangent} to the family $F$ obtained by restricting $\overline{X}$ to $U = B^{i}(d)$, letting $V = B^{i}_{1/2}(d)$, taking $M = S^{n-1}_{a_j}$ and finally letting $W = \smash{S^{n-1}_{(a_j-\eta_j,a_j+\eta_j)}}$. The conclusion is that we can isotope $X$ and change its $\upsilon$-structure with a homotopy supported in $V \times W$ such that there is an open neighbourhood $U_j(d)$ of $d$ over which the end result is transverse to $S^{n-1}_{a_j}$ and radially constant near $S^{n-1}_{a_j}$. 
		
		In contrast to Proposition \ref{prop.longworestrictionstocob}, we are not done yet. Let $\overline{Y}_{d,j}$ denote the transverse intersection $X_d \cap S^{n-1}_{a_j}$, which is a compact topological manifold of dimension $(d-1)$ with $\upsilon$-structure. This extends to a radially constant topological manifold $(a_j-\epsilon_j,a_j+\epsilon_j) \times^\mr{rad} \overline{Y}_{d,j} $ with $\upsilon$-structure. This $\upsilon$-structure endows the tangent microbundle of $(a_j-\epsilon_j,a_j+\epsilon_j) \times^\mr{rad}  Y_{d,j} $ with a vector bundle structure. We claim that we can find a smooth structure $\sigma$ on $Y_{d,j}$ (possibly after modifying $Y_{d,j}$ in the case $d=4$), and a homotopy of $\upsilon$-structures so that $TY_{d,j} \oplus \epsilon \to \upsilon$ is a map of vector bundles. At this point we need to distinguish three cases, working from easiest to hardest:
		\begin{description}
			\item[$d \leq 4$] Every topological manifold of dimension $d' \leq 3$ admits a unique smooth structure, and every topological $\bR^{d'}$-bundle reduces uniquely to a vector bundle up to isomorphism.
			\item[$d \geq 6$] Smoothing theory and the product-structure theorem give rise to  a smooth structure $\sigma$ on $Y_{d,j}$. Furthermore, the $\pi_0$-version of smoothing theory with tangential structure as in Proposition \ref{prop.smoothingtangential}, says that this smooth structure can be chosen so that the map $\varphi_{d,j} \colon (a_j-\epsilon_j,a_j+\epsilon_j))\times^\mr{rad} T(Y_{d,j} )  \to \upsilon$ of topological $\bR^d$-bundles is homotopic to a map of $d$-dimensional vector bundles $T(Y_{d,j})_{\sigma} \oplus \epsilon \to \upsilon$.  
			\item[$d=5$] In this case is not necessarily true that the intersections $X_d \cap S^{n-1}_{a_j}$ admits a smooth structure. However, this is possible when after stabilizing by connected sum with some copies of $S^2 \times S^2$ by the Sum-Stable Smoothing Theorem of \cite[Section 8.6]{freedmanquinn} with corresponding vector bundle structure homotopic to the map to $\upsilon$. Let us pick a chart $\bR^4 \hookrightarrow Y_{d,j}$ in each path-component and isotope $Y_{d,j}$ such that the composite $\bR^4 \hookrightarrow Y_{d,j} \hookrightarrow S^{n-1}_{a_j}$ is smooth; this is possible because $\mr{Emb}^\mr{Diff}(\bR^4,\bR^{n-1}) \hookrightarrow \mr{Emb}^\mr{Top}(\bR^4,\bR^{n-1})$ is a $\pi_0$-isomorphism when $n - 4 \geq 3$, by \cite[Proposition (t/d)]{lashofembeddings}.
			
			We obtain a map $(a_j-\epsilon_j,a_j+\epsilon_j) \times^\mr{rad}  \bR^4 \to (a_j-\epsilon_j,a_j+\epsilon_j)\times^\mr{rad} Y_{d,j} $ such that the radial height function to $(a_j-\epsilon_j,a_j+\epsilon_j)$ has no singularities. We explain how to modify the embedding of $(a_j-\epsilon_j,a_j+\epsilon_j) \times^\mr{rad} \bR^4$ into $S^{n-1}_{(a_j-\epsilon_j,a_j+\epsilon_j)}$ so that in the radial height function we introduce a canceling pair of Morse singularities of index $2$ and $3$; in between the critical values, this changes the level sets by connected sum with $S^2 \times S^2$. It is then easy but tedious to add enough copies of $S^2 \times S^2$ to each path component of $Y_{d,j}$ so as to be able to apply the Sum-Stable Smoothing Theorem.
			
			To do so, we assume that $n \geq 6$, switch to cartesian coordinates, and take $a_j = 0$ and $\epsilon_j = 1$. Fix a family of generalized Morse functions $f_s \colon \bR^4 \times (-1,1) \to (-1,1)$ which (i) is equal to the height function when $x_1^2+\ldots+x_4^2+t^2 \geq 1/2$, (ii) is equal to the height function for $t=0$, and (iii) has two cancelling critical points index $2$ and $3$ for $t=1$. Then we write the height function $(x_1,\ldots,x_{n}) \mapsto x_n$ and take the embedding of $\bR^4 \times (-1,1)$ to be given by $(x_1,\ldots,x_4,t) \mapsto (x_1,\ldots,x_4,0,\ldots,0,t)$. Then we deform this embedding through the family of embeddings
			\[\begin{tikzcd} (x_1,\ldots,x_4,t) \dar[|->] \\[-2pt] \left(x_1,\ldots,x_4,st\eta(x_1^2+\ldots+x_4^2+t^2),0,\ldots,0,f_s(t,x_1,\ldots,x_4)\right)\end{tikzcd}\]
			for $s \in [0,1]$, where $\eta \colon [0,1] \to [0,1]$ is a smooth function that is $1$ on $[0,1/2]$ and $0$ on near $1$.
		\end{description}
		Of course, the $(d-1)$-dimensional smooth manifold $Y_{d,j} \subset S^{n-1}_{a_j}$ is not necessarily smoothly embedded yet. However, the space of smooth embeddings is non-empty when $n-1 \geq 2(d-1)$ (i.e.\ $n \geq 2d+1$) by the smooth Whitney embedding theorem, and by Corollary \ref{thm.catwhitney} the space of topological embeddings is path-connected when $(n-1-2(d-1)-2) \geq 0$, i.e.\ $n \geq 2d-1$, as long as $n \geq 5$ and $n-d \geq 3$ (in fact, this is one of the cases where one can reduce to the smooth case using \cite[Theorem on page 147]{lashofembeddings}). In particular, under the hypothesis of the lemma we can isotope the $(Y_{d,j})_\sigma \subset S^{n-1}_{a_j}$ to a smoothly embedded submanifold.
		
		Adding this isotopy and the homotopy of $\upsilon$-structures to the ones previously obtained, we conclude that we can isotope $X$ and change its $\upsilon$-structure with a homotopy supported in $V \times W$ so that conditions (i)--(iv) are satisfied. At this point we can follow the proof of Claim in the proof of Proposition \ref{prop.longworestrictionstocob}, and we shall spare the reader the details.\end{proof}
	
	\vspace{.1em}
	
		Once the claim has been established, the proof follows that of Proposition \ref{prop.longworestrictionstocob}.	
\end{proof}

It is much easier to prove the smooth version; no smoothing theory or microbundle transversality is required and a version of the proof of \cite[Theorem 3.10]{grwmonoids} gives us:

\begin{lemma}\label{lem.sphericaldiff} The map $||\Psi^{\mr{Diff},\upsilon}_d(\bR^n)_\bullet|| \to \Psi^{\mr{Diff},\upsilon}_d(\bR^n)$ is a pointed weak equivalence.\end{lemma}

There is a map of semi-simplicial spaces $\Psi^{\mr{Diff},\upsilon}_d(\bR^n)_\bullet \to \Psi^{\mr{Top},\upsilon}_d(\bR^n)_\bullet$.

\begin{lemma}\label{lem.sphericalconn} If $d \neq 4$, the pointed map $||\Psi^{\mr{Diff},\upsilon}_d(\bR^n)_\bullet|| \to ||\Psi^{\mr{Top},\upsilon}_d(\bR^n)_\bullet||$ is $(2n-3d-3)$-connected.\end{lemma}

\begin{proof}It suffices to prove that the map is levelwise $(2n-3d-3)$-connected. Fixing $p$ and a collection $0<r_0<\ldots<r_p$, there is a commutative diagram
\[\begin{tikzcd}\Psi^{\mr{Diff},\upsilon}(\bR^n)_p \rar \dar & \Psi^{\mr{Top},\upsilon}(\bR^n)_p \dar \\
\Psi^\mr{Diff,\upsilon}_{d-1}(S^{n-1})^{p+1} \rar[equals] & \Psi^\mr{Diff,\upsilon}_{d-1}(S^{n-1})^{p+1}.\end{tikzcd}\]
Applying the argument in Lemma \ref{lem.sourcetargetfib} in the category of smooth manifolds tells us that the vertical maps are Kan fibrations, so it suffices to prove that the map on fibres has the desired connectivity. But the map on fibres is a disjoint union of maps as in Proposition \ref{prop.embeddedsmoothing}.
\end{proof}

We shall now finish the proof:

\begin{proof}[Proof of Theorem \ref{thm.smoothspectra}] To prove that $\Psi^\mr{Diff,\upsilon}(d) \to \Psi^\mr{Top,\upsilon}(d)$ is a weak equivalence of spectra, it suffices to prove that there exists a $c \in \bZ$ so that for $n$ sufficiently large $\Psi^\mr{Diff,\upsilon}_d(\bR^n) \to \Psi^\mr{Top,\upsilon}_d(\bR^n)$ is a pointed $(2n-c)$-connected map. 

There is a commutative diagram of pointed simplicial sets
\[\begin{tikzcd}  {||\Psi^{\mr{Diff},\upsilon}_d(\bR^n)_\bullet||} \dar{\simeq} \rar & {||\Psi^{\mr{Top},\upsilon}_d(\bR^n)_\bullet||} \dar{\simeq} \\
 \Psi^{\mr{Diff},\upsilon}_d(\bR^n) \rar & \Psi^{\mr{Top},\upsilon}_d(\bR^n)\end{tikzcd}\]
with vertical maps weak equivalences if $d \neq 4$ and $n$ is large enough. Furthermore, Lemma \ref{lem.sphericalconn} proves that the top horizontal map is $(2d-3d-3)$-connected, so we may take $c = 3d+3$.
\end{proof}

\section{The homotopy type of the topological cobordism category} In this section we shall prove Theorem \ref{thm.mainwith}. We start by proving the analogue of Theorem \ref{thm.excisiontangentsmooth}, next discuss the topological Madsen-Tillmann spectrum, and in Corollary \ref{cor.topspectrum} prove that for a topological $\bR^d$-bundle $\xi$ the map $MT^\mr{Top} \xi \to \Psi^{\mr{Top},\xi}(d)$ is a weak equivalence. Using Corollary \ref{cor.topbordism} there is then is a zigzag of weak equivalences
\[B\cat{Cob}^{\mr{Top},\xi}(d) \longleftarrow \cdots \lra \Omega^{\infty-1} MT^\mr{Top}\xi.\]

As in the previous section, our convention is to use simplicial sets of smooth simplices instead of topological spaces when discussing smooth manifolds.

\subsection{Excision in the tangential structure} \label{sec.excision} 

In this section we prove the topological version of excision in the tangential structure. To make sense of the statement, we need the analogues of Lemmas \ref{lem.realisationvectorbundle} and \ref{lem.trivialvectorbundle}. Since their proofs only used facts about general principal $G$-bundles, they also hold for topological $\bR^d$-bundles.

\begin{lemma}If $\xi_\bullet$ is a semisimplicial topological $\bR^d$-bundle which is levelwise numerable, then $||\xi_{\bullet}||$ is also a topological $\bR^d$-bundle. This is numerable if $\xi_{\bullet}$ is levelwise trivial.\end{lemma}

\begin{lemma}\label{lem.tangenttrivialbundle} For every numerable topological $\bR^d$-bundle $\xi$, there is a levelwise trivial semisimplicial topological $\bR^d$-bundle $\xi_\bullet$ with a homotopy equivalence $||\xi_\bullet|| \to \xi$ of numerable topological $\bR^d$-bundles.\end{lemma}

We start with the analogue of Lemma \ref{lem.vectorweq}, proven in an analogous manner: 

\begin{lemma}\label{lem.topologicalrdweq} The construction $\xi \mapsto \Psi^\mr{Top,\xi}(d)$ takes weak equivalences to weak equivalences.
\end{lemma}

\begin{proof}Let $\varphi \colon \xi \to \xi'$ be a weak equivalence of topological $\bR^d$-bundles, with underlying map of base spaces $\phi \colon B \to B'$. It suffices to establish a levelwise pointed weak equivalence $\Psi^{\mr{Top},\xi}_d(\bR^n) \to \Psi^{\mr{Top},\xi'}_d(\bR^n)$; given a commutative diagram
	\[\begin{tikzcd} \partial D^i \rar \dar & \Psi^{\mr{Top},\xi}_d(\bR^n) \dar \\[-2pt]
	D^i \rar & \Psi^{\mr{Top},\xi'}_d(\bR^n)\end{tikzcd}\]
	we need to provide a lift after a homotopy through commutative diagrams fixing $\varnothing$. 	This is represented by a closed topological submanifold $X \subset D^i \times \bR^n$ such that the map $\pi \colon (D^i \times \bR^n,X) \to D^i$ is a relative topological submersion of relative dimension $(n,d)$, with a $\xi'$-structure on its vertical tangent microbundle $T_{\pi}X$. Additionally, the   $\xi'$-structure on $T_{\pi}X$ has a lift to a $\xi$-structure over $\partial D^i$.
	
	By \cite{milnorcw}, the pair $(X,X|_{\partial D^i})$ of topological manifolds has the homotopy type of a pair of CW-complexes. Hence that the map $\phi$ is a weak equivalence implies we can lift the map $\phi_X \colon X \to B'$ underlying the $\xi'$-structure $T_{\pi}X \rightarrow \xi'$ to a map $\varphi_X \colon X \to B$, at least after a homotopy through commutative diagrams which only modifies the maps to $B$ and $B'$. Using the microbundle homotopy covering theorem, Theorem \ref{thm.microbundlecovering}, this can be covered by maps of $d$-dimensional topological microbundles. For details, see the proof of Lemma \ref{lem.vectorweq}.\end{proof}

We now explain the analogue of Theorem \ref{thm.excisiontangentsmooth}, also proven in analogous fashion. We give a map of spectra $||\Psi^{\mr{Top},\xi_{\bullet}}(d)|| \rightarrow \Psi^{\mr{Top},||\xi_{\bullet}||}(d)$ by describing a map on the $n$th levels $||\Psi^{\xi_{\bullet}}_d(\bR^n)|| \to \Psi^{||\xi_{\bullet}||}_d(\bR^n)$ of these spectra. To do so, it suffices to provide maps $\Delta^p \times \Psi^{\xi_p}_d(\bR^n) \to \Psi^{||\xi_\bullet||}_n(\bR^n)$ sending $\varnothing$ to $\varnothing$, which are compatible with the face maps. 

A $k$-simplex of the right hand side is a pair $(\sigma,(X,\varphi_X))$ where $\sigma \colon [k] \rightarrow [p]$ is a non-decreasing function and $X \subset \Delta^k \times \bR^n$ is a closed topological submanifold such that $(\Delta^k \times \bR^n,X) \to \Delta^k$ is a relative topological submersion of relative dimension $(n,d)$, with a $\xi_p$-structure on its vertical tangent microbundle. This is sent to the pair $(X,\overline{\sigma}_X \times \varphi_X)$, where $X$ is as before but the $||\xi_{\bullet}||$-structure on its vertical tangent microbundle is given by the composition $\overline{\sigma}_X\times \varphi_X \colon T_\pi X \to \Delta^p \times \xi_p \to ||\xi_{\bullet}||$. Here $\overline{\sigma}_X$ is the composition of the projection $T_\pi X \rightarrow \Delta^k$ followed by the map $\Delta^k \to \Delta^p$ obtained from $\sigma$ by geometric realisation. This is compatible with the structure maps of the spectra.

\begin{theorem}\label{thm.excisiontangent} If $\xi_\bullet$ is a semisimplicial topological $\bR^d$-bundle which is levelwise numerable, then we have a weak equivalence of spectra 
	\[||\Psi^{\mr{Top},\xi_{\bullet}}(d)|| \overset{\simeq}{\longrightarrow} \Psi^{\mr{Top},||\xi_{\bullet}||}(d).\]\end{theorem}

\begin{proof}Let us define $\Psi^\mr{Top,\xi,const}_d(\bR^n)$ as the subsimplicial set of $\Psi^\mr{Top,\xi}_d(\bR^n)$ with $k$-simplices given by those pairs $(X,\varphi_X)$ of a closed topological submanifold $X \subset \Delta^k \times \bR^n$ such that $\pi \colon (\Delta^k \times \bR^n,X) \to \Delta^k$ is a relative topological submersion of relative dimension $(n,d)$, and $\varphi_X \colon T_\pi X \to \xi$ is a map of topological microbundles such that the underlying map $\phi_X \colon X \to B$ has the property that for each $b \in \Delta^k$ its restriction to $X_b$ factors over a point.
	
The only smooth tool used in the proof of Lemma \ref{lem.tangentconstsmooth} was the vector bundle homotopy covering theorem, so replacing this by the microbundle homotopy covering theorem we obtain that $\Psi^\mr{Top,\xi,const}_d(\bR^n) \hookrightarrow \Psi^\mr{Top,\xi}_d(\bR^n)$ is a weak equivalence when the base of  $\xi$ is a CW complex. As in the proof of Theorem \ref{thm.excisiontangentsmooth}, we may then use Lemma \ref{lem.topologicalrdweq} to assume that each $B_k$ and $||B_\bullet||$ are CW complexes and get a commutative diagram
\[\begin{tikzcd} {||\Psi_d^\mr{Top,\xi_\bullet,const}(d)||} \rar{\simeq} \dar & {||\Psi_d^\mr{Top,\xi_\bullet}(d)||} \dar \\
\Psi_d^\mr{Top,||\xi_{\bullet}||,const}(d) \rar{\simeq} & \Psi_d^\mr{Top,||\xi_{\bullet}||}(d)\end{tikzcd}\]
with horizontal maps weak equivalences. 

Let us compare the two simplicial sets on the left: $k$-simplices in the domain are $(\sigma,(X,\varphi_X))$ as above with $\varphi_X \colon T_\pi X \to \xi_p$ of the desired form (that is, the map underlying the tangential structure is fibrewise constant), while $k$-simplices in target are $(X,\varphi_X)$ with an $\varphi_X \colon T_\pi X \to ||\xi_\bullet||$ of the desired form. These are not isomorphic: in particular, in a commutative diagram
\[\begin{tikzcd} \partial D^i \rar \dar & {||\Psi_d^\mr{Top,\xi_\bullet,const}(d)||} \dar \\[-2pt]
D^i \rar & \Psi_d^\mr{Top,||\xi_{\bullet}||,const}(d) \end{tikzcd}\]
the top map comes with a simplicial triangulation of $\partial D^i$ such that on the map $\phi_X|_{\partial D^i}$ underlying the tangential structure is simplicial in the above sense, while the bottom map does not. However, we can apply a relative simplicial approximation theorem as in \cite{jardineapproximation} to the map underlying the tangential structure, and then use the microbundle homotopy covering theorem, to homotope the bottom map appropriately.\end{proof}

\subsection{The Thom spectrum of a topological $\bR^d$-bundle} \label{sec:thom-spectrum-top}
We discuss a functorial construction $\xi \mapsto MT^\mr{Top}\xi$ associating to a topological $\bR^d$-bundle $\xi$ a Thom spectrum.

To do so we reinterpret the Madsen-Tillmann spectrum $MT^\mr{Diff}\upsilon$ for a $d$-dimensional vector bundle $\upsilon$. As described in Section \ref{sec.smoothscanning}, its $(n+d)$-th level is the Thom space $\mr{Thom}(\pi_\upsilon^* \gamma^\perp_{d,n+d})$, which can also be constructed as the homotopy cofibre of the sphere bundle $\sigma^{\upsilon}_{d,n+d} \colon S\pi_\upsilon^* \gamma^\perp_{d,n+d} \to \mr{Gr}^{\mr{Diff},\upsilon}_{d}(\bR^{n+d})$. This may be explicitly modeled by the mapping cone
\[\left([0,\infty] \times S\pi_\upsilon^* \gamma^\perp_{d,n+d} \sqcup \mr{Gr}^{\mr{Diff},\upsilon}_{d}(\bR^{n+d})\right)/{\sim}\]
where the equivalence relation $\sim$ identifies $(0,x) \in [0,\infty] \times S\pi_\upsilon^* \gamma^\perp_{d,n+d}$ with $\sigma^{\upsilon}_{d,n+d}(x) \in \mr{Gr}^{\mr{Diff},\upsilon}_{d}(\bR^{n+d})$ and $(\infty,x) \in [0,\infty] \times S\pi_\upsilon^* \gamma^\perp_{d,n+d}$ with $(\infty,x') \in [0,\infty] \times S\pi_\upsilon^* \gamma^\perp_{d,n+d}$, which is easily identified with the pushout used before.

The mapping cone construction is homotopy invariant, so we get a homotopy equivalent space if we replace the sphere bundle $\smash{S\pi_\upsilon^* \gamma^\perp_{d,n+d}}$ with the space $\smash{S\mr{Gr}^{\mr{Diff},\upsilon}_{d}(\bR^{n+d})}$ consisting of triples $(x,P,\phi)$ where $x \in \bR^{n+d} \setminus P$. Forgetting $x$ gives a map $S^\upsilon_{d,n+d} \colon S\mr{Gr}^{\mr{Diff},\upsilon}_{d}(\bR^{n+d})  \to \mr{Gr}^{\mr{Diff},\upsilon}_{d}(\bR^{n+d})$ which fits in a commutative diagram
\[\begin{tikzcd}S\pi_\upsilon^* \gamma^\perp_{d,n+d} \arrow{rr}{\simeq} \arrow{rd}[swap]{\sigma^{\upsilon}_{d,n+d}} &[-10pt] &[-10pt] S\mr{Gr}^{\mr{Diff},\upsilon}_{d}(\bR^{n+d}) \arrow{ld}{S^\upsilon_{d,n+d}} \\
& \mr{Gr}^{\mr{Diff},\upsilon}_{d}(\bR^{n+d}). &\end{tikzcd} \]

It is this definition that adapts well to the topological setting. Recall that the topological Grassmannian $\mr{Gr}^\mr{Top}_d(\bR^{n+d})$ of $d$-planes in $\bR^{n+d}$ of Definition \ref{def.topgrassmannian} is given by the simplicial set $\mr{Top}(n+d)/(\mr{Top}(d) \times \mr{Top}(n+d,d))$. 

\begin{definition}We define the \emph{topological Grassmannian $\mr{Gr}^\mr{Top,\xi}_d(\bR^{n+d})$ of $d$-planes in $\bR^{n+d}$ with $\xi$-structure} to be the simplicial set with $k$-simplices given by pairs $(P,\varphi_P)$ of $P \in\mr{Gr}^\mr{Top}_d(\bR^{n+d})_k$ and $\varphi_P \colon T_\pi P \oplus \epsilon^{d'-d} \to \xi$, interpreting $P$ through its image as a closed topological submanifold of $\Delta^k \times \bR^{n+d}$.\end{definition}

\begin{lemma}The construction $\xi \mapsto \mr{Gr}^{\mr{Top},\xi}_d(\bR^{n+d})$ takes weak equivalences to weak equivalences.\end{lemma}

\begin{proof}The hypothesis implies that $\mr{Bun}^\mr{Top}(T\bR^d,\xi)  \to \mr{Bun}^\mr{Top}
(T\bR^d,\xi')$ is a weak equivalence, hence so is its product with $\mr{id}_{\mr{Top}(d)}$. Since $\mr{Top}(n+d\,\mr{pres}\,d) \hookrightarrow \mr{Top}(n+d)$ is a monomorphism of simplicial groups, taking the quotient by $\mr{Top}(n+d\,\mr{pres}\,d)$ preserves weak equivalences.	 
\end{proof}

A $k$-simplex in $\mr{Gr}^\mr{Top,\xi}_d(\bR^{n+d})$ in particular contains the data of a topological submersion $P \to \Delta^k$ of relative dimension $d$, whose fibres are topological $d$-planes in $\bR^{n+d}$ that are the image of the standard plane $\bR^d \subset \bR^{n+d} $ under a homeomorphism of $\bR^{n+d}$. Thus for each $b \in \Delta^k$, the complement $\bR^{n+d} \backslash P_b$ is homeomorphic to $\bR^{n+d} \backslash \bR^d$ and hence homotopy equivalent to an $(n-1)$-dimensional sphere.

\begin{definition}We define $S\mr{Gr}^{\mr{Top},\xi}_d(\bR^{n+d})$ to be the simplicial set with $k$-simplices given by triples $(x,P,\varphi)$ of $(P,\varphi) \in \mr{Gr}^{\mr{Top},\xi}_d(\bR^{n+d})_k$ and a continuous map $x \colon \Delta^k \to \Delta^k \times \bR^{n+d} \backslash P$ over $\Delta^k$. \end{definition}

\begin{remark}Using transitivity of the action of homeomorphisms on points, this is isomorphic to the quotient
\[(\mr{Bun}^\mr{Top}(T\bR^d,\xi) \times \mr{Top}(n+d))/ \mr{Top}(n+d\, \mr{pres}\, d\,\mr{fix}\,\{e_{d+1}\})\]
where $\mr{Top}(n+d\, \mr{pres}\, d\,\mr{fix}\,\{e_{d+1}\}) \subset \mr{Top}(n+d\, \mr{pres}\, d)$ is the subsimplicial group fixing the point $e_{d+1} \in \bR^{n+d}$. \end{remark}

Sending $(x,P,\varphi)$ to $(P,\varphi)$ gives a map of simplicial sets \[S^\xi_{d,n+d} \colon S\mr{Gr}^\mr{Top,\xi}_d(\bR^{n+d}) \lra \mr{Gr}^{\mr{Top},\xi}_d(\bR^{n+d}),\]
which is a Kan fibration with fibres homotopy equivalent to $S^{n-1}$. As a consequence of the previous lemma, the construction $\xi \mapsto S\mr{Gr}^{\mr{Top},\xi}_d(\bR^{n+d})$ also takes weak equivalences to weak equivalences. The $(n+d)$th level of the topological Madsen-Tillmann spectrum will then be given by the mapping cone of the map $S^\xi_{d,n+d}$.

To explain the structure maps, we let $\Sigma_\mr{fib}$ denote the unreduced fibrewise suspension and construct a map $\Sigma_\mr{fib} S\mr{Gr}^{\mr{Top},\xi}_d(\bR^{n+d}) \to S\mr{Gr}^{\mr{Top},\xi}_d(\bR^{1+n+d})$. To do so, let $\iota \colon \bR^{n+d} \hookrightarrow \bR^{1+n+d}$ denote the inclusion on the last $n+d$ coordinates and note that the map
\begin{align*}[-1,1] \times S\mr{Gr}^{\mr{Top},\xi}_d(\bR^{n+d}) &\lra S\mr{Gr}^{\mr{Top},\xi}_d(\bR^{1+n+d}) \\
(t,x,P,\varphi_P) &\longmapsto ((1-|t|)\iota(x)+t e_1,\iota(P),\varphi_P),\end{align*}
sends for fixed $(P,\varphi_P)$ the set $\{(1,x,P,\varphi_P) \mid x \in \bR^{n+d} \backslash P\}$ to a point, and similarly sends the set $\{(-1,x,P,\varphi_P) \mid x \in \bR^{n+d} \backslash P\}$ to a point. To get the structure map, use that the mapping cone of $\Sigma_\mr{fib} S\mr{Gr}^{\mr{Top},\xi}_d(\bR^{n+d}) \to \mr{Gr}^{\mr{Top},\xi}_d(\bR^{n+d})$ is isomorphic to the reduced suspension of the mapping cone of $S^\xi_{d,n+d}$.

\begin{definition}The \emph{topological Madsen-Tillmann spectrum} $MT^\mr{Top}\xi$ has $(n+d)$th level given by the mapping cone of the map $S^\xi_{d,n+d}$, and structure maps as above.\end{definition}

To see that this is the right definition, we explain why the stable spherical fibration described is the inverse to the canonical one. The space $\mr{Gr}^\mr{Top}_d(\bR^{n+d})$ carries a canonical topological $\bR^d$-bundle with fibre over $(P,\varphi)$ given by those $y \in \bR^{n+d}$ that are contained in $P$. By Lemma \ref{lem.topgrassmannian}, this becomes the universal topological $\bR^d$-bundle upon letting $n \to \infty$. The complement of the origin provides an associated spherical fibration, and the stable spherical fibration $\{S\mr{Gr}^\mr{Top}_d(\bR^{n+d})\}_{(n+d) \geq 0}$ is the inverse to this canonical spherical fibration; their fibrewise join is a trivial $-1$-dimensional stable spherical fibration.

The construction $\xi \mapsto MT^\mr{Top}\xi$ is natural in maps of topological $\bR^d$-bundles, and satisfies analogues of Lemmas \ref{lem.smooththomspectraweq} and \ref{lem.smooththomgeomrel}. These are proven in a similar fashion.

\begin{lemma}\label{lem.topthomspectraweq} The construction $\xi \mapsto MT^\mr{Top}\xi$ takes weak equivalences to weak equivalences.\end{lemma}

\begin{proof}It suffices to establish a levelwise pointed weak equivalence. Since the mapping cone construction preserves weak equivalences of the pushout diagrams, it suffices to recall that if $g \colon \xi \to \xi'$ is a weak equivalence then both of the vertical maps in
	\[\begin{tikzcd}S\mr{Gr}^{\mr{Top},\xi}_d(\bR^{n+d}) \rar{S^\xi_{d,n+d}} \dar & \mr{Gr}^{\mr{Top},\xi}_d(\bR^{n+d}) \dar \\
	S\mr{Gr}^{\mr{Top},\xi'}_d(\bR^{n+d}) \rar{S^{\xi'}_{d,n+d}} & \mr{Gr}^{\mr{Top},\xi'}_d(\bR^{n+d})\end{tikzcd}\]
are weak equivalences.
\end{proof}

\begin{lemma}\label{lem.topthomgeomrel} If $\xi_\bullet$ is a levelwise trivial topological $\bR^d$-bundle, then the canonical map $||MT\xi_{\bullet}|| \to MT||\xi_{\bullet}||$ is a weak equivalence.\end{lemma}

\begin{proof}Let us define $\mr{Gr}^{\mr{Top},\xi,\mr{const}}_d(\bR^{n+d})$ to the subspace of $\mr{Gr}^{\mr{Top},\xi}_d(\bR^{n+d})$ where the underlying map $\phi_P \colon P \to B$ of $\varphi_P \colon TP \oplus \epsilon^{d'-d} \to \xi$ factors over a point. As in Lemma \ref{lem.thomconstsmooth}, the inclusion is a weak equivalence; the argument only used that the planes through the origin deformation retract to the origin and the bundle homotopy covering theorem. Now proceed as in the proof of Lemma \ref{lem.smooththomgeomrel}.
\end{proof}

\subsection{Framed planes and the proof of Theorem \ref{thm.mainwith}} Let us now focus on the case $\xi = \epsilon^d \times B$, i.e.\ framings with a map to a background space. Theorem \ref{thm.smoothspectra} tells us that the map $\Psi^{\mr{Diff},\epsilon^d \times B}(d) \to \Psi^{\mr{Top},\epsilon^d \times B}(d)$ is a weak equivalence of spectra.

\begin{proposition}\label{prop.thomcomparisontrivial} For a trivial $d$-dimensional vector bundle $\epsilon^d \times B$, the map $MT^\mr{Diff}(\epsilon^d \times B) \to MT^\mr{Top}(\epsilon^d \times B)$ is a weak equivalence.\end{proposition}

\begin{proof}We have a commutative diagram before taking Thom spaces
	\[\begin{tikzcd}\mr{O}(n+d)/\mr{O}(n) \times B \dar \rar & \mr{Top}(n+d)/\mr{Top}(n+d\,\mr{fix}\,d) \times B \dar \\
	\mr{Gr}^\mr{Diff,\epsilon^d \times B}_d(\bR^{n+d}) \rar &	\mr{Gr}^\mr{Top,\epsilon^d \times B}_d(\bR^{n+d}).\end{tikzcd}\]
The left map is induced by applying $- \times_{\mr{O}(d)} \mr{O}(n+d)/\mr{O}(n)$ to the weak equivalence $\mr{O}(d) \times B \to \mr{Bun}^\mr{Diff}(T\bR^d,\epsilon^d \times B)$ sending $(A,b)$ to the bundle map $T \bR^d \cong \epsilon^d \times \bR^d \to \epsilon^d \times B$ given by $(\vec{x},y) \mapsto (A\vec{x},b)$. The right map is induced by applying $- \times_{\mr{Top}(d)} \mr{Top}(n+d)/\mr{Top}(n+d\,\mr{fix}\,d)$ to a similar weak equivalence $\mr{Top}(d) \times B \to \mr{Bun}^\mr{Top}(T\bR^d,\epsilon^d \times B)$. Since both $- \times_{\mr{O}(d)} \mr{O}(n+d)/\mr{O}(n)$ and $- \times_{\mr{Top}(d)} \mr{Top}(n+d)/\mr{Top}(n+d\,\mr{fix}\,d)$ preserve weak equivalences, we conclude that both vertical maps are weak equivalences.

We explained in the proof of Lemma \ref{lem.topgrassmannian} that the top horizontal map is one between $n$-connected spaces, and thus it is $(n+1)$-connected. Upon taking Thom spaces of spherical fibrations with $(n-1)$-dimensional fibres, this map becomes $2n$-connected. Thus the map of Thom spaces induced by the bottom horizontal map is also $2n$-connected. Since this is the map between the $(n+d)$th levels of $MT^\mr{Diff}(\epsilon^d \times B)$ and $MT^\mr{Top}(\epsilon^d \times B)$, we can conclude that the map  $MT^\mr{Diff}(\epsilon^d \times B) \to MT^\mr{Top}(\epsilon^d \times B)$ is a weak equivalence.
\end{proof}

\begin{remark}In fact, by \cite[Proposition (t/d)]{lashofembeddings}, for $n \geq 5$ and $d \geq 3$ the relative homotopy groups of the map $\mr{O}(n+d)/\mr{O}(n) \to \mr{Top}(n+d)/\mr{Top}(n+d\,\mr{fix}\,d)$ are given by groups $\pi_j(G,O,G_n)$, which vanish for $j \leq 2n-3$ by \cite{milgramknotgroups}. Thus the maps discussed in the previous proof are even more highly-connected.\end{remark}

\begin{corollary}\label{cor.topspectrum} For any topological tangential structure $\xi$, the map $MT^\mr{Top}\xi \to \Psi^{\mr{Top},\xi}(d)$ is a weak equivalence.
\end{corollary}

\begin{proof}By Lemma \ref{lem.tangenttrivialbundle}, we may pick a levelwise trivial topological $\bR^d$-bundle $\xi_{\bullet}$ with homotopy equivalence $||\xi_{\bullet}|| \to \xi$. Then there is a commutative diagram
	\[\begin{tikzcd}
	{||MT^\mr{Top}\xi_{\bullet}}|| \rar \dar[swap]{\simeq} & {||\Psi^{\mr{Top},\xi_{\bullet}}(d)||} \dar{\simeq} \\
	{MT^\mr{Top}||\xi_{\bullet}||} \rar \dar[swap]{\simeq} & \Psi^{\mr{Top},||\xi_{\bullet}||}(d) \dar{\simeq} \\
	MT^\mr{Top}\xi \rar & \Psi^{\mr{Top},\xi}(d),\end{tikzcd}\]
	where the vertical maps are weak equivalences by Lemmas \ref{lem.topologicalrdweq}, \ref{lem.topthomspectraweq}, \ref{lem.topthomgeomrel}, and Theorem \ref{thm.excisiontangent}.
	
	Hence it suffices to prove that each map $MT^\mr{Top} \xi_p \to \Psi^{\mr{Top},\xi_p}(d)$ is a weak equivalence. But since $\xi_p$ is trivial, that is, $\xi_p \cong \epsilon^d \times B_p$, it can be regarded as a (trivial) $d$-dimensional vector bundle, and there is a commutative diagram
	\[\begin{tikzcd}MT^\mr{Diff}(\epsilon^d \times B_p) \rar{\simeq} \dar[swap]{\simeq} & \Psi^{\mr{Diff},\epsilon^d \times B_p}(d) \dar{\simeq} \\ MT^\mr{Top}(\epsilon^d \times B_p) \rar & \Psi^{\mr{Top},\epsilon^d \times B_p}(d).\end{tikzcd}\]
	The top horizontal map is a weak equivalence by Theorem \ref{thm.gmtw}, the left horizontal map is a weak equivalence by Proposition \ref{prop.thomcomparisontrivial}, and the right horizontal map is a weak equivalence by the case $\upsilon = \epsilon^d \times B$ of Theorem \ref{thm.smoothspectra}. Therefore, we may conclude that the bottom map is also a weak equivalence.
\end{proof}

At this point we have concluded the proof of Theorem \ref{thm.mainwith}, and hence of Theorem \ref{thm.mainwithout}.

\subsection{The homotopy type of the topological cobordism category with boundary} \label{sec.boundary} In this section we explain how to extend Theorem \ref{thm.mainwith} to a category of cobordisms of topological manifolds \emph{with boundary}. We shall not give all details, as the proofs involve no new ideas.

The definition of the topological cobordism category with boundary requires first of all a definition of a simplicial set $\Psi^{\mr{Top},\xi}_{d,\partial}(M)$ of $d$-dimensional manifolds with boundary in a manifold $M$ with boundary $\partial M$.

\begin{definition}We define $\Psi^{\mr{Top},\xi}_{d,\partial}(M)$ to be the simplicial set with $k$-simplices given by closed topological manifolds $X \subset \Delta^k \times M$ such that $\pi \colon (\Delta^k \times M,X) \to \Delta^k$ is a relative topological submersion of relative dimension $(md)$ of manifolds with boundary with fibrewise boundary $\partial_{\pi} X \subset \Delta^k \times \partial M$ and a map $\varphi_X \colon T_\pi X \to \xi$ of $d$-dimensional topological microbundles. 
	
We also fix a collar $\partial M \times [0,1) \hookrightarrow M$, and demand that $\overline{X}$ is \emph{neat}, in the sense that there exists an $\epsilon > 0$ such that $\overline{X}$ coincides with $\partial_{\pi} \overline{X} \times [0,\epsilon)$ with respect to these coordinates (here as before, the overline denotes that we also take into account the $\xi$-structure).\end{definition}

Using this, we may define $\psi^\mr{Top,\xi}_{\partial}(d,n,p)$ as the subspace of $\Psi^{\mr{Top},\xi}_{d,\partial}(\bR^{n-1} \times [0,\infty))$ of those $X$ contained in $\bR^{p} \times (0,1)^{n-p-1} \times [0,1)$. The definitions of $\cat{Cob}^{\mr{Top},\xi}_{\partial}(d,n)$ and $\cat{Cob}^{\mr{Top},\xi}_{\partial}(d)$ follow as in Definition \ref{sec.cobcatdefinition}.

Inclusion and boundary restriction give rise to functors
\[\cat{Cob}^{\mr{Top},\xi}(d,n) \overset{i}\lra \cat{Cob}^{\mr{Top},\xi}_\partial(d,n) \overset{\partial}\lra \cat{Cob}^{\mr{Top},\iota^*\xi}(d-1,n-1),\]
where $\iota^* \xi$ is the $(d-1)$-dimensional tangential structure obtained by taking a map $B \to B\mr{Top}(d)$ classifying $\xi$ (this is unique up to homotopy), and taking the homotopy pullback along the inclusion $B\mr{Top}(d-1) \to B\mr{Top}(d)$ (cf.\ Section \ref{sec.asidebundles} in the smooth case). In other words, a $\xi$-structure of a $(d-1)$-dimensional manifold $X$ is a $\xi$-structure on $TX \oplus \epsilon$.

The transversality and microflexibility arguments of Sections \ref{sec:cobordismcat} and \ref{sec.microflexibility} go through for topological manifolds with boundary with straightforward modifications. Using these we obtain a zigzag of sequences of maps
	\begin{equation}\label{eqn:bdyzigzag} \begin{tikzcd}B\cat{Cob}^\mr{Top,\xi}(d,n) \rar{Bi} &[-1pt] B\cat{Cob}^\mr{Top,\xi}_\partial(d,n) \rar{B\partial} &[-3pt]\cat{Cob}^\mr{Top,\iota^* \xi}(d-1,n-1) \\
	\psi^{\mr{Top},\xi}(d,n,1) \rar{i} \uar[swap]{\simeq} \dar{\simeq} & \psi_{\partial}^{\mr{Top},\xi}(d,n,1) \rar{\partial} \uar[swap]{\simeq} \dar{\simeq} & \psi^{\mr{Top},\iota^*\xi}(d-1,n-1,1) \uar[swap]{\simeq} \dar{\simeq} \\
	\Omega^{n-2} \psi^\mr{Top,\xi}(d,n,n-1) \rar{\Omega^{n-2} i} & \Omega^{n-2} \psi_{d,\partial}^\mr{Top,\xi}(d,n,n-1)  \rar{\Omega^{n-2} \partial} & \Omega^{n-2} \Psi^\mr{Top,\iota^* \xi}_{d-1}(\bR^{n-1}),\end{tikzcd}\end{equation}
all of the horizontal maps given by inclusion and boundary restriction respectively. The bottom-left term $\Omega^{n-2} \psi^\mr{Top,\xi}(d,n,n-1)$ may be delooped once more; the $h$-principle proven in Section \ref{sec.microflexibility} says that it is weakly equivalent to $\Omega^{n-1} \Psi^\mr{Top,\xi}_{d}(\bR^n)$ using a scanning map. However, the middle and bottom-right terms can not be further delooped. 

The bottom row of \eqref{eqn:bdyzigzag} is obtained by applying $\Omega^{n-2}$ to the pair of maps
\[\psi^\mr{Top,\xi}(d,n,n-1) \overset{i}\lra \psi_{d,\partial}^\mr{Top,\xi}(d,n,n-1) \overset{\partial}\lra \Psi^\mr{Top,\iota^* \xi}_{d-1}(\bR^{n-1})\]
pointed at $\varnothing$, given by inclusion and restriction to the boundary:

\begin{lemma}The map
	\[\psi_{d,\partial}^\mr{Top,\xi}(d,n,n-1) \overset{\partial}\lra \Psi^\mr{Top,\iota^* \xi}_{d-1}(\bR^{n-1})\]
is a Kan fibration, with inclusion of the fibre over $\varnothing$ given by $i \colon \psi^\mr{Top,\xi}(d,n,n-1) \to \psi_{d,\partial}^\mr{Top,\xi}(d,n,n-1)$.\end{lemma}

\begin{proof}[Sketch of proof] In the lifting problem for a Kan fibration
	\[\begin{tikzcd} \Lambda^k_i \dar \rar & \psi_{d,\partial}^\mr{Top,\xi}(d,n,n-1) \dar \\[-2pt]
	\Delta^k \rar & \Psi^\mr{Top,\iota^* \xi}_{d-1}(\bR^{n-1})\end{tikzcd}\]
	we may replace the pair $(\Delta^k,\Lambda^k_i)$ by the homeomorphic pair $(D^{k-1} \times [0,1],D^{k-1})$. That is, we are given a family $\overline{X}$ of $d$-dimensional topological manifolds with $\xi$-structure and boundary in $\bR^{n-1} \times [0,1)$ over $D^{k-1}$, as well as a family $\overline{Y}$ of $(d-1)$-dimensional topological manifolds with $\iota^* \xi$-structure in $\bR^{n-1}$ over $D^{k-1} \times [0,1]$. By restricting $\overline{Y}$ to $[0,t] \subset [0,1]$ and forgetting that it is fibreed over $[0,t]$, we obtain from the latter a family $\overline{Y}_t$ of $d$-dimensional topological manifolds in $\bR^{n-1} \times [0,t]$ over $D^{k-1}$ with $\xi$-structure, whose intersection with $\bR^{n-1} \times \{0\}$ coincides with intersection of $\overline{X}$ with $\bR^{n-1} \times \{0\}$. We then obtain the lift by taking the fibres over $D^{k-1} \times \{t\}$ to be $\overline{X}$ and $\overline{Y}_t$ glued along these common intersections, rescaled to fit into $\bR^{n-1} \times [0,1)$.	
\end{proof}

The diagram \eqref{eqn:bdyzigzag} is natural in $n$, if on the top two rows we use the inclusion and on the bottom row we use maps analogous to the structure maps of the spectrum $\Psi^\mr{Top,\xi}(d)$ of Definition \ref{def.topspectrum}. In particular we have a spectrum $\Psi_{\partial}^\mr{Top,\xi}(d)$ with $n$th level $\psi_{d,\partial}^\mr{Top,\xi}(d,n+1,n)$ and structure maps of the form $(t,\overline{X}) \mapsto \iota(\overline{X})+t \cdot e_1$ an $(\infty,\overline{X}) \mapsto \varnothing$.
	
Upon letting $n \to \infty$, this is obtained by applying $\Omega^{\infty-1}$ to the maps of spectra
\begin{equation} \label{eqn.spectra} \Omega \Sigma \Psi^\mr{Top,\xi}(d) \lra \Psi_{\partial}^\mr{Top,\xi}(d) \lra \Psi^\mr{Top,\iota^*\xi}(d-1),\end{equation}
where the first map is given on the $n$th level by inclusion
\[( \Omega \Sigma \Psi^\mr{Top,\xi}(d))_n = \Omega \Psi^\mr{Top,\xi}_d(\bR^{n+1}) \simeq \psi^\mr{Top,\xi}(d,n+1,n) \hookrightarrow \psi_{d,\partial}^\mr{Top,\xi}(d,n+1,n),\]
and the second map is given by restriction to $\bR^{n} \times \{0\}$, thus mapping to the $n$th level $\Psi^\mr{Top,\iota^* \xi}_{d-1}(\bR^{n})$ of $\Psi^\mr{Top,\iota^*  \xi}(d)$. Of course, there is a natural weak equivalence $\Omega \Sigma \simeq \mr{id}$ on the category of spectra, so we may replace the left term by $\Psi^\mr{Top,\xi}(d)$.

We can use Theorem \ref{thm.mainwith} to identify the left and right term of \eqref{eqn.spectra}, and we now explain how to obtain the homotopy type of $\Psi_{\partial}^\mr{Top,\xi}(d)$ from this. There is a map $- \times \bR \colon \Sigma^{-1} \Psi^{\mr{Top},\iota^* \xi}(d-1) \to \Psi^\mr{Top,\xi}(d)$ given on $n$th level by sending a $(d-1)$-dimensional affine plane $P$ in $\bR^{n-1}$ to a $d$-dimensional affine plane $P \times \bR$ in $\bR^{n}$. This is weakly equivalent to the map obtained by looping the fibre sequence. To see this, taking the graphs of initial segments of a loop in $\Psi^\mr{Top,\iota^* \xi}_{d-1}(\bR^n)$ based at $\varnothing$ lifts these to $\psi^\mr{Top,\xi}_\partial(d,n+1,n)$ and the full loop lands in $\psi^\mr{Top,\xi}(d,n+1,n)$. This construction fits in a diagram commuting up to based homotopy
\[\begin{tikzcd} \Omega \Psi^\mr{Top,\iota^* \xi}_{d-1}(\bR^n) \rar  \arrow{rd}[swap]{\Omega(- \times \bR)} & \Omega \psi^\mr{Top,\xi}(d,n+1,n) \dar{\simeq}  \\
& \Omega^{2} \Psi^\mr{Top,\xi}_{d}(\bR^{n+1}) \end{tikzcd}\]
with the vertical map a scanning map.

The map $- \times \bR$ restricts to the top horizontal map in the commutative diagram
\[\begin{tikzcd} \Sigma^{-1} MT^\mr{Top}\iota^*\xi \dar{\simeq}\rar{- \times \bR} & MT^\mr{Top}\xi \dar{\simeq} \\ 
 \Sigma^{-1} \Psi^{\mr{Top},\iota^* \xi}(d-1) \rar{- \times \bR} & \Psi^\mr{Top,\xi}(d).\end{tikzcd}\]
 
The following corollary collects the conclusions of the above discussion, a topological analogue of work of Genauer \cite{genauer}. We also remark that the first part can easily be obtained using the techniques in recent work of Steimle \cite{steimle}.

\begin{corollary}If $d \neq 4,5$, there is a fibre sequence
	\[B\cat{Cob}^\mr{Top,\xi}(d) \lra B\cat{Cob}_\partial^\mr{Top,\xi}(d) \lra B\cat{Cob}^{\iota^* \xi}(d-1)\]
obtained by applying $\Omega^{\infty-1}$ to the fibre sequence of spectra obtained by delooping once
\[\Sigma^{-1} MT^\mr{Top}\iota^*\xi \lra MT^\mr{Top}\xi \lra \mr{cofib}\left[\Sigma^{-1} MT^\mr{Top}\iota^*\xi \to MT^\mr{Top}\xi\right].\]
\end{corollary}

\begin{remark}If we work without tangential structure, the cofibre appearing in the previous corollary may not be weakly equivalent to $\Sigma^\infty B\mr{Top}(d)_+$. This would be the naive generalization of the smooth case, cf.\ \cite[Proposition 3.1]{gmtw} and \cite[Theorem 6.4]{genauer}, and we suspect it is false.
\end{remark}

\appendix

\section{Foundational results about topological manifolds} \label{app.results} In this appendix we collect the relevant results on topological manifolds, many of which should be familiar from the theory of smooth or $\mr{PL}$ manifolds. The classical reference for topological manifolds is \cite{kirbysiebenmann} and many technically refined results can be found in \cite{siebenmannstratified}.

\subsection{Embeddings}

\begin{definition}\label{def.loc-flat} Let $M$ and $N$ be topological manifolds of dimensions $m$ and $n$ and with boundaries $\partial M$ and $\partial N$, and let $B$ be a topological space. A \textit{family of locally flat embeddings $M \hookrightarrow N$ over $B$} is a commutative diagram
	\[\begin{tikzcd} B \times M \arrow{rr}{e} \arrow{rd} & & B \times N \arrow{ld} \\
		& B & \end{tikzcd}\]
of continuous maps such that 
	\begin{enumerate}[(i)]
		\item \label{enum.loc-flat:i} $e^{-1}(B \times \partial N) = B \times \partial M$,
		\item \label{enum.loc-flat:ii} $e$ is a homeomorphism onto its image,
		\item \label{enum.loc-flat:iii} for all $(b,m) \in B \times M$ there exists an open subset $U_B \subset B$ around $b$, and charts $\bR^{m-1} \times \bR_{\geq 0} \supset U_M \hookrightarrow M$ around $m$ and $\bR^{n-1} \times \bR_{\geq 0} \supset U_N \hookrightarrow N$ around $\pi_2(e(b,m))$, such that the following diagrams commute
		\[\begin{tikzcd} U_B \times U_M \rar[hook]{\mr{id}_{U_B} \times i} \dar[hook] &[10pt] U_B \times U_N \dar[hook] \\
			B \times M \rar{e} & B \times N,\end{tikzcd}\]
		where $i$ is the restriction of the inclusion $\mr{std} \times \mr{id} \colon \bR^{m-1} \times \bR_{\geq 0} \hookrightarrow \bR^{n-1} \times \bR_{\geq 0}$.
	\end{enumerate}\end{definition}

For $B = \ast$, this is the usual definition of a locally flat embedding. Families over locally flat embedding over $B$ are preserved by postcomposition with locally flat embeddings $N \to N'$ and precomposition with locally flat embeddings $M' \to M$. Moreover, they are preserved by pullback along continuous maps $f \colon B' \to B$ as the data in \eqref{enum.loc-flat:iii} for $(b',m) \in B' \times M$ can be taken $U'_{B'} = f^{-1}(U_B)$, $U'_M = U_M$, and $U'_N = U_N$. Observe restriction to subspaces of $B$ is a special case of pullbacks.

\begin{convention}\label{conv.locallyflat} All (families of) embeddings are locally flat unless mentioned otherwise. Similarly, all submanifolds are locally flat in the sense that the inclusion is a locally flat embedding.\end{convention}

Having established these basic definitions, we can define spaces of embeddings.

\begin{definition}\label{def.embeddings} Let $M,N$ be topological manifolds and fix an embedding $e_\partial \colon \partial M \to \partial N$. Then $\mr{Emb}^\mr{Top}_\partial(M,N)$ is the simplicial set with $k$-simplices given by families of locally flat embeddings
\[\begin{tikzcd} \Delta^k \times M \arrow{rr}{e} \arrow{rd} & & \Delta^k \times N \arrow{ld} \\
 & \Delta^k & \end{tikzcd}\]
over $\Delta^k$ such that $e|_{\Delta^k \times \partial M} = \mr{id}_{\Delta^k} \times e_\partial$. The face and degeneracy maps are given by pullback.\end{definition}

This definition is chosen so that a parametrized isotopy extension theorem is true (this uses \cref{conv.locallyflat}, so in particular $e$ in \cref{def.embeddings} is locally flat); Theorem 6.17 of \cite{siebenmannstratified}. The condition $\mathscr{D}(Q;\{f_t(M)\})$ required in this theorem holds by Example 1.3(4) and the remarks on page 144 ($f_t$ will be $e_t$ in our case).

\begin{theorem}[Isotopy extension] \label{thm.isotopyextensionmfd} Let $M$ and $N$ be topological manifolds and fix an embedding $e_\partial \colon \partial M \to \partial N$, $t \in I^k$, and suppose that $e \colon I^k \to \mr{Emb}^\mr{Top}_\partial(M,N)$ is a family of closed embeddings. If $e$ is constant outside a compact subset of $M$, then there exists a map $f \colon I^k \to \mr{Top}(N)$ such that $f_t = \mr{id}$, and $f \circ e_t = e$.
\end{theorem}

At some point we will need a more refined version; Theorem 6.5 and Complement 6.6 of \cite{siebenmannstratified}. The latter concerns a respectful version: an open embedding $e \colon V \hookrightarrow N$ is said to \emph{thoroughly respect} a collection $\mathscr{C}$ of closed subsets of $N$ if it restricts to an open embedding $e^{-1}(Y \cap e(V)) \to Y$ for all $Y \in \mathscr{C}$. In addition to the condition $\mathscr{D}(V\backslash C)$ required in this theorem ($C$ will be $K$ in our case), which holds as above, this requires a condition $\mathscr{D}(V \backslash C;\mathscr{C}_0)$ holds: once more using Example 1.3(4) and the remarks on page 144, this is in particular the case whenever all sets $Y \cap (V \backslash C)$ for $Y \in \mathscr{C}$ are locally flat submanifolds of $V \backslash C$.

\begin{theorem}[Refined isotopy extension] \label{thm.isotopyextension} Let $V$ and $N$ be topological manifolds, $t \in I^k$, and suppose that $e \colon I^k \to \mr{Emb}^\mr{Top}(V,N)$ is a family of open embeddings. If $K \subset V$ is compact, then there exists a map $f \colon I^k \to \mr{Top}(N)$ such that $f_t = \mr{id}$, $f \circ e_t|_K = e|_K$, and $f$ is constant outside a compact subset of $N$. 
	
Furthermore, if $\mathscr{C}$ is a class of closed subsets of $N$ that are locally flat submanifolds and $e$ thoroughly respects $\mathscr{C}$, then $f$ can also be made to thoroughly respect $\mathscr{C}$.
\end{theorem}

\subsection{Submersions} 

\begin{definition}\label{def.submersion} Let $p \colon E \to B$ be a continuous map. 
\begin{itemize}
	\item A \emph{topological submersion chart for $p$ near $e \in E$} consists of a topological manifold $U$, an open subset $N$ of $p(e)$ in $B$ and a map $f \colon U \times N \to E$ over $N$ that is a homeomorphism onto a neighbourhood of $e$.
	\item The map $p$ is a \emph{topological submersion of relative dimension d} if every $e \in E$ has a submersion chart $f \colon U \times N \to E$ with $U$ a $d$-dimensional topological manifold.
\end{itemize}\end{definition}

Note each fibre of a topological submersion $p$ with relative dimension $d$ is a $d$-dimensional topological manifold, and that we may take $U$ be $\bR^d$. We also need e a relative version.

\begin{definition}\label{def.relsubmersion} Let $p \colon (E,E') \to B$ be a continuous map with domain a pair.
\begin{itemize}
	\item A \emph{relative topological submersion chart for $p$ near $e \in E$} consists of a pair $(U,U')$ of  a topological manifold $U$ and a locally flat submanifold $U' \subset U$, an open subset $N$ of $p(e)$ in $B$, and a map $f \colon (U,U') \times N \to (E,E')$ over $N$ that is a homeomorphism onto a neighbourhood of $e$.
	\item The map $p$ is a \emph{relative topological submersion of relative dimensions $(d,d')$} if every $e \in E$ has a relative submersion chart $f \colon (U,U') \times N \to E$ with $U$ a $d$-dimensional topological manifold and $U' \subset U$ a locally flat $d'$-dimensional submanifold.
\end{itemize} \end{definition}

\begin{example}If $e \colon B \times M \to B \times N$ is a family of locally flat embeddings over $B$ as in \cref{def.loc-flat}, then $(B \times N,e(B \times M)) \to B$ is a relative topological submersion.
\end{example}

The following are Theorem 6.9 and Complement 6.10 of \cite{siebenmannstratified}. If $\mathscr{C}$ is a class of closed subsets of $E$, a submersion chart $f \colon U \times N \to E$ is said to \emph{respect $\mathscr{C}$} if $f^{-1}(Y) = f^{-1}(f(U \times \{p\}) \cap Y) \times N$ for each $Y \in \mathscr{C}$, and $f$ restricted to $f^{-1}(f(U \times \{p\}) \times N$ gives a submersion chart for $Y$. The condition $\mathscr{I}(F;\mathscr{C}_F)$ required in this theorem is explained in 6.7; it holds in particular if $(E,Y) \to B$ is a relative topological submersion for all $Y \in \mathscr{C}$.

\begin{theorem}[Union lemma] \label{thm.unionsubmersion} Let $p \colon E \to B$ be a topological submersion, pick a basepoint $b_0 \in B$ and let $F = p^{-1}(b_0)$ be the fibre over it. Suppose we have compact subsets $A,B \subset F$ of the fibre, with open neighbourhoods $U,V$ and submersion charts $f \colon U \times N \to E$, $g \colon V \times N' \to E$. Then there exists an open neighbourhood $W$ of $A \cup B$ and a submersion chart $h \colon W \times N'' \to E$ which equals $f$ near $A \times N''$ and $g$ near $(B \setminus U) \times N''$.
	
Furthermore, if $\mathscr{C}$ is a class of closed subsets of $E$ that are locally flat submanifolds of some fibre, and $f,g$ respect $\mathscr{C}$, then $h$ can also be made to respect $\mathscr{C}$.
\end{theorem}

Observe that if $(E,E') \to B$ is a relative topological submersion, then taking $\cE = \{E'\}$ we get a union lemma for relative topological submersion charts. Similarly, we obtain an appropriate respectful version in the relative case. When the fibres of $p$ are compact and $p$ is a closed map, we can glue finitely many submersion charts to a submersion chart covering an entire fibre to obtain a topological version of Ehresmann's theorem, \cite[Corollary 6.14]{siebenmannstratified}.

\begin{corollary}[Bundle theorem] \label{cor.unionfibrebundle} If $p \colon E \to B$ is a closed map which is a topological submersion with compact fibres, then it is a topological fibre bundle.\end{corollary}

\subsection{Microbundles} Microbundles play the role of vector bundles for topological manifolds.

\begin{definition}A \emph{$d$-dimensional topological microbundle $\xi = (B,E,i,p)$} consists of spaces $B$, $E$, and continuous maps $i \colon B \to E$, $p \colon E \to B$ with the following properties:
	\begin{itemize}
		\item The composition $p \circ i$  is equal to $\mr{id}_B$.
		\item For all $b \in B$ there exist open neighbourhoods $U \subset B$ of $b$ and $V \subset E$ of $i(b)$, and a homeomorphism $\phi \colon \bR^d \times U \to V$ so that the following diagrams commute
		\[\begin{tikzcd} \bR^d \times U \rar{\phi} & V \\
		U \uar{\{0\} \times \mr{id}} \rar & B \uar[swap]{i}\end{tikzcd} \qquad \qquad \begin{tikzcd}\bR^d \times U \rar{\phi} \dar[swap]{\pi_2} & V \dar{p} \\
		U \rar & B.\end{tikzcd}\]
	\end{itemize}

The space $B$ is called the \emph{base}, the space $E$ the \emph{total space}.
\end{definition}

\begin{example}\label{exam.vertical} If $p \colon X \to P$ is a topological submersion of relative dimension $d$, then $(X,X \times_p X,\Delta,\pi_2)$ is a $d$-dimensional  topological microbundle called the \emph{vertical microtangent bundle}. That this is indeed a $d$-dimensional topological microbundle is a small generalization of Lemma 2.1 of \cite{milnormicrobundles}.\end{example}

\begin{definition}An \emph{isomorphism of topological microbundles} $\varphi \colon \xi \to \xi'$ over the same base $B$ is a homeomorphism $f \colon W \to W'$ between an open neighbourhood $W \subset E$ of $i(B)$ and an open neighbourhood $W' \subset E'$ of $i'(B)$, such that the following diagrams commute
\[\begin{tikzcd}W \rar{f} & W' \\
B \uar{i} \arrow[equals]{r} & B \uar[swap]{i'},\end{tikzcd} \qquad \begin{tikzcd} W \rar{f} \dar[swap]{p} & W' \dar{p'} \\
B \arrow[equals]{r} & B .\end{tikzcd}\]\end{definition}

Microbundles can be pulled back along continuous maps.

\begin{definition}Given a map $f\colon B \to B'$ and a microbundle $\xi' = (B',E',i',p')$ over $B'$, the \emph{pullback microbundle} $f^* \xi'$ over $B$ is given by $E = E' \times_{B'} B$, $i =( i'\circ f) \times_{B'} \mathrm{id}_B$ and $p = \pi_2$, where $\pi_2 \colon E' \times_{B'} B \rightarrow B$ is the projection onto the second component.\end{definition}

If $\iota \colon A \to B$ is the inclusion of a subspace, we denote $\iota^* \xi$ by $\xi|_A$.

\begin{definition}A \emph{map of microbundles} $\phi \colon \xi \to \xi'$ is a pair consisting of a continuous map $f \colon B \to B'$ and an isomorphism $\varphi \colon \xi \to f^* \xi'$.\end{definition}

\begin{definition}Let $\xi = (E,B,i,p)$ be a $d$-dimensional topological microbundle over $B$ and $\xi' = (E',B',i',p')$ a $d$-dimensional microbundle over $B'$. The \emph{simplicial set of microbundle maps} $\mr{Bun}^\mr{Top}(\xi,\xi')$ has $k$-simplices microbundle maps $\Delta^k \times \xi \to \Delta^k \times \xi'$ over $\Delta^k$.\end{definition}

In particular, there can only be maps of microbundles between microbundles of the same dimension. Microbundles in many ways behave like vector bundles. We shall mainly use the microbundle homotopy covering theorem proven in Section 6 of \cite{milnormicrobundles}:

\begin{theorem}[Microbundle homotopy covering] \label{thm.microbundlecovering} If $\phi \colon \xi \to \xi'$ is a map of microbundles with underlying map $f \colon B \to B'$ with $B$ and $B'$ paracompact and $H \colon B \times I \to B'$ is a homotopy starting at $f$, then $H$ may be covered by a map of microbundles $\Phi \colon \xi \times I \to \xi'$.\end{theorem}

\subsection{The Kister-Mazur theorem} \label{sec.kister} Let $\mr{Top}(d)$ be the simplicial group of homeomorphisms of $\bR^d$ fixing the origin (this is equal to the singular simplicial set of the topological group of homeomorphisms of $\bR^d$ in the compact-open topology). Any locally trivial fibre bundle with fibre $\bR^d$ and structure group $\mr{Top}(d)$ gives rise to a $d$-dimensional topological microbundle. The Kister-Mazur theorem tells us that essentially every $d$-dimensional topological microbundle arises this way; in the following statement, a nice space can be an ENR or locally finite simplicial complex \cite{kister}, or a paracompact space \cite{holm}.

\begin{theorem}[Kister-Mazur] \label{thm.kister} Every $d$-dimensional topological microbundle over a nice base $B$ contains a topological $\bR^d$-bundle, unique up to isotopy.\end{theorem}

Kister proved his result by showing that the inclusion $\mr{Top}(d) \to \mr{Emb}^\mr{Top}(\bR^d,\bR^d) $ is a weak equivalence, a statement which we also use.

\begin{remark}\label{rem.germsweq} Note that $\mr{Emb}^\mr{Top}(\bR^d,\bR^d)$ and $\mr{Top}(d)$ have the same image in the simplicial group $\mr{Top}_{(0)}(d)$ of germs near the origin of homeomorphisms of $\bR^d$ fixing the origin. As remarked in \cite[page 410]{rourkesandersontop}, Kister's result implies that $\mr{Top}(d) \to \mr{Top}_{(0)}(d)$ is a weak equivalence. To apply results in the literature, we will need this weak equivalence, as well as the following stronger result: let $\mr{Top}(n+d\,\mr{fix}\,d)$ denote the homeomorphisms of $\bR^{n+d}$ that are the identity on $\bR^d$ and $\mr{Top}_{(0)}(n+d\,\mr{fix}\,d)$ germs of such homeomorphisms near the origin. Then the map $\mr{Top}(n+d\,\mr{fix}\,d) \to \mr{Top}_{(0)}(n+d\,\mr{fix}\,d)$ is a weak equivalence. This is proven for the PL case in \cite[Corollary 2.23(iii)]{millettmetastable} and the same proof works for the topological case; indeed, the input is the proof of Kister's result and the isotopy extension theorem.
\end{remark}

\subsection{Transversality} A commonly-used notion of transversality for topological submanifolds is microbundle transversality. It is imperfect, as it requires the existence of a normal bundle for one of the two submanifolds, but it suffices for most purposes.

\begin{definition}\label{def.microbundletransversality} Suppose $M$ is a submanifold of $N$ with a normal microbundle $\nu$. A submanifold $X$ of $N$ is said to be \emph{microbundle transverse} to $\nu$ if $M \cap X$ is a submanifold of $M$ and there exists a normal microbundle $\nu'$ of $M \cap X$ in $X$ such that $\nu' \to \nu$ is an open embedding on fibres.\end{definition}

The following is proven in \cite{quinntrans} generalizing \cite[Essay III.1]{kirbysiebenmann}:

\begin{theorem}[Microbundle transversality] \label{thm.microbundletransversality} Suppose $M$ and $X$ are proper submanifolds of $N$, $M$ has a normal microbundle $\nu$, $C \subset D \subset N$ are closed subsets, and $X$ is microbundle transverse to $\nu$ near $C$. Then there is an isotopy of $X$ supported in any given neighbourhood of $(D \setminus C) \cap M \cap X$ to a submanifold which is transverse to $\nu$ near $D$.\end{theorem}

\subsection{Smoothing theory} The majority of \cite{kirbysiebenmann} concerns smoothing theory for topological manifolds, which is also discussed in \cite{burghelealashof} and \cite{lashofsmoothing}. We state their results, and provide a modern restatement more convenient for the applications in this paper.
 
\begin{definition}Let $M$ be a topological manifold with $\partial M = \varnothing$. We will take $\mr{Sm}(M)$ to be the simplicial set where the underlying space of $V$ is $\Delta^k \times M$, i.e.\ diagrams
 	\[\begin{tikzcd} \Delta^k \times M \arrow{rr}{\mr{id}} \arrow{rd}[swap]{\pi_1} & & (\Delta^k \times M)_{\tau} \arrow{ld}{\pi_1} \\
 		& \Delta^k & \end{tikzcd}\]
where $\tau$ is a smooth structure on $\Delta^k \times M$ and $\pi_1$ on the right-hand side is a smooth submersion.
\end{definition}
 
The assignment $M \mapsto \mr{Sm}(M)$ is an invariant simplicial sheaf on topological manifolds. By taking extensions of given smooth structures, we can similarly define $\mr{Sm}(M \rel \partial M)$ for $\partial M \neq \varnothing$ with a given smooth structure near $\partial M$. 
 
Kirby and Siebenmann proved that when $d \neq 4$, $\mr{Sm}(-)$ is flexible, i.e.\ restriction maps are Kan fibrations \cite[Essay II]{kirbysiebenmann}. Gromov's $h$-principle machinery then gives a description of the space of smooth structures on $M$ in terms of a space of lifts; this is the content of \cite[Essay V]{kirbysiebenmann}. To state its main result, let us pick a tangent microbundle classifier $\tau_M \colon M \to B\mr{Top}(d)$ and a model of the map $B\mr{O}(d) \to B\mr{Top}(d)$ which is a fibration. Let $\mr{Lift}(M,B\mr{O}(d) \to B\mr{Top}(d) \rel \partial M)$ be the space of lifts of the tangent microbundle classifier
\[\begin{tikzcd} & B\mr{O}(d) \dar \\
\rar[swap]{\tau_M} M \rar \arrow{ru} & B\mr{Top}(d).\end{tikzcd} \]
 
\begin{theorem}[Smoothing theory] \label{thm.smoothingtriangulation} Let $M$ be a topological manifold of dimension $d \neq 4$ with smooth structure near $\partial M$, then $\mr{Sm}(M \rel \partial M) \simeq \mr{Lift}(M,B\mr{O}(d) \to B\mr{Top}(d) \rel \partial M)$.\end{theorem}

More precisely, \cite[Essay IV]{kirbysiebenmann} proves the $\pi_0$-statement and \cite[Essay V]{kirbysiebenmann} the statement for higher homotopy groups. Informally, it says prescribing a smooth structure requires only tangential data. The map is given as follows: every smooth structure on $M$ has a tangent bundle classified by a map $M \to B\mr{O}(d)$, unique up to contractible space of choices, and that the composition $M \to B\mr{O}(d) \to B\mr{Top}(d)$ is homotopic to the given choice $M \to B\mr{Top}(d)$; now use the lifting property of $B\mr{O}(d) \to B\mr{Top}(d)$ to produce from this data a lift of $M \to B\mr{Top}(d)$.

Let us rephrase the result of Theorem \ref{thm.smoothingtriangulation} in terms more appropriate for this paper (see \cite[Section 1.6]{weisswilliams} for a closely related rephrasing of smoothing theory). Recall from \cite[Essay I]{kirbysiebenmann} that a \emph{concordance} between smooth structures $\sigma_0,\sigma_1$ on $M$ is a smooth structure on $M \times I$ restricting to $\sigma_i$ on $M \times \{i\}$, and that an \emph{isotopy} between the same smooth structures is an isotopy $h_t \colon M \to M$ of homeomorphisms starting at the identity such that $h_1^* \sigma_1 = \sigma_0$. Taking $\pi_0$ of the statement of Theorem \ref{thm.smoothingtriangulation}, we obtain that concordance classes of smooth structures are in bijection with vertical homotopy classes of lifts. Isotopic smooth structures are concordant, and the converse is true when $d \neq 4$: for $d \geq 5$ this is \cite[Theorem I.4.1]{kirbysiebenmann}, and in dimension $\leq 3$, every topological manifold admits a unique smooth structure up to isotopy (and hence concordance). The conclusion is that the $\pi_0$-statement of Theorem \ref{thm.smoothingtriangulation} gives a classification of isotopy classes of smooth structures in homotopy-theoretic terms.
 
We next describe the higher homotopy groups. Given a smooth structure $\sigma$, there is a map $\mr{Top}_\partial(M) \to \mr{Sm}(M \rel \partial M)$ which takes the adjoint $f \colon \Delta^k \times M \to M$ of a map $\hat{f} \colon \Delta^k \to \mr{Top}_\partial(M)$, and sends this to the diagram
\[\begin{tikzcd} \Delta^k \times M \arrow{rr}{\mr{id}} \arrow{rd}[swap]{\pi_1} & & (\Delta^k \times M)_{f(\Delta^k \times \sigma)} \arrow{ld}{\pi_1} \\
 & \Delta^k. & \end{tikzcd}\]
 
This construction maps those $\hat{f}$ that arise from a map $\Delta^k \to \mr{Diff}_\partial(M_\sigma)$ with smooth adjoint $f \colon \Delta^k \times M_\sigma \to \Delta^k \times M_\sigma$, to the $k$-fold degeneracy of the $0$-simplex $\sigma$. In fact, the quotient of simplicial groups $\mr{Top}_\partial(M)/\mr{Diff}_\partial(M_\sigma)$ is exactly the union of those path components of $\mr{Sm}(M \rel \partial M)$ that are in the orbit of $\sigma$ under $\mr{Top}_\partial(M)$. These path components are those isotopy classes of smooth structures which are diffeomorphic to $M_\sigma$; if a homeomorphism $f$ sends $\sigma$ to $\sigma'$, it is a diffeomorphism from $M_\sigma$ to $M_{\sigma'}$. Because the map $\mr{Diff}_\partial(M_\sigma) \to \mr{Top}_\partial(M)$ is a monomorphism of simplicial groups, we obtain a fibre sequence
 \[\mr{Diff}_\partial(M_\sigma) \lra \mr{Top}_\partial(M) \lra \mr{Sm}(M \rel \partial M)_{\sigma},\]
 where the subscript indicates we only hit those smooth structures diffeomorphic to $\sigma$. 
 
This finishes the discussion of the left-hand side of Theorem \ref{thm.smoothingtriangulation}, and we continue with the right-hand side. Even though the space of classifying maps for the tangent microbundle is weakly contractible, it is more natural to not make a choice. That is, we get a weakly equivalent but more natural space if we replace $\mr{Lift}(M,B\mr{O}(d) \to B\mr{Top}(d) \rel \partial M)$ by the space $\cX_\partial(M)$ of commutative diagrams of topological $\bR^d$-bundle maps
 \[\begin{tikzcd}  & \upsilon^\mr{univ} \dar \\
 TM \rar \arrow{ru} & \xi^\mr{univ}.\end{tikzcd}\]
Note that this weak equivalence uses the assumption that the map $B\mr{O}(d) \to B\mr{Top}(d)$ is a fibration, as well as the fact that the space of topological $\bR^d$-bundles map $TM \to \xi^\mr{univ}$ is weakly contractible.
 
 \begin{corollary}\label{cor.hofib-diff-top} 
 For $d \neq 4$ the homotopy fibre of the map
 	\[\bigsqcup_{[\sigma]} B\mr{Diff}_\partial(M_\sigma) \lra B \mr{Top}_\partial(M)\]
 	is given by $\cX_\partial(M)$, and the disjoint union is indexed by the diffeomorphism classes of smooth structures on $M$ rel $\partial M$.
 \end{corollary}
 
We can rephrase this as a description of $\bigsqcup_{[\sigma]} B\mr{Diff}_\partial(M_\sigma)$ in terms of a moduli space of topological manifolds with tangential structures. The space $\mr{Bun}^\mr{Diff}_\partial(TM_\sigma,\upsilon^\mr{univ})$ of vector bundle maps $TM_\sigma \to \upsilon^\mr{univ}$ classifying the smooth tangent bundle is weakly contractible. It admits an action of $\mr{Diff}_\partial(M_\sigma)$, and by Lemma \ref{lem.universaltangential} there is a weak equivalence $B\mr{Diff}_\partial(M_\sigma) \simeq \mr{Bun}^\mr{Diff}_\partial(TM_\sigma,\upsilon^\mr{univ}) \sslash \mr{Diff}_\partial(M_\sigma)$. The map $p \colon \upsilon^\mr{univ} \to \xi^\mr{univ}$ gives a map $\mr{Bun}^\mr{Diff}_\partial(TM_\sigma,\upsilon^\mr{univ}) \to \cX_\partial(M)$ by sending $\varphi \colon TM_\sigma \to \upsilon^\mr{univ}$ to the diagram
 \ \[\begin{tikzcd}  & \upsilon^\mr{univ} \dar{p \circ \varphi} \\
 	TM \rar \arrow{ru}{\varphi} & \xi^\mr{univ}.\end{tikzcd}\]
The action of $\mr{Diff}_\partial(M_{\sigma})$ on $\mr{Bun}^\mr{Diff}_\partial(TM_\sigma,\upsilon^\mr{univ})$ extends to an action of the simplicial group $\mr{Top}_\partial(M)$ on $\cX_\partial(M)$. As a consequence, there is a commutative diagram
 \[\begin{tikzcd}\parbox{6cm}{\centering $\bigsqcup_{[\sigma]}\mr{Bun}^\mr{Diff}_\partial(TM_\sigma,\upsilon^\mr{univ}) \sslash \mr{Diff}_\partial(M_\sigma)$ \\ $\simeq\bigsqcup_{[\sigma]} B\mr{Diff}_\partial(M_\sigma)$} \arrow{rd} \arrow{rr} &[-10pt] &[-10pt] \cX_\partial(M) \sslash \mr{Top}_\partial(M) \arrow{ld} \\
 & B\mr{Top}_\partial(M). &\end{tikzcd}\]
By \cref{cor.hofib-diff-tpo}, the horizontal map induces a weak equivalence between vertical homotopy fibres, and hence is also a weak equivalence:
 
\begin{corollary}[Smoothing theory for moduli spaces] \label{cor.smoothingx} For $d \neq 4$ there is a weak equivalence
 	\[\bigsqcup_{[\sigma]} B\mr{Diff}_\partial(M_\sigma) \lra \cX_\partial(M) \sslash \mr{Top}_\partial(M),\]
 	with the disjoint union indexed by the diffeomorphism classes of smooth structures on $M$.
\end{corollary}

\subsection{The product structure theorem} \cref{cor.smoothingx} is hard to use computationally, as the homotopy type of $\mr{Top}(d)$ is not completely understood. However, if one is only interested in $\pi_0$, one can replace $\mr{Top}(d)$ and $\rm{O}(d)$ by $\mr{Top}$ and $\rm{O}$. This is a consequence of the following so-called \emph{product structure theorem} \cite[Essay I.1]{kirbysiebenmann}:
 
\begin{theorem}[Product structure theorem] The map
	\[\begin{tikzcd} \cfrac{\left\{\text{smooth structures on $M$}\right\}}{concordance} \rar{- \times \bR} &  \cfrac{\left\{\text{smooth structures on $M \times \bR$}\right\}}{concordance}\end{tikzcd} \]
is a bijection when $\dim(M) \geq 5$.\end{theorem}

As before, one may replace concordance by isotopy. This statement is equivalent to a stabilization result for the homotopy groups of $\mr{Top}(d)/\rm{O}(d)$, cf.\ \cite[Essay V.5]{kirbysiebenmann}.

\section{Variants of the Whitney embedding theorem} In this appendix we establish analogues of the Whitney embedding theorem for topological manifolds. This is most easily done through piecewise-linear techniques, so we will simultaneously discuss the PL setting; in this appendix $\mr{CAT}$ always denotes $\mr{PL}$ or $\mr{Top}$. Basic references for PL manifolds include \cite{rourkesanderson,hudsonpl}.

\subsection{A weak Whitney embedding theorem} \label{app.whitneyeasy} A closed $d$-dimensional topological or PL manifold can be embedded in some Euclidean space, and relative result exists for compact manifolds with boundary. A single such embedding together with a swindle gives the following well-known result (for details, see e.g.\ \cite[Lemma 2.2]{kupershomeo}):

\begin{lemma}\label{lem.weakwhitney} Let $\mr{CAT} = \mr{Top}$ or $\mr{PL}$. Fix a $\mr{CAT}$-embedding $e_\partial \colon \partial M \hookrightarrow \{0\} \times \bR^\infty$. Then $\mr{Emb}^\mr{CAT}_\partial(M,[0,\infty) \times \bR^\infty)$ is weakly contractible.\end{lemma}

\subsection{A quantitative Whitney embedding theorem} \label{app.whitney} In this section we prove a result that is missing from the literature; the PL and topological manifold analogues of the quantitative Whitney embedding theorem, which says that the space of embeddings of a compact manifold into high-dimensional Euclidean space is highly-connected. 

\begin{theorem}\label{thm.catwhitney} Let $n \geq 5$, $n-m \geq 3$. Let $M$ be a compact $m$-dimensional $\mr{CAT}$-manifold with boundary $\partial M$ (with handle decomposition if $\mr{CAT} = \mr{Top}$ and $m=4$), and $N$ a $k$-connected $n$-dimensional $\mr{CAT}$-manifold with boundary $\partial N$. Fix a $\mr{CAT}$-embedding $\varphi_\partial \colon \partial M \inj \partial N$.  Then $\mr{Emb}^\mr{CAT}_\partial(M,N)$ is $\mr{min}(k-m-1,n-2m-2)$-connected.
\end{theorem}

The smooth case is well-known, following from a transversality argument and openness of smooth embeddings. For $\mr{CAT} = \mr{Top}$ or $\mr{PL}$ we compare to block embeddings instead (note that if $M$ admits a PL structure, one may use \cite[Corollary 2]{lashofembeddings} to deduce the topological case from the PL case, as long as $n-m \geq 3$).

\begin{definition}Let $M$ and $N$ be $\mr{CAT}$ manifolds, possibly with boundary. Fix a $\mr{CAT}$ embedding $\varphi_\partial \colon \partial M \inj \partial N$. The simplicial set of \emph{block $\mr{CAT}$-embeddings}
	\[\widetilde{\mr{Emb}}{}^\mr{CAT}_\partial(M,N)\]
	has $k$-simplices given by the set of $\mr{CAT}$-embeddings $\varphi \colon \Delta^k \times M \to \Delta^k \times N$ which satisfy the following: 
	\begin{itemize}
	\item[$\cdot$] For all faces $\sigma$ of $\Delta^k$ we  have $\varphi^{-1}(\sigma \times N) = \sigma \times M$. 
	\item[$\cdot$] $\varphi|_{\Delta^k \times \partial M} = \mr{id}_{\Delta^k} \times \varphi_\partial$.

\end{itemize}
\end{definition}
 
\vspace{0.3cm}

Note that $\mr{CAT}$-embeddings include into block $\mr{CAT}$-embeddings. 

\begin{theorem}\label{thm.catblockcomparison} Let $n \geq 5$, $n-m \geq 3$ and $\mr{CAT} = \mr{Top}$ or $\mr{PL}$. Let $M$ be a compact $m$-dimensional $\mr{CAT}$-manifold with boundary $\partial M$ (with handle decomposition if $\mr{CAT} = \mr{Top}$ and $m=4$) and $N$ an $n$-dimensional $\mr{CAT}$-manifold with boundary $\partial N$. Fix a $\mr{CAT}$-embedding $\varphi_\partial \colon \partial M \inj \partial N$. Then the map
	\[\mr{Emb}^\mr{CAT}_\partial(M,N) \hookrightarrow \widetilde{\mr{Emb}}{}^\mr{CAT}_\partial(M,N)\]
	is $(n-m-2)$-connected.\end{theorem}

\begin{proof}[Proof outline] For $\mr{CAT} = \mr{PL}$ this is a specialization of \cite[Corollary 4.8]{millett} or \cite[Theorem B and C]{lashofembeddings}. Millett deduces his corollary from \cite[Corollary 4.6]{millett} by an induction over handles, for which the basic case is \cite[Theorem 3.2]{millett}, which has $(M,\partial M) = (D^m,S^{m-1})$. This deduction only uses isotopy extension and the Alexander trick. Thus it suffices to describe how to prove the basic case $(M,\partial M) = (D^m,S^{m-1})$, which also appears as \cite[Theorem 2.8]{burgheleaaut}. This is deduced in Section 2 of \cite{burgheleaaut} from ``the 2nd form of the lemma of disjunction.'' This uses the Alexander trick and ``the 1st form of the lemma of disjunction'' in Section 1 of \cite{burgheleaaut}. 
	
For $\mr{CAT} = \mr{Top}$, we note that all the tools used in the proof of the PL case are available in the topological case. In particular, the crucial input is the 1st form of the lemma of disjunction and this is proven in Pedersen's appendix to \cite{burgheleaaut}.\end{proof}

Block embeddings are easier to study, as elements of the homotopy groups of the space of block embeddings are represented by a \emph{single} embedding, so that no theory of parametrized $\mr{CAT}$-embeddings is required.

\begin{lemma}Under the assumptions of Theorem \ref{thm.catblockcomparison}, if $N$ is $k$-connected then the space of block embeddings $\widetilde{\mr{Emb}}{}^\mr{CAT}_\partial(M,N)$ is $\min(k-m-1,n-2m-2)$-connected.
\end{lemma}

\begin{proof}We first prove that the space $\widetilde{\mr{Emb}}{}^\mr{CAT}_\partial(M,N)$ is non-empty if $-1 \leq \min(k-m-2,n-2m-2)$. Because $-1 \leq k-m-1$ implies $m \leq k$ there exists a $\rm{CAT}$-map $\phi \colon M \to N$ extending $\phi_\partial \colon \partial M \to \partial N$. We make this map into an embedding using general position results for maps between $\mr{PL}$ (or topological) manifolds, saying that under mild conditions maps can be perturbed to have self-intersection locus of the expected dimension. A general position map is an embedding when (a) $2m<n$. For $\rm{CAT} = \rm{PL}$ this is \cite[Theorem 5.4]{rourkesanderson}, and for $\rm{CAT} = \rm{Top}$ this is \cite[Topological General Position Lemma 1]{dancis}. The latter requires $m \leq n-3$, which is a hypothesis, and (b) $3m \leq 2n-1$. In our case, these general position results apply: we have that (a) holds as $-1 \leq n-2m-2$ implies $2m < n$, and (b) holds as $-1 \leq n-2m-2$ and $m \leq n-3$ imply $-2 \leq 2n-3m-2$.

\medskip
	
For higher homotopy groups, we fix a $\mr{CAT}$-embedding $\phi_0 \colon M \hookrightarrow N$ relative to the boundary as a base point. An element of $\pi_i(\widetilde{\mr{Emb}}{}^\mr{CAT}_\partial(M,N))$ is represented by a $\mr{CAT}$-embedding $\phi \colon D^i \times M \to D^i \times N$ that is equal to $\mr{id} \times \phi_0$ on $(\partial D^i \times M) \cup (D^i \times \partial M)$. Let us now take $i \leq \min(k-m-2,n-2m-2)$ and explain how to isotope $\phi$ rel $(\partial D^i \times M) \cup (D^i \times \partial M)$ to $\mr{id} \times \phi_0$.
	
Because $i \leq k-m-1$ we have $i + m + 1 \leq k$, and the connectivity assumptions on $N$ imply that there exists a $\mr{CAT}$-map $\overline{\phi} \colon [0,1] \times D^i \times M \to [0,1] \times N$ that equals $\mr{id}_{[0,1]} \times \phi_0$ on $([0,1] \times \partial D^i \times M) \cup ([0,1] \times D^i \times \partial M) \cup (\{0\} \times D^i \times M)$ and $\phi$ on $\{1\} \times D^i \times M$. We can then apply the same general position results; (a) holds as $i \leq n-2m-2$ implies $2(m+i+1)<n+i+1$, and (b) holds as $i \leq n-2m-2$ and $m \neq n-3$ imply $-2 \neq 2(n+i) -3(m+i)-2$. Thus $\overline{\phi}$ can be made into an embedding $\Phi \colon [0,1] \times D^i \times M \to [0,1] \times N$, which is a concordance from $\phi$ to $\mr{id} \times \phi_0$.\end{proof}

Theorem \ref{thm.catwhitney} follows by combining the previous lemma with Theorem \ref{thm.catblockcomparison}; to get the range use $n-2m-2 \leq n-m-2$.

\bibliographystyle{amsalpha}
\bibliography{refs}

\def\cprime{$'$}
\providecommand{\bysame}{\leavevmode\hbox to3em{\hrulefill}\thinspace}
\providecommand{\MR}{\relax\ifhmode\unskip\space\fi MR }
\providecommand{\MRhref}[2]{%
  \href{http://www.ams.org/mathscinet-getitem?mr=#1}{#2}
}
\providecommand{\href}[2]{#2}
\begin{thebibliography}{GTMW09}

\bibitem[BL74]{burghelealashof}
D.~Burghelea and R.~Lashof, \emph{The homotopy type of the space of
  diffeomorphisms. {I}, {II}}, Trans. Amer. Math. Soc. \textbf{196} (1974),
  1--36; ibid. 196\ (1974), 37--50. \MR{0356103 (50 \#8574)}

\bibitem[BLR75]{burgheleaaut}
D.~Burghelea, R.~Lashof, and M.~Rothenberg, \emph{Groups of automorphisms of
  manifolds}, Lecture Notes in Mathematics, Vol. 473, Springer-Verlag,
  Berlin-New York, 1975, With an appendix (``The topological category'') by E.
  Pedersen. \MR{0380841}

\bibitem[BM14]{bokstedtmadsen}
M.~B\"okstedt and I.~Madsen, \emph{The cobordism category and {W}aldhausen's
  {$K$}-theory}, An alpine expedition through algebraic topology, Contemp.
  Math., vol. 617, Amer. Math. Soc., Providence, RI, 2014, pp.~39--80.
  \MR{3243393}

\bibitem[Dan76]{dancis}
J.~Dancis, \emph{General position maps for topological manifolds in the
  {${2\over 3}$}rds range}, Trans. Amer. Math. Soc. \textbf{216} (1976),
  249--266. \MR{0391098}

\bibitem[Dot14]{dottohprinciple}
E.~Dotto, \emph{A relative {$h$}-principle via cobordism-like categories}, An
  alpine expedition through algebraic topology, Contemp. Math., vol. 617, Amer.
  Math. Soc., Providence, RI, 2014, pp.~133--155. \MR{3243396}

\bibitem[ERW19]{rwebertsemi}
J.~Ebert and O.~Randal-Williams, \emph{Semisimplicial spaces}, Algebr. Geom.
  Topol. \textbf{19} (2019), no.~4, 2099--2150. \MR{3995026}

\bibitem[FP90]{fritschpiccinini}
R.~Fritsch and R.~A. Piccinini, \emph{Cellular structures in topology},
  Cambridge Studies in Advanced Mathematics, vol.~19, Cambridge University
  Press, Cambridge, 1990. \MR{1074175}

\bibitem[FQ90]{freedmanquinn}
M.H. Freedman and F.~Quinn, \emph{Topology of 4-manifolds}, Princeton
  Mathematical Series, vol.~39, Princeton University Press, Princeton, NJ,
  1990. \MR{1201584 (94b:57021)}

\bibitem[Gen12]{genauer}
J.~Genauer, \emph{Cobordism categories of manifolds with corners}, Trans. Amer.
  Math. Soc. \textbf{364} (2012), no.~1, 519--550. \MR{2833590}

\bibitem[GH07]{gepnerhenriques}
D.~Gepner and A.~Henriques, \emph{Homotopy theory of orbispaces}, 2007,
  arXiv:math/0701916.

\bibitem[GL16]{mauriciopaper}
M.~Gomez~Lopez, \emph{The homotopy type of the {PL} cobordism category. {I}},
  arXiv:1608.06236, 2016.

\bibitem[GL22]{mauriciopaper2}
\bysame, \emph{The homotopy type of the {PL} cobordism category. {II}}, $\geq
  2022$.

\bibitem[GLK22]{gomezlopezkupers2}
M.~Gomez~Lopez and A.~Kupers, \emph{Surgery in topological and piecewise-linear
  cobordism categories}, $\geq$ 2022.

\bibitem[Gro86]{gromovhp}
M.~Gromov, \emph{Partial differential relations}, Ergebnisse der Mathematik und
  ihrer Grenzgebiete (3) [Results in Mathematics and Related Areas (3)],
  vol.~9, Springer-Verlag, Berlin, 1986. \MR{864505}

\bibitem[GRW10]{grwmonoids}
S.~Galatius and O.~Randal-Williams, \emph{Monoids of moduli spaces of
  manifolds}, Geom. Topol. \textbf{14} (2010), no.~3, 1243--1302. \MR{2653727}

\bibitem[GTMW09]{gmtw}
S.~Galatius, U.~Tillmann, I.~Madsen, and M.~Weiss, \emph{The homotopy type of
  the cobordism category}, Acta Math. \textbf{202} (2009), no.~2, 195--239.
  \MR{2506750 (2011c:55022)}

\bibitem[Hol66]{holm}
P.~Holm, \emph{Microbundles and bundles}, Bull. Amer. Math. Soc. \textbf{72}
  (1966), 545--548. \MR{0208616}

\bibitem[Hud69]{hudsonpl}
J.~F.~P. Hudson, \emph{Piecewise linear topology}, University of Chicago
  Lecture Notes prepared with the assistance of J. L. Shaneson and J. Lees, W.
  A. Benjamin, Inc., New York-Amsterdam, 1969. \MR{0248844}

\bibitem[HW65]{haefligerwall}
A.~Haefliger and C.~T.~C. Wall, \emph{Piecewise linear bundles in the stable
  range}, Topology \textbf{4} (1965), 209--214. \MR{0184243}

\bibitem[Jar04]{jardineapproximation}
J.~F. Jardine, \emph{Simplicial approximation}, Theory Appl. Categ. \textbf{12}
  (2004), No. 2, 34--72. \MR{2056093}

\bibitem[Kis64]{kister}
J.~M. Kister, \emph{Microbundles are fibre bundles}, Ann. of Math. (2)
  \textbf{80} (1964), 190--199. \MR{0180986 (31 \#5216)}

\bibitem[KS77]{kirbysiebenmann}
R.C. Kirby and L.C. Siebenmann, \emph{Foundational essays on topological
  manifolds, smoothings, and triangulations}, Princeton University Press,
  Princeton, N.J.; University of Tokyo Press, Tokyo, 1977, With notes by John
  Milnor and Michael Atiyah, Annals of Mathematics Studies, No. 88. \MR{0645390
  (58 \#31082)}

\bibitem[Kup15]{kupershomeo}
A.~Kupers, \emph{{Proving homological stability for homeomorphisms of
  high-dimensional manifolds}}, arXiv:1510.02456.

\bibitem[Kup19]{kupershp}
\bysame, \emph{Three applications of delooping to {$h$}-principles}, Geom.
  Dedicata \textbf{202} (2019), 103--151. \MR{4001810}

\bibitem[Las71]{lashofsmoothing}
R.~Lashof, \emph{The immersion approach to triangulation and smoothing},
  Algebraic topology ({P}roc. {S}ympos. {P}ure {M}ath., {V}ol. {XXII}, {U}niv.
  {W}isconsin, {M}adison, {W}is., 1970), Amer. Math. Soc., Providence, R.I.,
  1971, pp.~131--164. \MR{0317332}

\bibitem[Las76]{lashofembeddings}
\bysame, \emph{Embedding spaces}, Illinois J. Math. \textbf{20} (1976), no.~1,
  144--154. \MR{0388403 (52 \#9239)}

\bibitem[LS71]{lashofshaneson}
R.~Lashof and J.~L. Shaneson, \emph{Smoothing {$4$}-manifolds}, Invent. Math.
  \textbf{14} (1971), 197--210. \MR{0295368}

\bibitem[Lur09]{lurietft}
J.~Lurie, \emph{On the classification of topological field theories}, Current
  developments in mathematics, 2008, Int. Press, Somerville, MA, 2009,
  pp.~129--280. \MR{2555928}

\bibitem[May75]{mayclassifying}
J.P. May, \emph{Classifying spaces and fibrations}, Mem. Amer. Math. Soc.
  \textbf{1} (1975), no.~1, 155, xiii+98. \MR{0370579}

\bibitem[Mil59]{milnorcw}
J.~Milnor, \emph{On spaces having the homotopy type of a {${\rm CW}$}-complex},
  Trans. Amer. Math. Soc. \textbf{90} (1959), 272--280. \MR{0100267}

\bibitem[Mil64]{milnormicrobundles}
\bysame, \emph{Microbundles. {I}}, Topology \textbf{3} (1964), no.~suppl. 1,
  53--80. \MR{0161346 (28 \#4553b)}

\bibitem[Mil69]{millettnormal}
K.C. Millett, \emph{Normal structures for locally flat embeddings}, Proc. Amer.
  Math. Soc. \textbf{20} (1969), 580--584. \MR{0246306}

\bibitem[Mil72]{milgramknotgroups}
R.J. Milgram, \emph{On the {H}aefliger knot groups}, Bull. Amer. Math. Soc.
  \textbf{78} (1972), 861--865. \MR{0315728}

\bibitem[Mil75]{millett}
K.C. Millett, \emph{Piecewise linear embeddings of manifolds}, Illinois J.
  Math. \textbf{19} (1975), 354--369. \MR{0385870}

\bibitem[Mil76]{millettmetastable}
\bysame, \emph{Piecewise linear bundles in the metastable range}, Trans. Amer.
  Math. Soc. \textbf{216} (1976), 337--350. \MR{0423361}

\bibitem[MS74]{milnorstasheff}
J.W. Milnor and J.D. Stasheff, \emph{Characteristic classes}, Princeton
  University Press, Princeton, N. J.; University of Tokyo Press, Tokyo, 1974,
  Annals of Mathematics Studies, No. 76. \MR{0440554}

\bibitem[Qui88]{quinntrans}
F.~Quinn, \emph{Topological transversality holds in all dimensions}, Bull.
  Amer. Math. Soc. (N.S.) \textbf{18} (1988), no.~2, 145--148. \MR{929089
  (89c:57016)}

\bibitem[RS70]{rourkesandersontop}
C.P. Rourke and B.J. Sanderson, \emph{On topological neighbourhoods},
  Compositio Math. \textbf{22} (1970), 387--424. \MR{0298671}

\bibitem[RS72]{rourkesanderson}
\bysame, \emph{Introduction to piecewise-linear topology}, Springer-Verlag, New
  York-Heidelberg, 1972, Ergebnisse der Mathematik und ihrer Grenzgebiete, Band
  69. \MR{0350744 (50 \#3236)}

\bibitem[RS16]{robertsstevenson}
D.M. Roberts and D.~Stevenson, \emph{Simplicial principal bundles in
  parametrized spaces}, New York J. Math. \textbf{22} (2016), 405--440.
  \MR{3504418}

\bibitem[RS20]{raptissteimle}
G.~Raptis and W.~Steimle, \emph{Topological manifold bundles and the
  {$A$}-theory assembly map}, Proc. Amer. Math. Soc. \textbf{148} (2020),
  no.~9, 3787--3799. \MR{4127825}

\bibitem[RW11]{rwembedded}
O.~Randal-Williams, \emph{Embedded cobordism categories and spaces of
  submanifolds}, Int. Math. Res. Not. IMRN (2011), no.~3, 572--608.
  \MR{2764873}

\bibitem[Seg68]{Se2}
G.~Segal, \emph{Classifying spaces and spectral sequences}, Inst. Hautes
  \'Etudes Sci. Publ. Math. (1968), no.~34, 105--112. \MR{0232393 (38 \#718)}

\bibitem[Seg74]{segalcategories}
\bysame, \emph{Categories and cohomology theories}, Topology \textbf{13}
  (1974), 293--312. \MR{0353298}

\bibitem[Sie72]{siebenmannstratified}
L.~C. Siebenmann, \emph{Deformation of homeomorphisms on stratified sets. {I},
  {II}}, Comment. Math. Helv. \textbf{47} (1972), 123--136; ibid. 47 (1972),
  137--163. \MR{0319207}

\bibitem[Spa63]{spanierquasi}
E.~Spanier, \emph{Quasi-topologies}, Duke Math. J. \textbf{30} (1963), 1--14.
  \MR{0144300}

\bibitem[Ste21]{steimle}
W.~Steimle, \emph{An additivity theorem for cobordism categories}, Algebr.
  Geom. Topol. \textbf{21} (2021), no.~2, 601--646. \MR{4250512}

\bibitem[WW01]{weisswilliams}
M.~Weiss and B.~Williams, \emph{Automorphisms of manifolds}, Surveys on surgery
  theory, {V}ol. 2, Ann. of Math. Stud., vol. 149, Princeton Univ. Press,
  Princeton, NJ, 2001, pp.~165--220. \MR{1818774}

\end{thebibliography}

\vspace{0.5cm}

\end{document}